\documentclass[english,reqno,11pt]{amsart}
\usepackage{amsmath, amsthm, amsfonts}
\usepackage{hyperref} 
\usepackage{hyperref}
\hypersetup{
	colorlinks=true,
		citecolor=blue!60!black,
		linkcolor=red!60!black,
		urlcolor=green!40!black,
		filecolor=yellow!50!black,
	breaklinks=true,
	pdfpagemode=UseNone,
	bookmarksopen=false,
}
\usepackage{amsmath,amsthm,amssymb}
\usepackage{amscd,indentfirst,epsfig}
\usepackage{latexsym}
\usepackage{times}
\usepackage{enumerate}
\usepackage{mathrsfs}
\usepackage{stmaryrd}
\usepackage{amsopn}
\usepackage{amsmath}
\usepackage{amssymb,dsfont,mathtools}
\usepackage{amsfonts,bm}
\usepackage{amsbsy,amsmath}
\usepackage{amscd}
\usepackage{xcolor}

\usepackage[mono=false]{libertine}
\usepackage[T1]{fontenc}
\usepackage{amsthm}
\usepackage[normalem]{ulem} 
\usepackage[cal=euler, scr=boondoxo]{}
\usepackage{microtype}

\usepackage{numprint}

\begin{document}

\makeatletter
\def \leq {\leqslant}
\def \le {\leq}
\def \geq {\geqslant}
\def\e{\alpha}
\def \ge {\geq}
\def\tend{\rightarrow}
\def\R{\mathbb R}
\def\S{\mathbb S}
\def\Z{\mathbb Z}
\def\N{\mathbb N}
\def\C{\mathscr{C}}
\def\D{\mathcal D}
\def\T{\mathcal T}
\def\E{\mathcal E}
\def\s{\sigma}
\def\k{\kappa}
\def \a{\beta}
 \def\be{\beta}
\def\t{\theta}
\def\l{\ell}
\def \M {\mathcal{M}}
\def\L{\Lambda}
\def\O{\Omega}
\def\g{\gamma}
\def \ds {\displaystyle}
\def\G{\Gamma}
\def \mM {\mathfrak{m}}
\def\f{\varphi}
\def\o{\omega} 
\def \d {\mathrm{d}}
\def \lm {\bm{m}}
\def \lM {\mathds{M}}
\def \lD {\mathds{D}}
\def\U{\Upsilon}
\def\z{\zeta}
\def \Q {\mathcal{Q}}
\def\over{\bm}
\def\b{\backslash}
\def\fet{f_{\ast}}
\def \get{g_{\ast}}
\def\fprim{f^{\prime}}
\def\fprimet{f^{\prime}_\ast }
\def\vet{v_{\ast}}
\def  \pa {\partial}
\def \var {\dd}
\def\vprim{v^{\prime}}
\def\vprimet{v^{\prime}_{\ast}}
\def\grad{\nabla}

\def\Log{\textrm{Log }}
\newtheorem{theo}{Theorem}[section]
\newtheorem{prop}[theo]{Proposition}
\newtheorem{cor}[theo]{Corollary}
\newtheorem{lem}[theo]{Lemma}
\newtheorem{hyp}[theo]{Assumptions}
\newtheorem{defi}[theo]{Definition}
\newtheorem{rmq}[theo]{Remark}
\def \vr {\vartheta}
\def \up {\Upsilon}
\DeclarePairedDelimiter\lnorm{\llbracket}{\rrbracket}		

\def \dd {\bm{\varepsilon}}

\def \ind {\mathbf{1}}
\numberwithin{equation}{section}

\title[ Landau-Fermi-Dirac equation]{About the Landau-Fermi-Dirac equation with moderately soft potentials.}

\author{R. Alonso}
\address{$^1$Texas A\&M University at Qatar, Science Department, Education City, Doha, Qatar.}
\address{$^2$Departamento de Matem\'atica, PUC-Rio, Rio de Janeiro, Brasil.} \email{ricardo.alonso@qatar.tamu.edu}

\author{V. Bagland}
\address{Universit\'{e} Clermont Auvergne, LMBP, UMR 6620 - CNRS,  Campus des C\'ezeaux, 3, place Vasarely, TSA 60026, CS 60026, F-63178 Aubi\`ere Cedex,
France.}\email{Veronique.Bagland@uca.fr}

\author{L. Desvillettes}
\address{Universit\'e de Paris and Sorbonne Universit\'e, CNRS, IMJ-PRG, F-75006 Paris, France} \email{desvillettes@imj-prg.fr}

\author{B.  Lods}
\address{Universit\`{a} degli
Studi di Torino \& Collegio Carlo Alberto, Department of Economics, Social Sciences, Applied Mathematics and Statistics ``ESOMAS'', Corso Unione Sovietica, 218/bis, 10134 Torino, Italy.}\email{bertrand.lods@unito.it}

\maketitle

\begin{abstract}
We present in this document some essential properties of solutions to the homogeneous Landau-Fermi-Dirac equation for moderately soft potentials. Uniform in time estimates for statistical moments, $L^{p}$-norm generation and Sobolev regularity are shown using a combination of techniques that include recent developments concerning level set analysis in the spirit of De Giorgi and refined entropy-entropy dissipation functional inequalities for the Landau collision operator which are extended to the case in question here.  
As a consequence of the analysis, we prove algebraic relaxation of non degenerate distributions towards the Fermi-Dirac statistics under a weak non saturation condition for the initial datum.  All quantitative estimates are uniform with respect to the quantum parameter. They therefore also hold for the classical limit, that is the Landau equation.

\end{abstract}

\tableofcontents

\section{Introduction}

\subsection{Setting of the problem}
In the following pages we study the essential properties of a dilute gas satisfying Pauli's exclusion principle in the Landau's grazing limit regime.  More specifically, we study the Landau-Fermi-Dirac (LFD) equation in the homogeneous setting for moderately soft potential interactions described as
\begin{equation}\label{LFD}
\partial_{t} f(t,v) =  \Q(f)(t,v),\qquad (t,v)\in (0,\infty)\times\mathbb{R}^{3}\,, \qquad f(0)=f_{\mathrm{in}}\,,
\end{equation}
where the collision operator $\Q$ is given by a modification of the Landau operator which includes Pauli's exclusion principle.  It  is defined as 
\begin{equation}\label{eq:Landau}
\Q(f)(v)= {\grad}_v \cdot \int_{\R^3} \Psi(v-\vet) \, \Pi(v-\vet) 
\Big\{ \fet (1- \dd \fet) \grad f - f (1- \dd f) {\grad f}_\ast \Big\}
\, \d\vet\, ,
\end{equation}
with the usual shorthand $f:= f(v)$, $f_* := f(v_*)$, and
\begin{equation*}
\Psi(v-\vet)=|v-\vet|^{\g+2}, \qquad \Pi(z)=\mathrm{Id} - \frac{z \otimes z}{|z|^{2}}.\end{equation*}

\smallskip
\noindent
The Pauli exclusion principle implies that a solution to \eqref{LFD} must \emph{a priori} satisfy the bound 
\begin{equation*}
0\leq f(t,v)\leq \dd^{-1},
\end{equation*}
where the \textit{quantum parameter} 
$$\dd:= \frac{(2\pi\hslash)^{3}}{m^{3}\beta} >0$$ 
depends on the reduced Planck constant $\hslash \approx 1.054\times10^{-34} \mathrm{m}^{2}\mathrm{kg\,s}^{-1}$, the mass $m$ and the statistical weight $\beta$ of the particles species, see \cite[Chapter 17]{chapman}.  In the case of electrons $\dd \approx 1.93\times10^{-10}\ll 1$. The parameter $\dd$ quantifies the quantum effects of the model. The case $\dd=0$ corresponds to the classical Landau equation.

In this paper, we are interested in moderately soft potentials, corresponding to the case when $\gamma\in(-2,0)$. We can already anticipate what are the main original features of this paper:
\begin{itemize}
\item It is the first systematic study of the LFD equation for moderately soft potentials, which are a class of potentials essentially closer to the most relevant case of Coulomb interactions than the recently studied hard potentials case, see \cite{ABL}.
\item Pointwise bounds are obtained thanks to a variant of the De Giorgi method, which leads to an elegant proof in which no high-order derivatives are manipulated. Such $L^{\infty}$-estimates are actually independent of the quantum parameter $\dd$ and yield the following pointwise lower bound
\begin{equation}\label{eq:1-1}
\inf_{v \in \R^{3}}\left(1-\dd f(t,v)\right) \geq \kappa_{0} >0\,, \qquad \forall t \geq 1,\end{equation}
which plays a fundamental role in the long-time behaviour analysis.
\item Stretched exponential decay towards equilibrium is recovered thanks to a careful analysis of the constants pertaining to the moments  bounds and to a complex interpolation procedure involving a nonstandard Gronwall-like lemma; we point out that, for soft potentials, exponential decay is not expected.
\item All estimates are uniform with respect to the quantum parameter (lying in a range fully determined by properties of the initial datum such as 
statistical moments and entropy), so that the statements and proofs also hold for the Landau equation with moderately soft potentials.  This provides a new approach for classical and novel results concerning this equation, in particular related to the long time behaviour.
\end{itemize}
Concerning the smallness of the parameter in the above point, let us make clear to the reader that our approach \emph{does not resort to any kind of perturbation argument}. The smallness of $\dd$ has to be interpreted rather as  a saturation condition since we need to ensure that $\dd$ lies in some physical range for which the above \eqref{eq:1-1} holds. In particular, this restriction on the range of parameters will be needed here only for the results regarding convergence towards equilibrium.

\subsection{Thermal equilibrium.} The relevant steady state of the LFD equation is the so-called Fermi-Dirac statistics.  
  \begin{defi}[Fermi-Dirac statistics]\label{defi:FDstats} Given $\varrho >0, u\in\R^3, \theta >0$ satisfying
    \begin{equation}\label{eq:Intheta}
 5\theta>\left(\frac{3\dd\varrho}{4\pi}\right)^{\frac23},\end{equation}
    we denote by $\M_{\dd}$ the unique Fermi-Dirac statistics (see \cite[Proposition 3]{Lu} for the proof of existence and uniqueness of such a function)
\begin{equation}\label{eq:FDS}
\M_{\dd}(v)=\frac{a_{\dd}\exp(-b_{\dd}|v-u|^{2})}{1+\dd\,a_{\dd}\exp(-b_{\dd}|v-u|^{2})}=: \frac{M_{\dd}}{1+\dd\,M_{\dd}},
\end{equation}
with $a_{\dd},$ $b_{\dd}$ defined in such a way that
$$\int_{\R^{3}}\M_{\dd}(v)\left(
\begin{array}{c}1\\v \\|v-u|^{2}
\end{array}\right) \, \d v
=\left(\begin{array}{c}\varrho \\\varrho\,u \\3\varrho\,\theta\end{array}\right)\,.$$
  Note that  $M_{\dd}$ is here a suitable Maxwellian distribution that allows to recover in the classical limit $\dd \to 0$ the Maxwellian equilibrium. 
\end{defi}
\smallskip

\noindent
Besides the Fermi-Dirac statistics \eqref{eq:FDS}, the distribution
\begin{equation}\label{eq:dege}
F_{\dd}(v)=\begin{cases}
\quad \dd^{-1}  \qquad &\text{ if } \quad |{v-u}|\leq \left(\dfrac{3\varrho\,\dd}{|\S^{2}|} \right)^{\frac{1}{3}}, \\
\quad 0 \qquad &\text{ if } \quad |{v-u}|> \left(\dfrac{3\varrho\,\dd}{|\S^{2}|}\right)^{\frac{1}{3}}\,,\end{cases}
\end{equation}   
can be a stationary state with prescribed mass $\varrho=\int_{\R^{3}}F_{\dd}(v)\d v$ (where  $|\S^{2}|=4\pi$ is the volume of the unit sphere).  Such a degenerate state, referred to as a \emph{saturated Fermi-Dirac} stationary state, can occur for very cold gases (with an explicit condition on the gas temperature).  For such saturated states, the condition
$$\int_{\R^{3}}F_{\dd}(v)\left(
\begin{array}{c}1\\v \\|v-u|^{2}
\end{array}\right) \, \d v
=\left(\begin{array}{c}\varrho \\\varrho\,u \\3\varrho\,\theta\end{array}\right)\,$$
makes the inequality \eqref{eq:Intheta}  an identity which enforces
\begin{equation*}
\dd = \dd_{\text{sat}}:=\frac{4\pi \,(5\,\theta)^{\frac{3}{2}}}{3\varrho}\,.
\end{equation*}
The fact that an initial distribution close to such degenerate state makes $1-\dd f$ arbitrarily small in non negligible sets affects the diffusion mechanism and the regularisation process induced by the parabolic nature of \eqref{LFD}.  As such, the existence of such saturated states impacts the gas relaxation towards the corresponding Fermi-Dirac statistics in a close-to-saturation situation.  It was shown in reference \cite{ABL} that, for hard potentials, explicit exponential relaxation rates exist when $\dd\in(0,c\,\dd_{\text{sat}})$ for some universal $c\in(0,1)$.  One of the central results of this document is the proof of an analogous statement for moderately soft potentials (with algebraic rates).  Proving explicit relaxation rates for $c=1$ remains an open problem for any potential.
\medskip 

\subsection{Notations} For $s \in \R $ and $ p\geq 1$, we define the Lebesgue space $L^{p}_{s}(\R^3)$ through the norm
$$\displaystyle \|f\|_{L^p_{s}} := \left(\int_{\R^3} \big|f(v)\big|^p \, 
\langle v\rangle^s \, \d v\right)^{\frac{1}{p}}, \qquad L^{p}_{s}(\R^3) :=\Big\{f\::\R^{3} \to \R\;;\,\|f\|_{L^{p}_{s}} < \infty\Big\}\, ,$$
where $\langle v\rangle :=\sqrt{1+|v|^{2}}$, $v\in \R^{3}.$ More generally, for any weight function $\varpi\::\:\R^{3} \to \R^{+}$, we define, for any $p \geq 1$,
$$L^{p}(\varpi) :=\Big\{f\::\:\R^{3} \to \R\;;\,\|f\|_{L^{p}(\varpi)}^{p}:=\int_{\R^{3}}\big|f\big|^{p}\,\varpi\,\d v  < \infty\Big\}\,.$$
With this notation, one can write for example $L^{p}_{s}(\R^{3})=L^{p}\big(\langle \cdot \rangle^{s}\big)$, for $p \geq 1,\,s \geq 0$.
\smallskip
\noindent
We define the weighted Sobolev spaces by 
$$W^{k,p}_{s}(\R^3) :=\Big\{f \in L^{p}_{s}(\R^3)\;;\;\partial_{v}^{\beta}f \in L^{p}_{s}(\R^3) \:\;\forall\, |\beta| \leq k\Big\}\,,\quad\text{with}\quad k \in \N\,,$$
with the standard norm
$$ \|f\|_{W^{k,p}_{s}} := \bigg( \sum_{0\leq |\beta| \leq k} 
\int_{\R^3} \big| \partial^{\beta}_v f(v)\big|^p\, 
\langle v\rangle^s \, \d v\bigg)^{\frac{1}{p}},$$
where  $\a=(i_1,i_2,i_3)\in \N^3$, $|\a|=i_1+i_2+i_3$ and 
$\partial^{\beta}_v f =\partial_1^{i_1}\partial_2^{i_2}\partial_3^{i_3} f$.
For $p=2$, we will simply write $H^{k}_{s}(\R^{3})=W^{k,2}_{s}(\R^{3})$, $k \in \N$, $s \geq 0$. An additional important shorthand that will be used when specifically referring to moments and weighted $L^{2}$-norm of \textit{solutions} is defined in the following:
\begin{defi}\label{defi:mom}
Given a nonnegative measurable mapping $g\::\:\R^{3}\to {\R^+}$, we introduce for any $s \in \R$,
$$\lm_{s}(g) :=\int_{\R^{3}}g(v)\langle v\rangle^{s}\d v, \qquad \lM_{s}(g) :=\int_{\R^{3}}g^{2}(v)\langle v\rangle^{s}\d v , $$
and
$$\bm{E}_{s}(g) :=\lm_{s}(g)+\frac{1}{2}\lM_{s}(g), \qquad \lD_{s}(g) :=\int_{\R^{3}}\left|\nabla \left(\langle v\rangle^{\frac{s}{2}}g(v)\right)\right|^{2}\d v.$$
Moreover, if $f=f(t,v)$ is a (weak) solution to \eqref{LFD}, we simply write
$$\lm_{s}(t) :=\lm_{s}(f(t)), \qquad \lM_{s}(t) :=\lM_{s}(f(t)), \qquad \bm{E}_{s}(t) :=\lm_{s}(t)+\frac{1}{2}\lM_{s}(t) , $$
and  $\lD_{s}(t) :=\lD_{s}(f(t)).$
\end{defi}
\subsection{Weak solutions for the moderately soft potential case $\g\in(-2,0)$}
In the sequel we perform the calculations in the following functional framework:
\begin{defi} \label{defi:fin}
Fix $\dd_{0}>0$ and a nonnegative $f_{\mathrm{in}}\in L^{1}_{2}(\R^{3})$ satisfying 
\begin{equation}\label{hypci}
0<\|f_{\mathrm{in}}\|_{L^{\infty}} =:\dd_{0}^{-1} <\infty\qquad \text{ and }  \qquad \mathcal{S}_{\dd_{0}}(f_{\mathrm{in}})>0, \qquad |H(f_{\mathrm{in}})| < \infty\,,
\end{equation}
where $\mathcal{S}_{\dd_{0}}(f_{\rm in})$ denotes the Landau-Fermi-Dirac entropy while $H(f_{\rm in})$ is the classical Boltzmann entropy (see Section \ref{sec:entrop} for precise definition). 

\medskip
\noindent
For any $\dd \in [0,\dd_{0}]$, we say that $f \in \mathcal{Y}_{\dd}(f_{\mathrm{in}})$ if $f\in L^{1}_{2}(\R^{3})$ satisfies $0\leq f\leq \dd^{-1}$ and 
\begin{equation}\label{eq0}
\int_{\R^{3}}f(v)\left(\begin{array}{c}1 \\v \\ |v|^{2}\end{array}\right)\d v=\int_{\R^{3}}f_{\mathrm{in}}(v)\left(\begin{array}{c}1 \\v \\ |v|^{2}\end{array}\right)\d v=\left(\begin{array}{c}\varrho  \\ \varrho u \\ 3\varrho \theta +\varrho|u|^{2}\end{array}\right),\end{equation}
and $\mathcal{S}_{\dd}(f) \geq \mathcal{S}_{\dd}(f_{\mathrm{in}}).$
\end{defi}
\noindent
By a simple scaling argument, there is no loss in generality in assuming 
\begin{equation}\label{eq:Mass}
\varrho =\theta =1, \qquad u=0.\end{equation}
This assumption will be made  \emph{throughout the manuscript} and $\M_{\dd}$ will always denote the Fermi-Dirac statistics corresponding to this normalisation.

{It is important to clarify the role of the class $\mathcal{Y}_{\dd}(f_{\rm in})$ in the sequel of the paper as well as that of $\dd_{0}$. In \emph{all the subsequent results}, the parameter $\dd_{0} >0$ is \emph{fixed} and $f_{\rm in}$ satisfying \eqref{hypci} is chosen. Then, in several results, we will consider a smaller threshold parameter, say $\dd_{\star} \in (0,\dd_{0}]$, and solutions $f=f(t,v)$ to \eqref{LFD} for all $\dd \in (0,\dd_{\star}]$. Such solutions will belong to the class $\mathcal{Y}_{\dd}(f_{\rm in})$ and properties of such solutions  as well as various bounds for them will be derived \emph{uniformly with respect to} $\dd \in (0,\dd_{\star}].$}

\smallskip
\noindent 

We adopt the notations of \cite{ABL}, namely,
\begin{equation*}
\left\{
\begin{array}{rcl}
a(z) & = & \left(a_{i,j}(z)\right)_{i,j} \quad \mbox{ with }
\quad a_{i,j}(z) 
= |z|^{\g+2} \,\left( \delta_{i,j} -\frac{z_i  z_j}{|z|^2} \right),\medskip\\
 b_i(z) & = & \sum_k \partial_k a_{i,k}(z) = -2 \,z_i \, |z|^\g,  \medskip\\
 c(z) & = & \sum_{k,l} \partial^2_{kl} a_{k,l}(z) = -2 \,(\g+3) \, |z|^\g. \\
\end{array}\right.
\end{equation*}
For any $f \in L^{1}_{2+\g}(\R^{3})$, we define then the matrix-valued mappings $\bm{\sigma}[f]$ and $\bm{\Sigma}[f]$ given by
$$\bm{\sigma}[f]=\big(\bm{\sigma}_{ij}[f]\big)_{ij}:=\big(a_{ij}\ast f\big)_{ij}, \qquad \qquad \bm{\Sigma}[f]=\bm{\sigma}[f(1-\dd\,f)].$$
In the same way, we set $\bm{b}[f]\::\:v \in \R^{3} \mapsto \bm{b}[f](v) \in \R^{3}$ given by 
$$\bm{b}_{i}[f](v)=\big(b_{i} \ast f\big)(v), \qquad \forall\, v \in \R^{3},\qquad i=1,2,3.$$
We also introduce
$$\bm{B}[f]=\bm{b}[f(1-\dd\,f)], \qquad \text{ and } \qquad  \bm{c}_{\g}[f]=c \ast f.$$
We emphasise the dependency with respect to the parameter $\g$ in $\bm{c}_{\g}[f]$ since, in several places, we apply the same definition with $\g+1$ replacing $\g$.

\smallskip
\noindent
With these notations, the LFD equation can then be written alternatively under the form
\begin{equation}\label{LFD}
\left\{
\begin{array}{ccl}
\;\partial_{t} f &= &\grad \cdot \big(\,\bm{\Sigma}[f]\, \grad f
- \bm{b}[f]\, f(1-\dd f)\big) , \medskip\\
\;f(t=0)&=&f_{\mathrm{in}}\,.
\end{array}\right.
\end{equation}
\begin{defi}\label{def15} Consider a non trivial initial datum $f_{\mathrm{in}} \in L^{1}_{2}(\R^{3})$ satisfying \eqref{hypci}--\eqref{eq:Mass} with $\dd_{0} >0$ and let $\dd \in (0,\dd_{0}]$. A weak solution to the LFD equation \eqref{LFD}
 is a function $f\::\:\R^{+}\times\R^{3}\to\R^{+}$ satisfying the following conditions:
\begin{enumerate}[(i)]
\item $f \in L^{\infty}(\R^{+};L^{1}_{2}(\R^{3})) \bigcap \mathscr{C}(\R^{+},\mathscr{D}'(\R^{3}))$,    
\item $f(t) \in \mathcal{Y}_{\dd}(f_{\mathrm{in}})$ for any $t \geq 0$  and $f(0)=f_{\mathrm{in}}$,
\item The mapping $t\mapsto\mathcal{S}_{\dd}(f(t))$ is {non-decreasing,} 
\item For any $\varphi=\varphi(t,v) \in \mathscr{C}_{c}^{2}([0,T)\times \R^{3})$,
\begin{multline}\label{weakform}
-\int_{0}^{T}\d t\int_{\R^{3}}f(t,v)\partial_{t}\varphi(t,v) \d v-\int_{\R^{3}}f_{\mathrm{in}}(v)\varphi(0,v)\d v\\
=\int_{0}^{T}\d t \int_{\R^{3}}\sum_{i,j}\bm{\Sigma}_{i,j}[f(t)]f(t,v)
\partial^{2}_{v_{i},v_{j}}\varphi(t,v)\d v\\+
\sum_{i=1}^{3}\int_{0}^{T}\d t\int_{\R^{6}}f(t,v)f(t,w)(1-\dd f(t,w))\\\bm{b}_{i}(v-w)\left[\partial_{v_{i}}\varphi(t,v)-\partial_{w_{i}}\varphi(t,w)\right]\d v\d w.\end{multline}
\end{enumerate}
\end{defi}
\noindent
Notice that, since $f(t) \in \mathcal{Y}_{\dd}(f_{\rm in})$, one has in particular $0 \leq f(t) \leq \dd^{-1}$ for any $t \geq0$. Since $\varphi$ has compact support together with its derivatives, all the terms in \eqref{weakform} are well defined.

\subsection{Main Results}
As mentioned, we study the existence, uniqueness, smoothness, large velocity and large time behavior of solutions to the spatially homogeneous Landau-Fermi-Dirac equation \eqref{LFD} with moderately soft potentials. We now present our main results and insist that all estimates provided are uniform in the vanishing limit of the quantum parameter $\dd$. 
\medskip

We start with a result of existence of weak solutions:

\begin{theo}\label{existence}
Let $\gamma \in (-2,0]$. Consider an initial datum $f_{\mathrm{in}}\in L^1_{s_0}(\R^3)$ for some $s_0>2$ satisfying \eqref{hypci}--\eqref{eq:Mass} with $\dd_{0} >0$. Then, for any $\dd \in (0,\dd_{0}]$ there exists a weak solution $f$ to \eqref{LFD} and one has $f\in L^\infty_{\mathrm{loc}}(\R_+,L^1_{s_0}(\R^3))$.
\end{theo}

The proof of this existence result can be found in Appendix \ref{app:cauchy}. It follows the same lines as the proof of the analogous theorem in the
hard potential case in \cite{bag}. We recall that for the classical Landau equation (that is for $\dd=0$) the theory of existence for the case when $\gamma<-2$ (very soft potentials) is substantially different from
the case $\gamma > -2$ (moderately soft potentials) \cite{DesvJFA,Wu}; we do not investigate the LFD equation with very soft potentials in this paper.
\medskip

We now turn to a result of smoothness which holds uniformly with respect to $\dd$, for any given time interval $[0,T]$, with $T>0$. Uniformity with respect to $T$ is not obtained at this level, and is considered only in next result.  
 \begin{theo}\label{smoothn}
Let $\gamma \in (-2,0)$. Consider an initial datum $f_{\mathrm{in}}\in L^1_{s}(\R^3) \cap L^{q_0}(\R^3)$ for all $s \ge 0$ and some $q_0\ge 2$, satisfying \eqref{hypci}--\eqref{eq:Mass} with $\dd_{0} >0$. Then, for any $\dd \in (0,\dd_{0}]$, any weak solution to equation \eqref{LFD} constructed in Theorem \ref{existence} lies in $L^{\infty}([0,T]; L^q_s(\R^3))$ for all $s\ge 0$, $q \in [1,q_{0})$ and $T>0$. 

\smallskip
\noindent
Moreover if the initial datum $f_{\mathrm{in}}$ also lies in $W^{1,p}_s(\R^3)$ for all $s\ge 0$ and all $p \in [1, \infty)$,  any weak solution constructed in Theorem \ref{existence} lies in $L^{\infty}([0,T]; W^{1,p}_s(\R^3)) \cap {L^2}([0,T]; H^{2}_s(\R^3))$ for all $s\ge 0$, $p \in [1, \infty)$ and $T>0$, as well as in $\mathscr{C}^{0,\alpha}([0,T] \times \R^3)$ for some $\alpha \in (0,1)$ and all $T>0$. 
Finally, all the norms of $f$ in the spaces described in this Theorem are uniform with respect to $\dd \in [0,\dd_0]$ and depend on the $W^{1,p}_s(\R^3)$ norms of $f_{\rm in}$ as well as $H(f_{\rm in})$.
\end{theo}
The fact that the solution $f=f(t,v)$ belongs to $\mathscr{C}^{0,\alpha}([0,T] \times \R^{3})$ can be used to show that $f$ is in fact a \emph{classical solution}.  The proof of this result of propagation of regularity can be found in Appendix \ref{ree}, see in particular Corollary \ref{cor:A6}. It follows the methods used in \cite{Wu} and \cite{DesvJFA}. Notice that stability (for finite intervals of time) and consequently uniqueness can be investigated thanks  to the study of smoothness (for sufficiently smooth initial data). 

\par
It can be improved in many directions: The assumptions on initial data can be changed (cf. the various propositions in Appendix \ref{ree}); Appearance of regularity can be shown (this can also be seen in the various propositions in Appendix \ref{ree}); The dependence w.r.t. time of the estimates can be obtained explicitly (and involves only powers and no exponentials, since Gronwall's lemma is not used), we refer to next theorem for the use of the large time behavior for obtaining uniformity w.r.t. time when (polynomial) moments of sufficient order are initially finite. Note that stretched exponential moments can be considered instead of algebraic moments, as is done in Section \ref{strexp}. 
\medskip

Concerning the long-time behaviour of the solution to \eqref{LFD}, the main result of this work can be summarised in the following theorem.

\begin{theo}\label{theo:main}
Assume that $\gamma\in\left(-2,0\right)$ and consider a nonnegative initial datum $f_{\mathrm{in}}$ satisfying \eqref{hypci}--\eqref{eq:Mass} with $\dd_{0} >0$, with moreover {$f_{\rm in} \in L^{1}_{s}(\R^{3})$ with $s > 14+6|\g|.$} Then, there exists $\dd_{\star} \in (0,\dd_{0}]$ depending only on $f_{\mathrm{in}}$ through its $L^{1}_{s}$-norm such that for any $\dd \in (0,\dd_{\star}]$, any nonnegative weak solution $f:= f(t,v)$  to \eqref{LFD} constructed in Theorem \ref{existence} satisfies:
\begin{enumerate}
\item \textit{\textbf{No Saturation:}} 
$$\kappa_{0} := 1 - \dd\,\sup_{t\geq1}\| f(t) \|_{\infty}>0.$$
\item \textit{\textbf{Algebraic Relaxation:}} there exists $C>0$ depending only  on  $\|f_{\mathrm{in}}\|_{L^{1}_{2}}$, $H(f_{\rm in})$ and  {$s$} such that
$$\mathcal{H}_{\dd}(f(t)|\M_{\dd})
\leq C\,\left(1+t\right)^{ {-\frac{s-8-6|\g|}{2|\g|}}}\,, \qquad t \geq1\,,$$
which implies in particular that
$$
\left\|f(t)-\M_{\dd}\right\|_{L^{1}} \leq \sqrt{2C}\,\left(1+t\right)^{{-\frac{s-8-6|\g|}{4|\g|}}}\,, \qquad t \geq 1.
$$
\end{enumerate}
{Finally,  if$$f_{\rm in} \in L^{1}_{r}(\R^{3}) \qquad \text{ with } \quad r >\max\left({2s + 8 + 2|\g|},\frac{s^{2}}{s-2|\g|}\right),$$
then there exists a constant $C(\g,s,f_{\mathrm{in}})$ depending on $H(f_{\rm in})$, $s$, $\|f_{\rm in}\|_{L^{1}_{2}}$ and $\|f_{\rm in}\|_{L^{1}_{r}}$ such that, for any $\dd \in (0,\dd_{\star})$
\begin{equation}\label{eq:unifo}
\sup_{t\geq1}\bm{E}_{s}(t) + \sup_{t\geq1}\| f(t) \|_{L^{\infty}} \leq   C(\gamma,s,f_{\mathrm{in}})\,.
\end{equation}}
We emphasise that the constants used above do not depend on $\dd$.\end{theo}

Notice that it is possible to interpolate the decay towards equilibrium in $L^1$ and estimate \eqref{eq:unifo} in order to get a decay towards equilibrium in $L^p$, for any $p \in (1, \infty)$, for suitable initial data.

The result of no saturation described above is crucial for the LFD equation. It was obtained in \cite{ABL} in the case of hard potentials using an indirect approach based on the analysis of higher regularity of solutions to ensure an $L^{\infty}$-bound independent of $\dd$ by Sobolev embedding.  In this work the approach is direct; it uses on one hand a careful study of the $L^1$ and $L^2$ moments of the solution of the equation, and on the other hand an original use of De Giorgi's level set method, see Theorem \ref{Linfinito*} hereafter for more details. In both cases, a repeated use of the following technical result will be made.
\begin{prop}\label{prop:GG} Assume that $-2 < \g < 0$ and $f_{\rm in}$ satisfies \eqref{hypci}--\eqref{eq:Mass} with $\dd_{0} >0$. For any $\dd \in (0,\dd_{0}]$, any $g \in \mathcal{Y}_{\dd}(f_{\mathrm{in}})$ and any smooth and compactly supported function $\phi$, there is $C_{0} >0$ (depending only on $\|f_{\rm in}\|_{L^{1}_{2}}$) such that

\begin{multline}\label{eq:estimatc}
-\int_{\R^{3}}\phi^{2}\bm{c}_{\g}[g]\d v \leq \delta\,\int_{\R^{3}}\left|\nabla \left(\langle v\rangle^{\frac{\g}{2}}\phi(v)\right)\right|^{2}\d v 
+C_{0}(1+\delta^{\frac{\g}{2+\g}})\int_{\R^{3}}\phi^{2}\langle v\rangle^{\g}\d v, \qquad \forall \,\delta >0.\end{multline}
\end{prop} {The above inequality \eqref{eq:estimatc} has been established in \cite[Theorem 2.7]{GG} with harmonic tools and study of $A_{p}$-weights. This inequality is referred to as a $\delta$-Poincar\'e inequality in \cite{GG}. The proof of \cite{GG} can be applied without major difficulty to the Landau-Fermi-Dirac context. We nevertheless provide here an elementary proof, based in particular on Pitt's inequality \cite{beckner}, with a slightly sharper estimate \eqref{eq:estimatc}. On the counterpart, our method seems to apply only for the range of parameters considered here, i.e. $-2 < \gamma < 0$. 
 Related convolution inequalities will be then established in Section \ref{sec:convo} and exploited for the implementation of the De Giorgi method in Section \ref{sec:level}. }

\par
The aforementioned proposition plays a fundamental role in the establishment of the following $L^{1}$-$L^{2}$ moments estimates for the solutions to \eqref{LFD}:

\begin{theo}\label{theo:main-moments}
Assume that $-2 < \g < 0$  and let a nonnegative initial datum $f_{\mathrm{in}}$ satisfying \eqref{hypci}--\eqref{eq:Mass} for some $\dd_0 >0$ be given. For $\dd \in (0,\dd_0]$, let  $f(t,\cdot)$ be a weak-solution to \eqref{LFD}. {Assume that
$$\lm_{s}(0)  < \infty, \qquad s > 4 + |\g|.$$}
Then, there exists a positive constant $\bm{C}_{s} >0$ depending on $s$ and $f_{\mathrm{in}}$ through $\lm_{s}(0)$,   $\|f_{\mathrm{in}}\|_{L^{1}_{2}}$,  $H(f_{\rm in})$ such that\begin{equation}\label{res19}
\bm{E}_{s}(t)\leq \bm{C}_{s}\left(t^{-\frac{3}{2}}+t\right), \qquad \quad  \lm_{s}(t) \leq \bm{C}_{s}\left(1+t\right) \quad  t > 0\,.
\end{equation}
Moreover, there exists $\beta_{1} >0$ depending only on  $\|f_{\mathrm{in}}\|_{L^{1}_{2}}$, $H(f_{\rm in})$ and {$\lm_{\frac{3|\g|}{2}}(0)$} such that, for {$s > 6+|\g|$},
\begin{equation}\label{rmq:Csfinal}
\bm{C}_{s} \leq \beta_{1}\left[\left(\beta_{1}s\right)^{ {{\frac{8-\g}{4+2\g}}(s+\g-2)+1}}+ {2^{\frac{s}{|\g|}} {(1+s)^{\frac{5}{2}}} \lm_{s}(0)}\right].\end{equation}
\end{theo}

It is worth noticing that Theorem \ref{theo:main-moments} shows the \emph{instantaneous appearance of weighted $L^{2}$-norms independent of $\dd$}.  Similar to hard potentials \cite{ABL}, we are required to investigate \emph{simultaneously} the evolution of the $L^{1}$ and $L^{2}$ moments through the evolution of $\bm{E}_{s}(t)=\lm_{s}(t)+ \frac{1}{2}\lM_{s}(t)$ since the quantum parameter $\dd$ induces a strong coupling between the two kinds of moments.  Our estimate shows a linear time growth of the combined $L^{1}$ and $L^{2}$ moments which depends on  the moment of order $s$ \emph{only through the pre-factor $C_{s}$}. Such a  bound is fundamental for the proof of the main Theorem \ref{theo:main} which combines its \emph{slowly increasing character} with an interpolation technique based upon an entropy/entropy production estimate established in \cite{ABDL-entro}. The use of such an interpolation process is typical of soft potential cases for kinetic equations (and briefly described in \cite{ABDL-entro}). Notice that combining the relaxation result together with the aforementioned slowly increasing bound proves, \emph{a posteriori}, the uniform-in-time estimate \eqref{eq:unifo}.

\smallskip
\noindent
In fact, to prove the no-saturation result of Theorem \ref{theo:main}, the key point is the following pointwise estimate.

\begin{theo} \label{Linfinito*} Assume that $f_{\mathrm{in}}$ satisfies \eqref{hypci}--\eqref{eq:Mass} with $\dd_{0} >0$. For $\dd \in (0,\dd_{0}]$, let $f(t,v)$ be a weak solution to \eqref{LFD}. Let  $s > \frac{3}{2}|\g|$ be given  and assume that $f_{\mathrm{in}} \in L^{1}_{s}(\R^{3})$. Then, there is a positive constant $C$ depending only on  $s$, $\|f_{\rm in}\|_{L^{1}_{2}}$, $H(f_{\rm in})$   such that, for any $T > t_{*} >0$,
\begin{equation}\label{eq:Linf}
\sup_{t \in [t_{*},T)}\left\|f(t)\right\|_{L^{\infty}} \leq C\,\Big( 1+ t_{*}^{- \frac{3s}{4s-3|\g|}-\frac{3}{4}} \Big)\,\Big[\,\sup_{t \in [0,T)}\lm_{s}(t)\Big]^{ \frac{3|\g|}{4s-3|\g|}}\,.
\end{equation}
\end{theo} 
We mentioned previously that we prove Theorem \ref{Linfinito*} thanks to an original use of the level set method of De Giorgi which is a well-known tool for parabolic equations, see the recent surveys \cite{caff,vasseur}, and became quite recently efficient for the study of spatially inhomogeneous kinetic equations \cite{golse,GG}.  In the spatially homogeneous situation considered here, the method has the flavour of the approach introduced in \cite{ricardo} for the Boltzmann equation, and recently extended to the inhomogeneous framework in \cite{AMSY}. The implementation of the level set method  uses a new critical parameter $\gamma = - \frac{4}{3}$, which is possibly of technical nature, but could be significant even if its physical interpretation seems difficult to give. If $\g > -\frac{4}{3}$ indeed, one can pick here above $s=2$ so that $\sup_{t\geq 0}\lm_{s}(t) < \infty$, and, of course, \eqref{eq:Linf} yields a pointwise estimate for $f(t)$ independent of both $T$ and $\dd$, proving in a direct way the saturation property in Theorem \ref{theo:main}.

\medskip
{It is worth noticing that a related pointwise estimate has been obtained in the classical case $\dd=0$ in \cite{GG} for the range $-2 \leq \g <0.$ Namely, for solutions $f(t)$ to the classical Landau equation in $\R^{3}$, \cite[Theorem 2.1]{GG} asserts that there exists $C >0$ such that
$$f(t,v) \leq C\left(1+\frac{1}{t}\right)^{\frac{3}{2}}\,\langle v\rangle^{\frac{3}{2}|\g|}\,\qquad \quad t > 0,\quad v \in \R^{3}.$$
Clearly, our method of proof applies directly to this case and, in some sense, improves the result of \cite{GG} since combining \eqref{eq:Linf} with the uniform bound on the moments \eqref{eq:unifo} yields the bound
$$\sup_{t \geq t_{*}}\|f(t)\|_{L^{\infty}} \leq C\Big( 1+{t_{*}^{- \frac{3s}{4s-3|\g|}-\frac{3}{4}}} \Big), \qquad t_{*} > 0\,.$$
This  eliminates the need of the  polynomial weight $\langle v\rangle^{-\frac{3}{2}|\g|},$}   at the price of a 
 slightly worse estimate for the short-time behaviour (notice that since  $s >\frac{3}{2}|\g|$, we have $\frac{3s}{4s-3|\g|}+\frac{3}{4} >\frac{3}{2}$). 
\medskip

We indicate that if stretched exponential moments initially exist, then the convergence towards equilibrium can be proved to have a stretched exponential rate as well, similar to related works on the Landau equation, see for example \cite{CDH}.  A precise result is given in Theorem \ref{theo:expon}. We mention here that such a result uses again  interpolation technique between \emph{slowly increasing bounds} for $L^{1}$ and $L^{2}$ weighted estimates for the solution to \eqref{LFD} and the entropy/entropy production. The slowly increasing bounds for moments  associated to stretched exponential weights is deduced directly from Theorem \ref{theo:main-moments} by exploiting the fact that we kept track of the dependence of $\bm{C}_{s}$ in terms of $s$ in \eqref{rmq:Csfinal}.

\subsection{Organization of the paper} After this Introduction, the paper is organized as follows. Section \ref{sec:preli} collects several known results about the Fermi-Dirac entropy and the entropy production associated to \eqref{LFD} and solutions to \eqref{LFD}. We also present in this Section the proof of the technical result stated in Proposition \ref{prop:GG} as well as some other related convolution estimates.  Section \ref{sec:moments} is devoted to the study of both the $L^{1}$ and $L^{2}$ moments of solutions to \eqref{LFD}, culminating with the proof of Theorem \ref{theo:main-moments}. In Section \ref{sec:level} we implement De Giorgi's level set methods resulting in Theorem \ref{Linfinito*} whereas in Section \ref{sec:converge} we collect the results of the previous sections which, combined with the study of the entropy production performed in \cite{ABDL-entro}, allow to derive the algebraic convergence towards equilibrium in Theorem \ref{theo:main}. We upgrade this rate of convergence in Section \ref{strexp} showing a stretched exponential rate of convergence for solutions associated with initial datum with finite stretched exponential moments. The paper ends with two Appendices. Appendix \ref{ree} is devoted to some additional regularity estimates for solutions to \eqref{LFD} resulting in Theorem \ref{smoothn}. The full proof of Theorem \ref{existence} is then postponed in Appendix \ref{app:cauchy}. 

 \subsection*{Acknowledgments}  R. Alonso gratefully acknowledges the support from Conselho Nacional de Desenvolvimento Cient\'{i}fico e Tecnol\'{o}gico (CNPq), grant Bolsa de Produtividade em Pesquisa (303325/2019-4).  B. Lods gratefully acknowledges the financial support from the Italian Ministry of Education, University and Research (MIUR), ``Dipartimenti di Eccellenza'' grant 2018-2022 as well as the support  from the \textit{de Castro Statistics Initiative}, Collegio Carlo Alberto (Torino). R. Alonso, V. Bagland and B. Lods would like to acknowledge the support of the Hausdorff Institute for Mathematics where this work started during their stay at the 2019 Junior Trimester Program on Kinetic Theory.

\section{Preliminary results}\label{sec:preli}
\subsection{Boltzmann and Fermi-Dirac Entropy and entropy production}\label{sec:entrop}
Recall the classical Boltzmann entropy  
\begin{equation*}
H(f)=\int_{\R^{3}}f\log f\d v\,.
\end{equation*}
The Fermi-Dirac entropy is introduced as
\begin{align}\label{eq:FDentro}
\begin{split}
\mathcal{S}_{\dd}(f) &= - \dd^{-1}\int_{\R^3} \Big[\dd f\log(\dd f)+(1-\dd f)\log (1-\dd f)\Big] \, \d v\\
&= -\dd^{-1}\big(H(\dd\,f)+H(1-\dd f)\big)\,.
\end{split}
\end{align}
The Fermi-Dirac relative entropy is defined as follows: given nonnegative $f,\,g \in L^1_2(\R^3)$ with $0 \leq f \leq \dd^{-1}$ and  $0 \leq g \leq \dd^{-1}$, set
$$\mathcal{H}_{\dd}(f|g)=-\mathcal{S}_{\dd}(f)+\mathcal{S}_{\dd}(g).$$
For the Fermi-Dirac relative entropy, a \emph{two-sided} Csisz\'ar-Kullback inequality holds true (see \cite[Theorem 3]{LW}).  There exists $C >0$ (depending only on $\dd$ and $\|g\|_{L^{1}_{2}}$) such that
\begin{equation}\label{eq:czisz}
\|g-\M_{\dd}\|_{L^{1}}^{2} \leq \left(2\int_{\R^{3}}g(v)\d v\right)\,\mathcal{H}_{\dd}(g|\M_{\dd}) \leq C\,\|g-\M_{\dd}\|_{L^{1}_{2}}.\end{equation}
 The long time behaviour of the solutions of the equation will be studied using the classical method consisting in comparing the relative entropy with the entropy production.  In our case, the entropy production is defined as
\begin{equation}\label{eq:Dee}
\mathscr{D}_{\dd}(g):=-\int_{\R^{3}}\Q(f)\big[\log f(v)-\log(1-\dd f(v))\big]\d v\,.
\end{equation} 
One can show that
\begin{equation}\label{eq:product}
\mathscr{D}_{\dd}(g)=\frac{1}{2} \int\int_{\R^{3}\times\R^{3}}\Psi(v-\vet)\, \bm{\Xi}_{\dd}[g](v,\vet)\d v\d\vet\,,\qquad \Psi(z)=|z|^{\g+2}\,,
\end{equation}
for any smooth function $0 < g < \dd^{-1}$, with 

\begin{align}\label{eq:Xidd}
\begin{split}
\bm{\Xi}_{\dd}[g](v,\vet)&:=\Pi(v-\vet)\Big(g_{\ast}(1-\dd g_{\ast})\nabla g - g(1-\dd g)\nabla g_{\ast}\Big)\left(\frac{\nabla g}{g(1-\dd g)}-\frac{\nabla g_{\ast}}{g_{\ast}(1-\dd g_{\ast})}\right)\\
&=gg_{\ast}(1-\dd g)(1-\dd g_{\ast})\left|\Pi(v-\vet)\left(\frac{\nabla g}{g(1-\dd g)}-\frac{\nabla g_{\ast}}{g_{\ast}(1-\dd g_{\ast})}\right)\right|^{2} \geq 0\,.
\end{split}
\end{align}
A thorough analysis of the link between the Landau-Fermi-Dirac entropy and its entropy production $\mathscr{D}_{\dd}$ has been established by the authors in a previous contribution \cite{ABDL-entro}, and we refer to the \emph{op. cit.} for more details on the topic.

\subsection{General estimates}
One has the following result, refer to \cite[Lemma 2.3 \& 2.4]{ABL}.
\begin{lem}\label{L2unif}
Let $0\leq f_{\mathrm{in}}\in L^{1}_{2}(\R^{3})$ be fixed and  satisfying \eqref{hypci}--\eqref{eq:Mass}  for some $\dd_{0} >0$.   Then, for any $\dd \in (0,\dd_{0}]$, the following hold:
\begin{enumerate}
\item For any  $f \in \mathcal{Y}_{\dd}(f_{\mathrm{in}})$, it holds that
\begin{equation}\label{e0}
\inf_{0<\dd\leq \dd_{0}}\int_{|v|\leq R(f_{\mathrm{in}})} f(1-\dd f)\, \d v \geq \eta(f_{\mathrm{in}})>0\,,
\end{equation}
for some $R(f_{\mathrm{in}})>0$ and $\eta(f_{\mathrm{in}})$ depending only on $\|f_{\rm in}\|_{L^{1}_{2}}$ and $H(f_{\mathrm{in}})$ but not on $\dd$.
\smallskip
\item For any $\delta >0$ there exists $\eta(\delta)>0$ depending only on $\|f_{\rm in}\|_{L^{1}_{2}}$ and  $H(f_{\mathrm{in}})$ such that for any $f \in \mathcal{Y}_{\dd}(f_{\mathrm{in}})$, and any measurable set $A\subset \R^3$, 
\begin{equation}\label{Lem6DV}
|A|\leq \eta(\delta) \Longrightarrow \int_A f(1-\dd f)\, \d v \leq \delta.
\end{equation}
\end{enumerate}
\end{lem}
\noindent
A consequence of Lemma \ref{L2unif} is the following technical result which will be used for the study of moments.
\begin{lem}\label{lem:jensen} Let $0\leq f_{\mathrm{in}}\in L^{1}_{2}(\R^{3})$ be fixed and bounded satisfying \eqref{hypci}--\eqref{eq:Mass}  for some $\dd_{0} >0$.  Let $\g <0$.
Then, there exists $\eta_{\star} >0$ depending only on  $H(f_{\mathrm{in}})$ and $\|f_{\rm in}\|_{L^1_2}$  such that, for any $\dd \in (0,\dd_{0}]$ and any $f \in \mathcal{Y}_{\dd}(f_{\mathrm{in}})$,
 one has
\begin{equation}
\int_{\R^{3}}\left(1+|v-\vet|^{2}\right)^{\frac{\g}{2}}\,f(\vet)\left(1-\dd\,f(\vet)\right)\d\vet \geq \eta_{\star}\langle v\rangle^{\g}, \qquad \forall\, v \in \R^{3}.
\end{equation}
\end{lem}

\begin{proof} For simplicity, given $f \in \mathcal{Y}_{\dd}(f_{\mathrm{in}})$, we set $F=f(1-\dd\,f)$. From Lemma \ref{L2unif}, 
$$\varrho_{F}:=\int_{\R^{3}}F(\vet)\d\vet \geq \eta(f_{\mathrm{in}}) >0.$$
Let $v \in \R^{3}$ be fixed and define the probability measure $\d\mu$ over $\R^{3}$ by
$$\mu(\d\vet)=F(v-\vet)\frac{\d \vet}{\varrho_{F}}.$$
We introduce the convex function $\Phi(r)=(1+r)^{\frac{\g}{2}}$, $r >0$. One has, thanks to Jensen's inequality,
\begin{equation*} 
\int_{\R^{3}}\left(1+|v-\vet|^{2}\right)^{\frac{\g}{2}} F(\vet)\d\vet=\varrho_{F}\,\int_{\R^{3}}\Phi(|\vet|^{2})\mu(\d\vet)\\
\geq \varrho_{F}\,\Phi\left(\int_{\R^{3}}|\vet|^{2}\mu(\d\vet)\right).\end{equation*}
Now,
$$\int_{\R^{3}}|\vet|^{2}\mu(\d\vet)=\frac{1}{\varrho_{F}}\int_{\R^{3}}|v-\vet|^{2}F(\vet)\d\vet \leq 2|v|^{2}+\frac{6}{\varrho_{F}} ,$$
and, since $\Phi$ is nonincreasing,{
\begin{equation*} \varrho_{F}\,\int_{\R^{3}}\Phi(|\vet|^{2})\mu(\d\vet) \geq \varrho_{F}\Phi\left(2|v|^{2}+\frac{6}{\varrho_{F}}\right)
 \geq \varrho_{F}\Phi\left(\frac{6+6|v|^{2}}{\varrho_{F}}\right){\ge 12^{\frac{\g}{2}}}\varrho_{F}^{1-\frac{\g}{2}}\langle v\rangle^{\g}\,,\end{equation*}
where we used that $\varrho_{F} \leq 1$ thanks to \eqref{eq:Mass} and $\Phi(r)\ge (2r)^{\frac{\g}{2}}$ for $r>1$. Since $\varrho_{F}^{1-\frac{\g}{2}} \geq \eta(f_{\mathrm{in}})^{\frac{2-\g}{2}}$ the result follows with 
  $\eta_{\star}= {12^{\frac{\g}{2}}}\eta(f_{\mathrm{in}})^{\frac{2-\g}{2}} >0.$}
\end{proof}
The following coercivity estimate for the matrix $\bm{\Sigma}[f]$ holds.  Its proof is a copycat of \cite[Proposition 2.3]{ALL} applied to $F=f(1-\dd f)$ after using Lemma \ref{L2unif} appropriately.
\begin{prop}\label{diffusion}
Let $0\leq f_{\mathrm{in}}\in L^{1}_{2}(\R^{3})$ be fixed and satisfying \eqref{hypci}--\eqref{eq:Mass}  for some $\dd_{0} >0$.  Then, there exists a  constant $K_{0} > 0,$ depending on {$H(f_{\mathrm{in}})$ and $\|f_{\rm in}\|_{L^{1}_{2}}$} but not $\dd$, such that
$$\forall\, v,\, \xi \in \R^3, \qquad 
\sum_{i,j} \, \bm{\Sigma}_{i,j}[f](v) \, \xi_i \, \xi_j 
\geq K_{0} \langle v \rangle^{\g} \, |\xi|^2 , $$
holds for any $\dd \in [0, \dd_0]$ and $f \in \mathcal{Y}_{\dd}(f_{\mathrm{in}})$.
\end{prop}

\subsection{Convolution inequalities}\label{sec:convo}

We establish here some of the main technical tools used in the paper. We begin with the proof of Proposition \ref{prop:GG} stated in the introduction,
 which provides suitable estimates on the zero-th order term $\bm{c}_{\g}[g]=-2(\g+3)|\cdot|^{\g}\ast g$.
\begin{proof}[Proof of Proposition \ref{prop:GG}] Let $g \in \mathcal{Y}_{\dd}(f_{\mathrm{in}})$ be fixed. For a given nonnegative $\phi$, set
$$I[\phi] :=-\int_{\R^{3}}\phi^{2}\bm{c}_{\g}[g]\d v=2\,(\g+3)\int_{\R^{3}\times \R^{3}}|v-\vet|^{\g}\phi^{2}(v)g(\vet)\d v\d\vet.$$
For any $v,\vet \in \R^{6}$, if $|v-\vet| < \frac{1}{2}\langle v\rangle$, then $\langle v\rangle\leq 2\langle \vet\rangle$, and
we deduce from this, see \cite[Eq. (2.5)]{amuxy},
\begin{equation}\label{eq:ineq}
|v-\vet|^{\g}\leq 2^{-\g}\langle v\rangle^{\g}\left(\ind_{\left\{|v-\vet| \geq\frac{\langle v\rangle}{2}\right\}}+
\langle \vet\rangle^{-\g}|v-\vet|^{\g}\ind_{\left\{|v-\vet|< \frac{\langle v\rangle}{2}\right\}}\right).\end{equation}
Thanks to this inequality, we get $I[\phi] \leq 6\cdot 2^{-\g}\left(I_{1}+I_{2}\right)$,
with
$$I_{1}=\int_{\R^{3}}\langle v\rangle^{\g}\phi^{2}(v)\d v \int_{|v-\vet| \geq \frac{\langle v\rangle}{2}} g(\vet)\d\vet \leq \, \|f_{\rm in}\|_{L^{1}} \|\langle \cdot\rangle^{\g}\phi^{2}\|_{L^{1}}\,,$$
while
$$I_{2}=\int_{\R^{3}}\langle \vet\rangle^{-\g}g(\vet)\d \vet\int_{|v-\vet| < \frac{1}{2}\langle v\rangle }|v-\vet|^{\g}\langle v\rangle^{\g}\phi^{2}(v)\d v.$$
Set $\psi(v)=\langle v\rangle^{\frac{\g}{2}}\phi(v)$,
from which
$$I_{2} \leq \int_{\R^{3}}\langle \vet\rangle^{-\g}g(\vet)\d \vet\int_{\R^{3}}|v-\vet|^{\g}\psi^{2}(v)\d v.$$
According to Pitt's inequality which reads, in $\R^{n}$, $\int_{\R^{n}}|x|^{-\alpha}|f(x)|^{2}\d x \lesssim \int_{\R^{n}}|\xi|^{\alpha}\,\left|\widehat{f}(\xi)\right|^{2}\d \xi$ for any $0 < \alpha < n$,  \cite{beckner}, there is a universal constant $c >0$ such that, for any $\vet \in \R^{3}$,
$$\int_{\R^{3}}|v-\vet|^{\g}\psi^{2}(v)\d v=\int_{\R^{3}}|v|^{\g}|\psi(v-\vet)|^{2}\d v \leq c\int_{\R^{3}}|\xi|^{-\g}\,\left|\widehat{\tau_{\vet}\psi}(\xi)\right|^{2}
\d \xi , $$
where $\tau_{\vet}\psi(\cdot)=\psi(\cdot-\vet).$ Since $|\widehat{\tau_{\vet}\psi}(\xi)|=|\widehat{\psi}(\xi)|$, we get
$$\int_{\R^{3}}|v-\vet|^{\g}\psi^{2}(v)\d v \leq c\int_{\R^{3}}|\xi|^{-\g}\,|\widehat{\psi}(\xi)|^{2}\d\xi
. $$
 This results in 
\begin{equation*}\begin{split}
I_{2} &\leq c\left(\int_{\R^{3}}\langle \vet\rangle^{-\g}g(\vet)\d \vet\right)\,\int_{\R^{3}}|\xi|^{-\g}\,|\widehat{\psi}(\xi)|^{2}\d \xi\\
&\leq c\,{\|f_{\rm in}\|_{L^{1}_{2}}}\,\int_{\R^{3}}|\xi|^{-\g}\,|\widehat{\psi}(\xi)|^{2}\d \xi=:c\,{\|f_{\rm in}\|_{L^{1}_{2}}}\,J, \end{split}\end{equation*}
where we used that $-\g < 2$.  Now, for any $R >0$, we split the above integral ${J}$ in Fourier variable as
$${J}=\int_{|\xi| < R}|\xi|^{-\g}|\widehat{\psi}(\xi)|^{2}\d\xi+\int_{|\xi| \geq R}|\xi|^{-\g}|\widehat{\psi}(\xi)|^{2}\d\xi={J}_{1}+{J}_{2}.$$
On the one hand, using Parseval identity, $J_{1} \leq R^{-\g}\|\psi\|_{L^{2}}^{2}=R^{-\g}\|\langle \cdot \rangle^{\g}\phi^{2}\|_{L^{1}}.$
On the other hand,
$${J}_{2}=\int_{|\xi| \geq R}|\xi|^{-(2+\g)}|\xi|^{2}|\widehat{\psi}(\xi)|^{2}\d\xi \leq R^{-(2+\g)}\int_{\R^{3}}|\xi|^{2}\,|\widehat{\psi}(\xi)|^{2}\d\xi , $$
that is,
${J}_{2}\leq R^{-(2+\g)}\left\|\nabla \psi\right\|_{L^{2}}^{2}.$
Thus,
$$J \leq   R^{-\g}\|\langle \cdot \rangle^{\g}\phi^{2}\|_{L^{1}}+ R^{-(2+\g)}\|\nabla \psi\|_{L^{2}}^{2} , $$
and
$$I[\phi] \leq {6\cdot }2^{-\g}{\|f_{\rm in}\|_{L^{1}_{2}}}\left((1 +cR^{-\g})\|\langle \cdot \rangle^{\g}\phi^{2}\|_{L^1}
+c\,R^{-(2+\g)}\|\nabla \psi\|_{L^{2}}^{2}\right)$$
for any $R >0$. This proves \eqref{eq:estimatc} with 
$\delta=6\cdot 2^{-\g}{\|f_{\rm in}\|_{L^{1}_{2}}}cR^{-(2+\g)}={6\cdot 2^{2-\g}cR^{-(2+\g)}}.$
\end{proof}
An alternative version of the above estimate involving $L^{p}$-norms instead of Pitt's inequality is given by the following Proposition, which now holds for the whole range of parameters between $(-3,0)$. In the sequel, we call the parameter $\lambda \in (-3,0)$ instead of $\g$ since we will apply the inequality later to $\lambda=\g$, $\lambda=\g+1$, etc.
\begin{prop}\label{Lemma-LS-1} Let $\lambda >-3$ and $p>1$ be such that $-\lambda\,q<3$
where $\frac{1}{p}+\frac{1}{q}=1$. Then there exists $C_{p}(\lambda) >0$ such that 
\begin{equation}\label{eq:estimacR}
\left|\int_{\R^{3}}\left(|\cdot|^{\lambda} \ast g\right)(v)\varphi(v)\d v\right| \leq 
\begin{cases} C_{p}(\lambda)\|\langle\cdot\rangle^{-\lambda}g\|_{L^{1}}\Big(\|\langle \cdot \rangle^{\lambda}\varphi\|_{L^{1}} + \|\langle \cdot \rangle^{\lambda}\varphi\|_{L^{p}}\Big)&\text{if }\, \lambda <0\,,\\
 \\
\qquad\qquad\qquad\|\langle\cdot\rangle^{\lambda}g\|_{L^{1}}\|\langle\cdot\rangle^{\lambda}\varphi\|_{L^{1}}&\text{if }\, \lambda \geq0.
\end{cases}
\end{equation}
\end{prop}
\begin{proof} For $\lambda \geq 0$, the result is trivial since $|v-\vet| \leq \langle v\rangle\,\langle \vet\rangle$ for any $v,\vet\in \R^{3}.$ Let us consider the case $-3<\lambda< 0.$ We can assume without loss of generality that $\varphi$ and $g$ are nonnegative. Write 
$$\mathscr{I}:=\int_{\R^{6}}g(\vet)\varphi(v)|v-\vet|^{\lambda}\d v\d \vet.$$
Using the inequality $\langle v \rangle \leq \sqrt{2} \langle v - \vet \rangle\langle \vet \rangle,$ which holds for any $v,\vet \in \R^{3}$, we get

\begin{multline*}
2^{\frac{\lambda}{2}}\mathscr{I} \leq \int_{\R^{3}}\langle \vet\rangle^{-\lambda}g(\vet)\d\vet\int_{\R^3}|v-\vet|^{\lambda}\langle v - \vet\rangle^{-\lambda}\langle v \rangle^{\lambda}\varphi(v)\, \d v \,\\
=\int_{\R^{3}}\langle \vet\rangle^{-\lambda}g(\vet)\d\vet\int_{|v-\vet|<1}|v-\vet|^{\lambda}\langle v - \vet\rangle^{-\lambda}\langle v \rangle^{\lambda}\varphi(v)\, \d v \\
+\int_{\R^{3}}\langle \vet\rangle^{-\lambda}g(\vet)\d\vet\int_{|v-\vet|\geq 1}|v-\vet|^{\lambda}\langle v - \vet\rangle^{-\lambda}\langle v \rangle^{\lambda}\varphi(v)\, \d v=:\mathscr{I}_{1}+\mathscr{I}_{2}.\end{multline*}
For a given $\vet \in \R^{3}$, on the set $\{|v-\vet| \geq 1\}$, we have $|v-\vet| \leq \langle v-\vet\rangle \leq \sqrt{2}\, |v-\vet|$ so that
$$
\mathscr{I}_{2} \leq 2^{-\frac{\lambda}{2}}\int_{\R^{3}}\langle \vet\rangle^{-\lambda}g(\vet)\d\vet\int_{|v-\vet|\geq1}\langle v \rangle^{\lambda}\varphi(v)\, \d v \leq 2^{-\frac{\lambda}{2}}\| \langle \cdot \rangle^{\lambda}\varphi \|_{L^{1}}\| \langle \cdot \rangle^{-\lambda}g\|_{L^{1}}\,.
$$
For a given $\vet \in \R^{3}$, on the set $\{|v - \vet|<1\}$, we have $\langle v -\vet\rangle\leq \sqrt{2}$.  Then, thanks to H\"older's inequality,
\begin{equation*}
\begin{split}
&\int_{|v-\vet| < 1}|v-\vet|^{\lambda}\langle v - \vet\rangle^{-\lambda}\langle v\rangle^{\lambda}\varphi(v)\, \d v \leq 2^{-\frac{\lambda}{2}}\int_{|v-\vet| < 1}|v-\vet|^{\lambda}\langle v\rangle^{\lambda}\varphi(v)\,\d v\\
&\quad\leq 2^{-\frac{\lambda}{2}}\|\langle \cdot\rangle^{\lambda}\varphi\|_{L^{p}}\,\left(\int_{|v-\vet|\leq 1}|v-\vet|^{\lambda\,q}\,\d v\right)^{\frac{1}{q}}=2^{-\frac{\lambda}{2}}\left(\frac{|\S^{2}|}{3+\lambda\,q}\right)^{\frac{1}{q}}\|\langle \cdot \rangle^{\lambda}\varphi\|_{L^{p}},
\end{split}
\end{equation*}
from which we deduce that
$$\mathscr{I}_{1} \leq 2^{-\frac{\lambda}{2}}\left(\frac{|\S^{2}|}{3+\lambda\,q}\right)^{\frac{1}{q}}\|\langle \cdot \rangle^{\lambda}\varphi\|_{L^{p}}\,\|\langle \cdot\rangle^{-\lambda}g\|_{L^{1}}.$$
This gives the result with $C_{p}(\lambda) :=2^{-\lambda}\max\left(1,\left(\frac{|\S^{2}|}{3+\lambda\,q}\right)^{\frac{1}{q}}\right).$
\end{proof}

\subsection{Consequences}

An important first consequence of Proposition \ref{prop:GG} is the following weighted Fisher estimate. Notice that a similar result can be deduced (for a larger range of parameters $\g <0$) from an alternative representation of the entropy in the spirit of \cite[Theorem 2]{Desv}, refer to \cite{ABDL-entro} for further details. 
\begin{prop}\label{cor:Fisherg} Let $0\leq f_{\mathrm{in}}\in L^{1}_{2}(\R^{3})$ be fixed and bounded satisfying \eqref{hypci}--\eqref{eq:Mass}  for some $\dd_{0} >0$.  Assume that 
$-2 < \gamma <0$
and $\dd \in (0,\dd_{0}]$. Then, there is a positive constant $C_{0}(\gamma)$ depending only on $f_{\mathrm{in}}$ through $\|f_{\rm in}\|_{L^{1}_{2}}$ and 
$H(f_{\mathrm{in}})$, such that for all $f \in \mathcal{Y}_{\dd}(f_{\mathrm{in}})$,
$$\int_{\R^{3}}\left|\nabla \sqrt{f(v)}\right|^{2}\langle v\rangle^{\gamma}\d v \leq C_{0}(\gamma)\left(1+\mathscr{D}_{\dd}(f)\right), \qquad \forall \,
\dd \in (0,\dd_{0}].$$
\end{prop}
\begin{proof} Let us fix $\dd \in (0,\dd_{0}]$ and $f \in \mathcal{Y}_{\dd}(f_{\mathrm{in}})$.  Recall from \eqref{eq:Dee} that $$\mathscr{D}_{\dd}(f)=-\int_{\R^{3}}\Q(f)\left[\log f(v)-\log(1-\dd f(v))\right]\d v, $$ 
where we recall that $\Q(f)=\nabla \cdot \left(\bm{\Sigma}[f]\nabla f - \bm{b}[f]\,F\right)$, $F=f(1-\dd f)$. Therefore,
\begin{multline}\label{eq:Dfb}
\mathscr{D}_{\dd}(f)=\int_{\R^{3}}\left(\bm{\Sigma}[f]\nabla f - \bm{b}[f]\,F\right)\cdot \nabla\left[\log f(v)-\log(1-\dd f(v))\right]\d v\\
=\int_{\R^{3}}\left(\bm{\Sigma}[f]\nabla f - \bm{b}[f]\,F\right)\cdot \frac{\nabla f}{F}\d v=\int_{\R^{3}}\frac{1}{F}\bm{\Sigma}[f]\nabla f \cdot \nabla f\d v + \int_{\R^{3}}f \nabla \cdot \bm{b}[f]\d v.\end{multline}
Using Proposition \ref{diffusion}, because $f \in \mathcal{Y}_{\dd}(f_{\mathrm{in}})$, one has
$$\bm{\Sigma}[f]\nabla f \cdot \frac{\nabla f}{F} \geq \frac{K_{0}}{f}\,\langle \cdot\rangle^{\g}\left|\nabla f\right|^{2} \geq 4K_{0}\langle \cdot \rangle^{\g}\left|\nabla \sqrt{f}\right|^{2} , $$
and, recalling $\nabla \cdot\bm{b}[f]=\bm{c}_{\g}[f]$, we deduce from \eqref{eq:Dfb} that
$$4K_{0}\int_{\R^{3}}\langle v\rangle^{\g}\left|\nabla \sqrt{f(v)}\right|^{2}\d v \leq \mathscr{D}_{\dd}(f)-\int_{\R^{3}}\bm{c}_{\g}[f]\,f\d v.$$
Then, applying Proposition \ref{prop:GG}, with $g=f$ and $\phi=\sqrt{f}$, there is $C_{0} >0$ such that, for any $\delta >0$,
\begin{multline*}
4K_{0}\int_{\R^{3}}\langle v\rangle^{\g}\left|\nabla \sqrt{f(v)}\right|^{2}\d v \leq \mathscr{D}_{\dd}(f) + \delta\,\int_{\R^{3}}\left|\nabla \left(\langle v\rangle^{\frac{\g}{2}}\sqrt{f(v)}\right)\right|^{2}\d v \\
+C_{0}(1+\delta^{\frac{\g}{2+\g}})\int_{\R^{3}}f(v)\langle v\rangle^{\g}\d v.\end{multline*}
Using that 
$$\int_{\R^{3}}\left|\nabla \left(\langle v\rangle^{\frac{\g}{2}}\sqrt{f(v)}\right)\right|^{2}\d v \leq 2\int_{\R^{3}}\langle v\rangle^{\g}\left|\nabla \sqrt{f(v)}\right|^{2}\d v +\frac{\g^{2}}{2}\int_{\R^{3}}\langle v\rangle^{\g}f(v)\d v , $$
we can choose $\delta >0$ small enough so that
\begin{equation}\label{eq:K0D}2K_{0}\int_{\R^{3}}\langle v\rangle^{\g}\left|\nabla \sqrt{f(v)}\right|^{2}\d v \leq \mathscr{D}_{\dd}(f) + C_{1}(\g)\int_{\R^{3}}\langle v\rangle^{\g}f(v)\d v\end{equation}
for some positive constant $C_{1}(\g)$ depending only on  $f_{\mathrm{in}}$. This gives the result.
\end{proof}
A significant consequence of the above result is the following corollary which regards solutions to the Landau-Fermi-Dirac equation \eqref{LFD}.

\begin{cor}\label{GradEntro} Assume $-2< \g< 0$ and let $f_{\rm in}$ be a nonnegative initial datum satisfying \eqref{hypci}--\eqref{eq:Mass}  for some $\dd_{0} >0$. Let $\dd \in (0,\dd_{0}]$ and  $f(t,\cdot)$ be  a weak solution to Landau-Fermi-Dirac equation, then for $0 < t_{1}< t_{2}$,
$$\int_{t_{1}}^{t_{2}}\d t \int_{\R^{3}}\left|\nabla_{v}\sqrt{f(t,v)}\right|^{2}\langle v\rangle^{\gamma}\d v \leq C_{0}(\gamma)\int_{t_{1}}^{t_{2}}\left(1+\mathscr{D}_{\dd}(f(t))\right)\d t,  \,$$
where $C_{0}(\g)$ is defined in Proposition \ref{cor:Fisherg}. As a consequence, there exist positive constants $\widetilde{C}_{0}$ and $\widetilde{C}_{1}$ depending only on  $\|f_{\rm in}\|_{L^{1}_{2}}$,  
$H(f_{\mathrm{in}})$  so that for  $0 < t_{1}< t_{2}$,
$$\int_{t_{1}}^{t_{2}}\d t \int_{\R^{3}}\left|\nabla_{v}\left(\langle v\rangle^{\frac{\g}{2}}\sqrt{f(t,v)}\right)\right|^{2}\d v \leq \widetilde{C}_{0}(1+t_{2}-t_{1})\,,$$
and
\begin{equation}\label{eq:t1t2}
\int_{t_{1}}^{t_{2}}\left\|\langle \cdot \rangle^{\g}f(t,\cdot)\right\|_{L^{3}}\d t \leq \widetilde{C}_{1}\big(1+t_{2}-t_{1}\big)  \qquad 0 < t_{1} < t_{2}\,.
\end{equation}
\end{cor}
\begin{proof}  The first inequality follows by simply integrating the inequality in Prop. \ref{cor:Fisherg}. In order to get the second inequality, we use
part $(iii)$ of Def. \ref{def15},
which ensures that  
$$ \int_{t_{1}}^{t_{2}}\mathscr{D}_{\dd}(f(t))\d t \le \mathcal{S}_{\dd}(\M_{\dd})-\mathcal{S}_{\dd}(f_{\mathrm{in}}).$$
Now, on the one hand,
$$ \mathcal{S}_{\dd}(\M_{\dd}) = -\log\dd  -\log  a_{\dd}  +3b_{\dd} +\frac1{\dd}\int_{\R^3}\log(1+\dd M_{\dd}(v))\d v, $$
and 
 $$\frac1{\dd}\int_{\R^3}\log(1+\dd M_{\dd}(v))\d v \le  \int_{\R^3} M_{\dd}(v)\d v = a_{\dd}\left(\frac{\pi}{b_{\dd}}\right)^{\frac32}.$$
On the other hand, 
\begin{align*}
\mathcal{S}_{\dd}(f_{\mathrm{in}})& = -\log(\dd) -H(f_{\mathrm{in}})-\frac1{\dd} \int_{\R^3} (1-\dd f_{\mathrm{in}}(v)) \log(1-\dd f_{\mathrm{in}}(v))\d v 
 \ge -\log \dd  -H(f_{\mathrm{in}}) .
\end{align*}
Hence, 
$$\mathcal{S}_{\dd}(\M_{\dd})-\mathcal{S}_{\dd}(f_{\mathrm{in}}) \le -\log a_{\dd} +3b_{\dd}  +a_{\dd}\left(\frac{\pi}{b_{\dd}}\right)^{\frac32}+H(f_{\mathrm{in}}). $$
It follows from \cite[Appendix A]{ABL} that $a_{\dd}$ and $b_{\dd}$ are uniformly bounded with respect to $\dd$. This means that 
$ \mathcal{S}_{\dd}(\M_{\dd})-\mathcal{S}_{\dd}(f_{\mathrm{in}}) \leq c_{0} < \infty,$
and the second inequality follows with $\widetilde{C}_{0}=C_{0}(\g)\max(1,c_{0})$ independent  of $\dd$. To prove \eqref{eq:t1t2}, we recall that the following Sobolev inequality
\begin{equation}\label{eq:sobC}
\|u\|_{L^{6}}\leq C_{\mathrm{Sob}}\,\|\nabla u\|_{L^{2}}, \qquad u \in H^{1}(\R^{3}),
\end{equation}
holds for some positive universal constant $C_{\mathrm{Sob}} >0$. Applying this with $u=\langle \cdot\rangle^{\frac{\g}{2}}\sqrt{f(t,\cdot)}$ which is such that $\|u\|_{L^{6}}^{2}=\|\langle \cdot \rangle^{\g}{f(t,\cdot)}\|_{L^{3}}$, one gets the result with $\widetilde{C}_{1}=C_{\mathrm{Sob}}\widetilde{C}_{0}.$
\end{proof}
One  can get rid of the degenerate weight in \eqref{eq:t1t2} to get a mere $L^{p}$ bound.  {We refer to \cite[Proposition 5.2]{GG} for a complete proof.}

\begin{lem}  Assume $-2< \g< 0$ and let $f_{\rm in}$ be a nonnegative initial datum satisfying \eqref{hypci}--\eqref{eq:Mass}  for some $\dd_{0} >0$. Let $\dd \in (0,\dd_{0}]$ and  $f(t,\cdot)$ be  a weak solution to Landau-Fermi-Dirac equation. Then, there exists $C >0$ depending only on $\|f_{\mathrm{in}}\|_{L^{1}_{2}}$, $H(f_{\mathrm{in}})$  such that for  $0 < t_{1} < t_{2}$,
$$\int_{t_{1}}^{t_{2}}\|f(t,\cdot)\|_{L^{p}}^{p}\d t \leq C\big(1+t_{2}-t_{1}\big),$$
holds with $p=\min\left(\frac{5}{3},\frac{3(2 +|\g|)}{2+3|\g|}\right).$\end{lem}
 \section{Moments estimates}\label{sec:moments} 
We study here the evolution of both $L^{1}_{s}$ and $L^{2}_{s}$ moments of weak solutions to \eqref{LFD}. Our goal is to prove Theorem \ref{theo:main-moments}.

\subsection{$L^{1}$-moments} We start with the following basic observation for the study of moments.
\begin{lem}\label{lem:mom}  Assume $-2< \g< 0$ and let $f_{\rm in}$ be a nonnegative initial datum satisfying \eqref{hypci}--\eqref{eq:Mass} for some $\dd_{0} >0$. Let $\dd \in (0,\dd_{0}]$ and  $f(t,\cdot)$ be  a weak solution to Landau-Fermi-Dirac equation. 
For any $s >2$, one has
\begin{equation}\label{eq:mom-s}
\frac{\d}{\d t}\int_{\R^{3}}f(t,v)\langle v\rangle^{s}\d v=\mathscr{J}_{s}(f,F)=\mathscr{J}_{s,1}(f,F)+\mathscr{J}_{s,2}(f,F), 
\end{equation}
where $F=f(1-\dd f)$ and, for any nonnegative measurable mappings $h,g \geq 0$ and $s >2$, we use the notations
\begin{align*}
\mathscr{J}_{s,1}(h,g)&=2s\int\int_{\R^{3}\times \R^{3}}h(v)g(\vet)\,|v-\vet|^{\g}\left(\langle v\rangle^{s-2}-\langle \vet\rangle^{s-2}\right)\left(|\vet|^{2}-(v \cdot \vet)\right)\d v\d\vet ,\\
\mathscr{J}_{s,2}(h,g)&=s(s-2)\int\int_{\R^{3}\times\R^{3}}\langle v\rangle^{s-4}h(v)g(\vet)|v-\vet|^{\g}\left(|v|^{2}\,|\vet|^{2}-(v\cdot {\vet})^{2}\right)\d v\d \vet .
\end{align*}
 Moreover, for any nonnegative $g$,
$$\mathscr{J}_{s,1}(g,g)=2s \int\int_{\R^{3} \times \R^{3}}g(v)\,g(v_{\ast})\,
|v-\vet|^{\gamma}\langle v\rangle^{s-2}\left(\langle \vet\rangle^{2}-\langle v\rangle^{2}\right)\d v\d \vet.$$
\end{lem}
\begin{proof} For a convex function $\Phi\::\:\R^{+} \to \R^{+}$, we get from \eqref{LFD}
$$\dfrac{\d}{\d t}\int_{\R^{3}}f(t,v)\,\Phi(|v|^{2})\d v=4\int_{\R^{3}}\d v\int_{\R^{3}}f\,F_{\ast}|v-\vet |^{\gamma}\,\Lambda^{\Phi}(v,\vet )\d \vet , $$
where $F=f(1-\dd f)$, and
\begin{equation*}
\Lambda^{\Phi}(v,\vet )=\left[|\vet |^{2}-(v\cdot \vet )\right]\left[\Phi'(|v|^{2})-\Phi'(|\vet |^{2})\right]
+\left[|v|^{2}|\vet |^{2}-(v \cdot \vet )^{2}\right]\,\Phi''(|v|^{2}).\end{equation*}
Picking $\Phi(r) := (1+r)^{\frac{s}{2}}$, one sees that
$$4\int_{\R^{3}}\d v\int_{\R^{3}}f\,F_{\ast}|v-\vet |^{\gamma}\,\Lambda^{\Phi}(v,\vet )\d \vet=\mathscr{J}_{s}(f,F).$$
Now, a symmetry argument  shows that
$$ \int\int_{\R^{3} \times \R^3}g\,g_{\ast}(v\cdot \vet )\left[\Phi'(|v|^{2})-\Phi'(|\vet |^{2})\right]\,|v-\vet |^{\gamma}\d v\d \vet =0 , $$
that is, 
$$\mathscr{J}_{s,1}(g,g)=2s \int\int_{\R^{3}\times \R^{3}}g\,g_{\ast}\,|v-\vet|^{\g}\left(\langle v\rangle^{s-2}-\langle \vet\rangle^{s-2}\right) |\vet|^{2}\d v\d\vet ,$$
and, using symmetry again, we get 
$$\mathscr{J}_{s,1}(g,g)=2s \int \int_{\R^{3}\times \R^{3}}g\,g_{\ast}\,|v-\vet|^{\g}\langle v\rangle^{s-2}\underset{=\langle \vet\rangle^{2}-\langle v\rangle^{2}}{\underbrace{\left[|\vet|^{2}-|v|^{2}\right]}}\d v\d\vet ,$$ 
which gives the new expression for $\mathscr{J}_{s,1}(g,g)$.\end{proof}

\begin{rmq}\label{rmq:negaJs1} According to Young's inequality, for $s >2$ one has
$\langle v\rangle^{s-2}\langle \vet\rangle^{2} \leq \frac{s-2}{s}\langle v\rangle^{s}+\frac{2}{s}\langle \vet\rangle^{s}$. Thus, 
\begin{multline*}
\int_{\R^{6}}g\,g_{\ast}|v-\vet|^{\gamma}\langle v\rangle^{s-2}\langle \vet\rangle^{2}\d v\d\vet \leq 
\frac{s-2}{s}\int_{\R^{6}}g\,g_{\ast}|v-\vet|^{\gamma}\langle v\rangle^{s}\d v\d\vet \\
+ \frac{2}{s}\int_{\R^{6}}g\,g_{\ast}|v-\vet|^{\gamma}\langle \vet\rangle^{s}\d v\d\vet=\int_{\R^{6}}g\,g_{\ast}|v-\vet|^{\gamma}\langle v\rangle^{s}\d v\d\vet ,\end{multline*}
where we used a simple symmetry argument for the last identity. In particular, one sees that
$$\mathscr{J}_{s,1}(g,g) \leq 0.$$
\end{rmq}

\noindent
Let us now estimate $\mathscr{J}_{s}(f,F)$. The basic observation is the following:

\begin{lem}\label{lem:Js1f}
For any $s \geq 0$, $f\ge 0$, $F=f\,(1 - \dd\,f)$,
\begin{multline}\label{Js1f}
\mathscr{J}_{s,1}(f,F)=\frac{1}{2}\Big(\mathscr{J}_{s,1}(F,F)+\mathscr{J}_{s,1}(f,f)\Big)\\
+\frac{\dd}{2}\Big(\mathscr{J}_{s,1}(f^{2},f)-\mathscr{J}_{s,1}(f,f^{2})\Big)
-\frac{1}{2}\dd^{2}\mathscr{J}_{s,1}(f^{2},f^{2}), 
\end{multline}
with 
$$\mathscr{J}_{s,1}(f^{2},f)-\mathscr{J}_{s,1}(f,f^{2})=2s\int \int_{\R^{3}\times\R^{3}}f^{2}(v)f(\vet) |v-\vet|^{\g+2}\left(\langle v\rangle^{s-2}-\langle \vet\rangle^{s-2}\right)\d v\d \vet.$$
\end{lem}

\begin{proof} This is proven by direct inspection, using that $F=f(1-\dd f)$, so that 
$$\mathscr{J}_{s,1}(f,F)=\mathscr{J}_{s,1}(f,f)-\dd \mathscr{J}_{s,1}(f,f^{2}), $$
and also
$$\mathscr{J}_{s,1}(f,F)=\mathscr{J}_{s,1}(F,F)+\dd \mathscr{J}_{s,1}(f^{2},F)=\mathscr{J}_{s,1}(F,F)+\dd \mathscr{J}_{s,1}(f^{2},f)-\dd^{2}\mathscr{J}_{s,1}(f^{2},f^{2}).$$
Taking the mean of these two identities gives \eqref{Js1f}. Now, write for simplicity
$$I :=\mathscr{J}_{s,1}(f^{2},f)-\mathscr{J}_{s,1}(f,f^{2}) .$$
One has
$$I=2s\int\int_{\R^{3}\times \R^{3}}f(v)f(\vet)\left(f(v)-f(\vet)\right)\,|v-\vet|^{\g}\left(\langle v\rangle^{s-2}-\langle \vet\rangle^{s-2}\right)\left(|\vet|^{2}-(v \cdot \vet)\right)\d v\d\vet , $$
from which we deduce, by a symmetry argument, that
\begin{equation*}
I=s \int\int_{\R^{3}\times\R^{3}}f(v)f(\vet)\left(f(v)-f(\vet)\right)|v-\vet|^{\g+2}\left(\langle v\rangle^{s-2}-\langle \vet\rangle^{s-2}\right)\d v\d \vet ,
\end{equation*}
which gives the desired expression using symmetry again.
\end{proof}
\noindent
We estimate separately the terms involved  in \eqref{Js1f} starting with the terms $\mathscr{J}_{s,1}(F,F)$ and $\mathscr{J}_{s,1}(f,f)$.

\begin{lem}\label{lem:T1} If $f_{\mathrm{in}}$ satisfies \eqref{hypci}--\eqref{eq:Mass}  for some $\dd_{0} >0$, for any $\dd \in (0,\dd_{0}]$ and any  $f \in \mathcal{Y}_{\dd}(f_{\mathrm{in}})$, it holds
\begin{equation}\label{eq:JsFF1}
\mathscr{J}_{s,1}(F,F)
\leq 2s\left(\|f_{\mathrm{in}}\|_{L^{1}_{2}}\int_{\R^{3}}F\langle v\rangle^{s-2}\d v 
- \eta_{\star}\int_{\R^{3}}F(v)\langle v\rangle^{s+\g}\d v\right)\,, \qquad \forall\, s >2,
\end{equation}
where $F=f(1-\dd\,f)$ and $\eta_{\star} >0$ is the constant in Lemma \ref{lem:jensen} which depends only on $\|f_{\rm in}\|_{L^{1}_{2}}$ and $H(f_{\mathrm{in}}).$  In the same way,
\begin{equation}\label{eq:Jsff1}
\mathscr{J}_{s,1}(f,f)
\leq 2s\left(\|f_{\mathrm{in}}\|_{L^{1}_{2}}\int_{\R^{3}}f\langle v\rangle^{s-2}\d v 
- \eta_{\star}\int_{\R^{3}}f(v)\langle v\rangle^{s+\g}\d v\right)\,, \qquad \forall\, s >2.
\end{equation}
\end{lem}
\begin{proof} For $f \in \mathcal{Y}_{\dd}(f_{\mathrm{in}})$ fixed, we recall that
$$\mathscr{J}_{s,1}(F,F):=2s\int_{\R^{6}}F\,F_{\ast}|v-\vet|^{\gamma}\langle v\rangle^{s-2}\left(\langle \vet\rangle^{2}-\langle v\rangle^{2}\right)\d v\d \vet ,$$
and replace, as in \cite{ALL}, $|v-\vet|^{\g}$ with its regularized version $\langle v-\vet\rangle^{\g}$. This gives
\begin{multline}\label{eq:Js}
\mathscr{J}_{s,1}(F,F)=2s\int_{\R^{6}}F\,F_{\ast}\langle v-\vet\rangle^{\gamma}\langle v\rangle^{s-2}\left(\langle \vet\rangle^{2}-\langle v\rangle^{2}\right)\d v\d \vet
\\
+2s\int_{\R^{6}}F\,F_{\ast}\left(|v-\vet|^{\gamma}-\langle v-\vet\rangle^{\g}\right)\langle v\rangle^{s-2}\left(\langle \vet\rangle^{2}-\langle v\rangle^{2}\right)\d v\d \vet.\end{multline}
Recall that $|v-\vet|^{\g}-\langle v-\vet\rangle^{\g} \geq0.$ Using H\"older's inequality with the measure $\d\mu(v,\vet)=F\,F_{\ast}\left(|v-\vet|^{\gamma}-\langle v-\vet\rangle^{\g}\right)\d v\d\vet$ and $p=\frac{s}{s-2},$ $q=\frac{s}{2}$ so that $1/p+1/q=1$, one gets
\begin{multline*}
\int_{\R^{6}}F\,F_{\ast}\left(|v-\vet|^{\gamma}-\langle v-\vet\rangle^{\g}\right)\langle v\rangle^{s-2} \langle \vet\rangle^{2}\d v\d \vet \\
\leq 
\left(\int_{\R^{6}}\langle v\rangle^{s}\d\mu(v,\vet)\right)^{\frac{s-2}{s}}\,\left(\int_{\R^{6}}\langle \vet\rangle^{s}\d\mu(v,\vet)\right)^{\frac{2}{s}} ,\end{multline*}
which, by symmetry, reads
$$\int_{\R^{6}}F\,F_{\ast}\left(|v-\vet|^{\gamma}-\langle v-\vet\rangle^{\g}\right)\langle v\rangle^{s-2} \langle \vet\rangle^{2}\d v\d \vet \leq  \int_{\R^{6}}\langle v\rangle^{s}\d\mu(v,\vet)\,.$$
Consequently, the second term in the right-hand side of \eqref{eq:Js} is nonpositive. Thus,
\begin{align*}
\mathscr{J}_{s,1}(F,F) &\leq 2s\int_{\R^{6}}F\,F_{\ast}\langle v-\vet\rangle^{\gamma}\langle v\rangle^{s-2}\left(\langle \vet\rangle^{2}-\langle v\rangle^{2}\right)\d v\d \vet\\
&=2s\int_{\R^{6}}F\,F_{\ast}\langle v-\vet\rangle^{\gamma}\langle v\rangle^{s-2}\langle \vet\rangle^{2}\d v\d\vet -2s\int_{\R^{6}}F\,F_{\ast}\langle v-\vet\rangle^{\gamma}\langle v\rangle^{s}\d v\d \vet.
\end{align*}
For $v \in \R^{3}$ fixed, one has
$$\int_{\R^{3}}F_{\ast}\langle v-\vet\rangle^{\gamma}\langle \vet\rangle^{2}\d\vet \leq \int_{\R^{3}}F_{\ast}\langle \vet\rangle^{2}\d\vet \leq  \|f_{\mathrm{in}}\|_{L^{1}_{2}}{=4}\,,$$
whereas, thanks to Lemma \ref{lem:jensen}, $\int_{\R^{3}}F_{\ast}\langle v-\vet\rangle^{\gamma}\d \vet \geq \eta_{\star}\langle v\rangle^{\g}.$
This easily gives \eqref{eq:JsFF1}. One proves the result in the same way for $\mathscr{J}_{s,1}(f,f)$, noticing that the above lower bound still holds if $f_{\ast}$ replaces $F_{\ast}$, since $f_{\ast}\geq F_{\ast}$.\end{proof}
One can evaluate the other terms in  \eqref{Js1f} as presented in the following lemma.

\begin{lem} Assume that $-2 < \g < 0$ and let $f_{\mathrm{in}}$ satisfy \eqref{hypci}--\eqref{eq:Mass} for some $\dd_0 >0$. Let $\dd \in (0,\dd_{0}]$ and $f \in \mathcal{Y}_{\dd}(f_{\mathrm{in}})$ be given. There is a positive constant $C_{0} >0$ depending only on  $\|f_{\rm in}\|_{L^1_2}$  such that, for any  $\delta \in (0,1)$  and any $s >2$, 
\begin{equation}\label{eq:ff22}
-\frac{\dd^{2}}{s}\mathscr{J}_{s,1}(f^{2},f^{2}) \leq  \delta\lD_{s+\g}(f) +C_{0}\delta^{\frac{\g}{2+\g}}\lM_{s+\g}(f),  \end{equation}
whereas,  
\begin{equation}\label{eq:diff}
\frac{\dd}{2}\left(\mathscr{J}_{s,1}(f^{2},f)-\mathscr{J}_{s,1}(f,f^{2})\right) \leq 2 s\dd\|f_{\mathrm{in}}\|_{L^{1}_{2}}\lM_{s+\g}(f) 
\leq 2 s\|f_{\mathrm{in}}\|_{L^{1}_{2}}\lm_{s+\g}(f).\end{equation}
\end{lem}
\begin{proof} Recall that
$$-\dd^{2}\mathscr{J}_{s,1}(f^{2},f^{2})=2s\dd^{2}\int_{\R^{6}}f^{2}f^{2}_{\ast}|v-\vet|^{\g}\langle v\rangle^{s-2}\left(\langle v\rangle^{2}-\langle \vet\rangle^{2}\right)\d v\d\vet.$$
Neglecting the negative term and using that $\dd\,f^{2}_{\ast} \leq f_{\ast}$, we obtain
\begin{equation}\label{eq:Js1f2f2}
-\dd^{2}\mathscr{J}_{s,1}(f^{2},f^{2}) \leq 2s\dd\int_{\R^{6}}f^{2}(v)f_{\ast}|v-\vet|^{\g}\langle v\rangle^{s}\d v\d\vet.\end{equation}
Inequality \eqref{eq:ff22} is obtained using Proposition \ref{prop:GG}  
with $g=f_{\ast}$ and $\phi(v)=\langle v\rangle^{\frac{s}{2}}f(v)$ and noticing that, for $\delta \in (0,1)$, $1+\delta^{\frac{\g}{2+\g}} \leq 2\delta^{\frac{\g}{2+\g}}$.  The proof of \eqref{eq:diff} is obvious since (thanks to Lemma \ref{lem:Js1f})
\begin{multline*}
\frac{\dd}{2}\left(\mathscr{J}_{s,1}(f^{2},f)-\mathscr{J}_{s,1}(f,f^{2})\right)\\
=\dd s\int_{\R^{3}\times\R^{3}}f^{2}(v)f(\vet) |v-\vet|^{\g+2}\left(\langle v\rangle^{s-2}-\langle \vet\rangle^{s-2}\right)\d v\d \vet\\
\leq \dd s\int_{\R^{3}\times\R^{3}}f^{2}(v)f(\vet) |v-\vet|^{\g+2} \langle v\rangle^{s-2}\d v\d \vet\\
\leq 2\dd s \int_{\R^{3}}f^{2}(v)\langle v\rangle^{\g+s}\d v\int_{\R^{3}}f(\vet)\langle \vet\rangle^{\g+2}\d\vet ,
\end{multline*}
where we use that, since  $\g +2 \in (0,2)$, we have $|v-\vet|^{\g+2}\leq 2\langle v\rangle^{\g+2}\langle \vet\rangle^{\g+2}$. This proves inequality \eqref{eq:diff} where the last inequality obviously comes from $\dd f^{2} \leq f.$
\end{proof}
\noindent
Let us now investigate the second term $\mathscr{J}_{s,2}(f,F)$ in the right-hand side of \eqref{eq:mom-s}.

\begin{lem}\label{lem:T2}
Assume that $-2 <\g < 0.$ Let $f_{\mathrm{in}}$ satisfy \eqref{hypci}--\eqref{eq:Mass}  for some $\dd_0 >0$,  $\dd \in (0,\dd_{0}]$ and $f \in \mathcal{Y}_{\dd}(f_{\rm in})$, $F= f \,(1 - \dd f)$ be given.  Then  for any $s >2$,
\begin{equation}\label{eq:Js2}
\mathscr{J}_{s,2}(f,F) \leq {6s(s-2)}\int_{\R^{3}}f\langle v\rangle^{s-2}\d v\, .
\end{equation}
\end{lem}

\begin{proof} Let $f \in \mathcal{Y}_{\dd}(f_{\mathrm{in}})$ be fixed. Recall that
$$\mathscr{J}_{s,2}(f,F):=s(s-2)\int_{\R^{6}}f\,F_{\ast}|v-\vet|^{\gamma}\langle v\rangle^{s-4}\left(|v|^{2}\,|\vet|^{2}-(v\cdot \vet)^{2}\right)\d v\d\vet ,$$ 
and split the integral according to $|v-\vet| < 1$ and $|v-\vet| \ge1.$ Since $$0 \leq |v|^{2}\,|\vet|^{2}-(v\cdot \vet)^{2} \leq |v|^{2}|\vet|^{2}\,,$$ one sees that
\begin{multline*}
\int_{|v-\vet|\geq1}f\,F_{\ast}|v-\vet|^{\gamma}\langle v\rangle^{s-4}\left(|v|^{2}\,|\vet|^{2}-(v\cdot \vet)^{2}\right)\d v\d\vet\\
\leq \int_{\R^{6}}f\,F_{\ast}\langle v\rangle^{s-4}|v|^{2}|\vet|^{2}\d v\d\vet
\leq \|f_{\rm in}\|_{L^{1}_{2}} \int_{\R^{3}}f\langle v\rangle^{s-2}\d v{=4\int_{\R^{3}}f\langle v\rangle^{s-2}\d v}.\end{multline*}
For the integral on the set $|v-\vet| <1$, one uses that $|v|^{2}\,|\vet|^{2}-(v\cdot \vet)^{2} \leq |v|\,|\vet|\,|v-\vet|^{2}$  to get
\begin{multline*}
\int\int_{|v-\vet|< 1}f\,F_{\ast}|v-\vet|^{\gamma}\langle v\rangle^{s-4}\left(|v|^{2}\,|\vet|^{2}-(v\cdot \vet)^{2}\right)\d v\d\vet\\
\leq \int\int_{|v-\vet|<1}f\,F_{\ast}|v-\vet|^{\g+2}\langle v\rangle^{s-4}|v|\,|\vet|\d v\d\vet
\leq \int_{\R^{3}}f\langle v\rangle^{s-3}\d v \int_{\R^{3}}F_{\ast}|\vet|\d\vet\,,
\end{multline*}
where we used that $\g+2 \geq0$ to deduce $|v-\vet|^{\g+2} \leq 1.$
Since, from Young's inequality,
$$\int_{\R^{3}}F_{\ast}|\vet|\d\vet \leq \frac{1}{2}\int_{\R^{3}}F_{\ast}\langle \vet\rangle^{2}\d \vet \leq \frac{1}{2}\|f_{\mathrm{in}}\|_{L^{1}_{2}}{=2}\,,$$
one deduces the result.
\end{proof}
\noindent
We apply the previous results to solutions $f(t,v)$ to \eqref{LFD} to obtain the following proposition.

\begin{prop}\label{lem:ms1}
Assume that $-2 < \g < 0$  and let a nonnegative initial datum $f_{\mathrm{in}}$ satisfying \eqref{hypci}--\eqref{eq:Mass} for some $\dd_0 >0$ be given. For $\dd \in (0,\dd_0]$, let  $f(t,\cdot)$ be a weak-solution to \eqref{LFD}. Then, there is a positive constant $C >0$ depending only on $\|f_{\mathrm{in}}\|_{L^{1}_{2}}$, such that, for any $s >2$ and $\delta \in (0,1)$, there are positive constant $\bm{K}_{s}$ which depend on $s$ and $H(f_{\rm in})$ and $\|f_{\mathrm{in}}\|_{L^{1}_{2}}$  satisfying
\begin{equation}\label{eq:mom-s-1}
\dfrac{\d}{\d t}\lm_{s}(t) +s\frac{\eta_{\star}}{4}\lm_{s+\g}(t) \leq 2s\bm{K}_{s} + {\frac{s\delta}{2}} \lD_{s+\g}(t) + C\,s \delta^{\frac{\g}{2+\g}}\lM_{s+\g}(t) .\end{equation}
Moreover, there exists $\beta> 0$ depending only on  $H(f_{\rm in})$ and $\|f_{\mathrm{in}}\|_{L^{1}_{2}}$   such that, for  $s\ge 3$,
\begin{equation}\label{rmq:Ks}
  \bm{K}_{s} \leq {\beta\left(\beta(s-2)\right)^{\frac{s-2}{\g+2}}\left(\frac{s-2}{s+\g}\right)^{\frac{s+\g}{\g+2}} \leq \beta\left(\beta\,s\right)^{\frac{s-2}{\g+2}}}.\end{equation}
\end{prop}

\begin{proof} According to \eqref{eq:mom-s} and \eqref{Js1f}, one has
\begin{multline*}
\dfrac{\d}{\d t}\lm_{s}(t)=\tfrac{1}{2}\Big(\mathscr{J}_{s,1}(F,F)+\mathscr{J}_{s,1}(f,f)\Big)\\+\frac{\dd}{2}\Big(\mathscr{J}_{s,1}(f^{2},f)-\mathscr{J}_{s,1}(f,f^{2})\Big)
-\frac{\dd^{2}}{2}\mathscr{J}_{s,1}(f^{2},f^{2})+\mathscr{J}_{s,2}(f,F),
\end{multline*}
with $f=f(t,v)$ and $F=f(1-\dd f).$ One sees from \eqref{eq:JsFF1}--\eqref{eq:Jsff1} that 
$$\tfrac{1}{2}\left(\mathscr{J}_{s,1}(F,F)+\mathscr{J}_{s,1}(f,f)\right) \leq s\left(\|f_{\mathrm{in}}\|_{L^{1}_{2}}\int_{\R^{3}}(f+F)\langle v\rangle^{s-2}\d v -\eta_{\star}\int_{\R^{3}}(F+f)\langle v\rangle^{s+\g}\d v\right), $$
whereas, from \eqref{eq:ff22} and \eqref{eq:diff},
\begin{align*}
\frac{\dd}{2}\left(\mathscr{J}_{s,1}(f^{2},f)-\mathscr{J}_{s,1}(f,f^{2})\right) & \leq 2 s\dd\|f_{\mathrm{in}}\|_{L^{1}_{2}}\lM_{s+\g}(t)\leq 8 s\lM_{s+\g}(t),\\
-\frac{\dd^{2}}{2}\mathscr{J}_{s,1}(f^{2},f^{2})&  \leq  \frac{s}{2}\left(\delta\lD_{s+\g}(t) +C_{0}\delta^{\frac{\g}{2+\g}}\lM_{s+\g}(t)\right),
\end{align*}
for any $\delta \in (0,1)$. Using then \eqref{eq:Js2} to estimate $\mathscr{J}_{s,2}(f,F)$, we deduce that
\begin{multline*}
\dfrac{\d}{\d t}\lm_{s}(t) \leq s\left(4\int_{\R^{3}}(f+F)\langle v\rangle^{s-2}\d v -\eta_{\star}\int_{\R^{3}}(F+f)\langle v\rangle^{s+\g}\d v\right)\\
+\frac{s}{2}\left(\delta\,\lD_{s+\g}(t) +\left(C_{0}\delta^{\frac{\g}{2+\g}}+16\right)\lM_{s+\g}(t)\right) +6s(s-2) \lm_{s-2}(t).
\end{multline*}
Since $\g +2 >0$, the mapping $v \in \R^{3} \mapsto 4\langle v\rangle^{s-2}-\frac{1}{2}\eta_{\star}\langle v\rangle^{s+\g}$ is bounded by some positive constant $\bm{K}_{s} >0$ which depends on $f_{\mathrm{in}}$ through $\eta_{\star}$. Thus, we deduce that
\begin{multline*}
\dfrac{\d}{\d t}\lm_{s}(t)+\frac{s\,\eta_{\star}}{2}\int_{\R^{3}}(f+F)\langle v\rangle^{s+\g}\d v \\
\leq s\bm{K}_{s} + {\frac{s\delta}{2}}\,\lD_{s+\g}(t) + s\bar{C}\left(\delta^{\frac{\g}{2+\g}}+1\right)\lM_{s+\g}(t) + 6s(s-2)\lm_{s-2}(t),
\end{multline*}
with $\bar{C}=\max\left(\frac{C_{0}}{2},8\right).$ Again, since $\g >-2$, up to a modification of $\bm{K}_{s}$, we have $6(s-2) \lm_{s-2}(t) \leq \bm{K}_{s}+\frac{\eta_{\star}}{4}\lm_{s+\g}(t)$, from which we easily deduce \eqref{eq:mom-s-1}. Let us now explicit $\bm{K}_{s}$.  One observes from the aforementioned computations that one can take $\bm{K}_{s}=\max(\sup_{x>0}u_{s}(x),\sup_{x>0}w_{s}(x))$, where
$$u_{s}(x):=4 x^{s-2}-\frac{\eta_{\star}}{2}x^{s+\g}\,,  \qquad 
w_{s}(x):=  6(s-2)x^{s-2}-\frac{\eta_{\star}}{4}x^{s+\g}, \qquad x >0.$$
It is clear that $\sup_{x >0}u_{s}(x)=u_{s}(\bar{x})$  and $\sup_{x >0}w_{s}(x)=w_{s}(\tilde{x})$, where
$$\bar{x}=\left(\frac{8(s-2)}{\eta_{\star}(s+\g)}\right)^{\frac{1}{2+\g}}, \qquad \tilde{x}=\left(\frac{24(s-2)^{2}}{\eta_{\star}(s+\g)}\right)^{\frac{1}{2+\g}},$$
and consequently, $\sup_{x>0}u_{s}(x)=4\bar{x}^{s-2}\frac{\g+2}{s+\g},$ $\sup_{x >0}w_{s}(x)=6(s-2)\tilde{x}^{s-2}\frac{\g+2}{s+\g}.$
Therefore, for any $s \geq 3$, we see that $\bm{K}_{s}=\sup_{x>0}w_{s}(x)$, and one checks that \eqref{rmq:Ks} holds for some explicit $\beta >0$.
\end{proof}

\subsection{$L^{2}$-estimates}
 
We now aim to  study the evolution of weighted $L^{2}$-norms of $f(t,v)$.  Keeping previous notations, we have the lemma.
\begin{lem}\label{lem:L2-Ms} Assume that $-2 < \g < 0$  and let a nonnegative initial datum $f_{\mathrm{in}}$ satisfying \eqref{hypci}--\eqref{eq:Mass} for some $\dd_0 >0$ be given. For $\dd \in (0,\dd_0]$, let  $f(t,\cdot)$ be a weak-solution to \eqref{LFD}. For any $s \geq0$, it holds 

\begin{multline}\label{eq:dtMs}
\frac{1}{2}\dfrac{\d}{\d t}\lM_{s}(t) + \frac{K_{0}}{2}\lD_{s+\g}(t) \leq  2(\g+3)\int_{\R^{3}}\langle v\rangle^{s}\left( \frac{1}{2} f^2 - \frac{\dd}{3} f^3 \right)\,\left(|\cdot|^{\g}\ast f\right)\d v\\
+s \int_{\R^{3}}\langle v\rangle^{s-2}\left(f^2 - \frac{2\dd}{3} f^3 \right)\left(\bm{b}[f]\cdot v\right)\d v  + K_{0}(s+\g)^{2}\int_{\R^{3}}\langle v\rangle^{s+\g-2}f^{2}(v)\d v
\\
{-\frac{s\dd}{2}\int_{\R^{3}}\langle v\rangle^{s-2}f^{2}\bm{b}[f^{2}]\cdot v\d v}+\frac{s}{2}\int_{\R^{3}}\langle v\rangle^{s-4}f^{2}\,\, \mathrm{Trace}\left(\bm{\Sigma}[f]\cdot \bm{A}(v)\right)\d v\,,
\end{multline}
where $\bm{A}(v)=\langle v\rangle^{2}\mathbf{Id}+(s-2)\,v\otimes v$, $v \in \R^{3}$ and $K_{0}$ is defined in Prop. \ref{diffusion}.
\end{lem}

\begin{proof} As in \cite{ABL}, for any $s \geq 0$,
 \begin{multline*}
\frac{1}{2}\frac{\d}{\d t} \int_{\R^3} f^2(t,v) \langle v\rangle^{s} \d v 
=  -\int_{\R^3}  \langle v\rangle^{s}(\bm{\Sigma}[f] \grad f) \cdot \grad f  \d v
-  \, s \int_{\R^3} f  \langle v\rangle^{s-2}(\bm{\Sigma}[f] \grad f)\cdot v\, \d v \\
+ \int_{\R^3}\langle v\rangle^{s} f(1-\dd f) \bm{b}[f] \cdot \grad f  \d v
+ \,s \int_{\R^3} (\bm{b}[f] \cdot v )\, f^2(1-\dd f)\langle v\rangle^{s-2} \d v.
\end{multline*}
Using the uniform ellipticity of the diffusion matrix $\bm{\Sigma}[f]$ (recall Proposition \ref{diffusion}), we deduce that
$$\int_{\R^3}  \langle v\rangle^{s}(\bm{\Sigma}[f] \grad f) \cdot \grad f  \d v
 \geq K_{0} \int_{\R^{3}}\langle v\rangle^{s+\g}\,\left|\grad f\right|^{2}\d v\,.$$
Moreover, writing
\begin{equation*}
\nabla \left(\langle v\rangle^{\frac{s+\g}{2}}\,f\right)=\langle v\rangle^{\frac{s+\g}{2}}\nabla f + \frac{s+\g}{2}v\,\langle v\rangle^{\frac{s+\g}{2}-2}f\,,
\end{equation*}
from which
\begin{equation}\label{eq:Gradient}
\langle v\rangle^{s+\g}\left|\nabla f\right|^{2} \geq \frac{1}{2}\left|\nabla \left(\langle v\rangle^{\frac{s+\g}{2}}f\right)\right|^{2} - (s+\g)^{2} \langle v\rangle^{s+\g-2}f^{2}(v), \end{equation}
we also have
\begin{multline*}
\int_{\R^3}\langle v\rangle^{s} f(1-\dd f) \bm{b}[f] \cdot \grad f  \d v
 =  - \int_{\R^3} \left( \frac{1}{2} f^2 - \frac{\dd}{3} f^3 \right)
\grad \cdot \Big(\bm{b}[f] \langle v\rangle^{s}\Big) \d v\\
=-s\int_{\R^{3}}\langle v\rangle^{s-2}\left( \frac{1}{2} f^2 - \frac{\dd}{3} f^3 \right) \bm{b}[f]\cdot v\,\d v
- \int_{\R^{3}}\langle v\rangle^{s}\left( \frac{1}{2} f^2 - \frac{\dd}{3} f^3 \right)\nabla \cdot \bm{b}[f]\, \d v.
\end{multline*}
Therefore, recalling that $\nabla \cdot \bm{b}[f]=\bm{c}_{\g}[f]=-2(\g+3)|\cdot|^{\g}\ast f$, we get 
\begin{multline*}
\frac{1}{2}\dfrac{\d}{\d t}\lM_{s}(t) + \frac{K_{0}}{2}\lD_{s+\g}(t) \leq  2(\g+3)\int_{\R^{3}}\langle v\rangle^{s}\left( \frac{1}{2} f^2 - \frac{\dd}{3} f^3 \right)\,\left(|\cdot|^{\g}\ast f\right)\d v\\
+s \int_{\R^{3}}\langle v\rangle^{s-2}\left( \frac{1}{2} f^2 - \frac{2\dd}{3} f^3 \right)\left(\bm{b}[f]\cdot v\right)\d v  \\
+  K_{0}(s+\g)^{2} \int_{\R^{3}}\langle v\rangle^{s+\g-2}f^{2}(v)\d v
-s\int_{\R^{3}}\langle v\rangle^{s-2}f \left(\bm{\Sigma}[f]\nabla f\cdot v \right)\d v.
\end{multline*}
Let us investigate more carefully the last term. Integration by parts shows that \begin{align*}
-s\int_{\R^{3}}\langle v\rangle^{s-2}f \left(\bm{\Sigma}[f]\nabla f\cdot v \right)\d v &= -\frac{s}{2}\int_{\R^{3}}\nabla f^{2} \cdot \Big(\bm{\Sigma}[f]
\, \langle v\rangle^{s-2}v \Big) \,\d v\\
&= \frac{s}{2}\int_{\R^{3}} f^{2} \; \nabla\cdot \Big(\bm{\Sigma}[f]\,\langle v\rangle^{s-2} v \Big) \,\d v\,.
\end{align*}
Using the product rule 
$$\nabla\cdot \Big(\bm{\Sigma}[f]\,\langle v\rangle^{s-2} v \Big)=\langle v\rangle^{s-2}\,{\bm{B}}[f]\cdot v\, + \mathrm{Trace}\left(\bm{\Sigma}[f] \cdot \,
D_{v}\left(\langle v\rangle^{s-2} v\right)\right), $$ 
where $D_{v}\big( \langle v\rangle^{s-2} v\big)$ is the matrix with entries $\partial_{v_{i}}\big( \langle v\rangle^{s-2} v_{j}\big)$, $i,j=1,2,3$, or more compactly, 
$$D_{v}\big( \langle v\rangle^{s-2} v\big)=\langle v\rangle^{s-4}\bm{A}(v)\,,$$
one gets the desired inequality, recalling that $\bm{B}[f]=\bm{b}[f]-\dd\,\bm{b}[f^{2}].$
\end{proof}
We deduce from the previous arguments the following proposition.

\begin{prop}\label{prop:L2}
Assume that $-2 < \g < 0$ and let a nonnegative initial datum $f_{\mathrm{in}}$ satisfying \eqref{hypci}--\eqref{eq:Mass} for some $\dd_0 >0$ be given. For $\dd \in (0,\dd_0]$, let  $f(t,\cdot)$ be a weak-solution to \eqref{LFD}. There exists some positive constant $\bar{C}(f_{\mathrm{in}})$ depending on  $\|f_{\mathrm{in}}\|_{L^{1}_{2}}$ and $H(f_{\rm in})$,  such that
\begin{equation}\label{final-L2}
  \frac{1}{2}\dfrac{\d}{\d t}\lM_{s}(t) +  {\frac{K_{0}}{8}}\lD_{s+\g}(t) \leq \bar{C}(f_{\mathrm{in}}) {\left(1+s^{\frac{10}{2+\g}}\right)}\|\langle \cdot \rangle^{\frac{\gamma+s}{2}}f\|^{2}_{L^{1}}\,
\end{equation} holds for any $s \geq0.$
\end{prop} 
\begin{proof} We denote by $I_{1},I_{2},I_{3},I_{4},I_{5}$ the various terms on the right-hand-side of \eqref{eq:dtMs}, i.e.
$$\frac{1}{2}\dfrac{\d}{\d t}\lM_{s}(t) + \frac{K_{0}}{2}\lD_{s+\g}(t) \leq  I_{1}+I_{2}+I_{3}+I_{4}+I_{5}, $$
and we control each term starting from $I_{1}$. Since $0 \leq \frac{1}{6}f^{2} \leq \frac{1}{2}f^{2}-\frac{\dd}{3}f^{3} \leq \frac{1}{2}f^{2}$, one has
$$|I_{1}|\leq (\g+3) \int\int_{\R^{3}\times\R^{3}}|v-\vet|^{\g}f^{2}(t,v)\langle v\rangle^{s}f(t,\vet)\d \vet \d v, $$
so that, using Proposition \ref{prop:GG} with $g=f(t)$ and $\phi^{2}=\langle \cdot \rangle^{s}f^{2}(t)$, we deduce that, for any $\delta \in (0,1)$,
$$|I_{1}| \leq \delta\,\lD_{s+\g}(t)+C_{1}{\delta^{\frac{\g}{2+\g}}}\lM_{s+\g}(t)\,,$$
where $C_{1}$ depends on $\|f_{\rm in}\|_{L^{1}_{2}}$. For the term $I_{2}$, since $0\leq \frac{1}{3}f^{2} \leq f^{2}-\frac{2\dd}{3}f^{3} \leq  f^{2}$, it  holds that
$$|I_{2}| \leq s\int_{\R^{3}}\langle v\rangle^{s-1}f^{2}(t,v)\,|\bm{b}[f(t)](v)|\d v \leq 2s\int_{\R^{6}}\langle v\rangle^{s-1}f^{2}(t,v)|v-\vet|^{\g+1}f(t,\vet)\d \vet\d v\,.$$
Therefore, if $\g+1 <0$, applying Proposition \ref{prop:GG} with  {$\bm{c}_{\g+1}[g]$ instead of $\bm{c}_{\g}[g]$}, and $g=f(t,v)$, $\phi^{2}=\langle \cdot\rangle^{s-1}f^{2}(t)$, we get
$$|I_{2}| \leq s\left(\delta\,\lD_{s+\g}(t) + C_{1}{\left(1+\delta^{ {\frac{\g+1}{3+\g}}}\right)}\,\lM_{s+\g}(t)\right),$$
 whereas, if $\g+1 >0$, one has obviously $|I_{2}| \leq s\|\langle \cdot\rangle^{\gamma+1}f(t)\|_{L^{1}}\lM_{s+\g}(t).$
In both cases, for any  {$\delta >0$}, 
$$|I_{2}| \leq s\left(\delta\,\lD_{s+\g}(t) + C_{1}\left(1+\delta^{\frac{\g+1}{3+\g}}\right)\,\lM_{s+\g}(t)\right).$$
In the same way,
$$|I_{4}| \leq \frac{\dd s}{2}\int_{\R^{3}}\langle v\rangle^{s-1}f^{2}(t,v)\,|\bm{b}[f^{2}(t, \cdot)](v)|\d v \leq\frac{s}{2}\int_{\R^{3}}\langle v\rangle^{s-1}f^{2}(t,v)\,|\bm{b}[f(t)](v)|\d v, $$
since $\dd f^{2}\leq f.$ Then, as before, for any $\delta >0$, there is $C_{1} >0$ such that
$$|I_{4}| \leq s\left(\delta\,\lD_{s+\g}(t) + C_{1} {\left(1+\delta^{\frac{\g+1}{3+\g}}\right)}\,\lM_{s+\g}(t)\right).$$
For the term $I_{5}$, one checks easily that, for any $i,j \in \{1,2,3\}$,
$$\left|\bm{\Sigma}_{i,j}[f]\right| \leq 2|\cdot|^{\g+2}\ast f, \qquad \left|\bm{A}_{i,j}(v)\right| \leq s\langle v\rangle^{2}, $$
and 
$$|I_{5}| \leq 9s^{2}\int_{\R^{6}}\langle v\rangle^{s-2}f^{2}(t,v)|v-\vet|^{\g+2}f(t,\vet)\d v\d\vet.$$
One has, since $\g+2 >0$,
$$|I_{5}| \leq s^{2}\,C\|\langle \cdot \rangle^{\g+2}f(t)\|_{L^{1}}\,\|\langle\cdot\rangle^{\g+s}f^{2}(t)\|_{L^{1}}=s^{2}\,C\|\langle \cdot \rangle^{\g+2}f(t)\|_{L^{1}}\,\lM_{s+\g}(t).$$
Finally, it is easy to see that $|I_{3}| \leq K_{0}(s+\g)^{2}\,\lM_{s+\g}(t).$
Overall, recalling mass and energy conservation to estimate all the weighted $L^{1}$-terms, one sees that, for any ${\delta \in (0,1)}$,  there is some positive constant $C(f_{\mathrm{in}})$ depending on  $\|f_{\mathrm{in}}\|_{L^{1}_{2}}$ and $H(f_{\rm in})$ (through $K_{0}$) such that
\begin{equation}\label{eq:Ms+g0}\frac{1}{2}\dfrac{\d}{\d t}\lM_{s}(t) + \frac{K_{0}}{2}\lD_{s+\g}(t) 
\leq C(f_{\mathrm{in}})\left(s^{2}+\delta^{\frac{\g}{\g+2}}+s+s\delta^{\frac{\g+1}{\g+3}}\right)\lM_{s+\g}(t) +  {(2s+1)}\delta\lD_{s+\g}(t)\,.\end{equation}
{For $s \in [0,1]$, \eqref{eq:Ms+g0} can be rephrased simply as 
$$\frac{1}{2}\dfrac{\d}{\d t}\lM_{s}(t) + \frac{K_{0}}{2}\lD_{s+\g}(t) 
\leq C(f_{\mathrm{in}})\left(\delta^{\frac{\g}{\g+2}}+\delta^{\frac{\g+1}{\g+3}}\right)\lM_{s+\g}(t)+3\delta\lD_{s+\g}(t)$$
and, picking $\delta \in (0,1)$ such that $3\delta \leq \frac{K_{0}}{4}$, one deduces that
\begin{equation}\label{eq:Ms+gs0}
\frac{1}{2}\dfrac{\d}{\d t}\lM_{s}(t) + \frac{K_{0}}{4}\lD_{s+\g}(t) 
\leq \widetilde{C}_{\g}(f_{\mathrm{in}})\lM_{s+\g}(t), \qquad s \in [0,1]\,,\end{equation}
for some positive constant $\widetilde{C}_{\g}(f_{\rm in})$ depending only on $\|f_{\rm in}\|_{L^{1}_{2}}$, $H(f_{\rm in})$ and $\g$. For $s >1$, since $2s+1 \leq 3s$, choosing $\delta:=\min\left(\frac{K_{0}}{16s},1\right)$ we deduce from \eqref{eq:Ms+g0} that there is 
$C_{\g}(f_{\mathrm{in}})$ depending only on $\|f_{\mathrm{in}}\|_{L^{1}_{2}}$, $H(f_{\rm in})$ and $\g >0$ such that
\begin{equation}\label{eq:Ms+g}
\frac{1}{2}\frac{\d}{\d t}\lM_{s}(t) + \frac{K_{0}}{4}\lD_{s+\g}(t) \leq C_{\g}(f_{\mathrm{in}})\,\left(s^{2}+{s^{-\frac{\g}{2+\g}}}{+s+s^{\frac{2}{\g+3}}}\right)\lM_{s+\g}(t), \qquad t \geq0.\end{equation}}
From Nash inequality, there is some universal constant $C >0$ such that
$$\lM_{s+\g}(t)=\|\langle \cdot \rangle^{\frac{s+\g}{2}}f(t)\|_{L^{2}}^{2} \leq C\,\left\|\langle \cdot \rangle^{\frac{s+\g}{2}}f(t)\right\|_{L^{1}}^{\frac{4}{5}}\,\left\|\nabla \left(\langle \cdot \rangle^{\frac{s+\g}{2}}f(t)\right)\right\|_{L^{2}}^{\frac{6}{5}}, $$
which, thanks to Young's inequality, implies that there is $C   >0$ such that, for any $\alpha >0$,  
\begin{equation}\label{eq:MsgYo}
\,\lM_{s+\g}(t) \leq C {\alpha}^{-\frac{3}{2}} \|\langle \cdot \rangle^{\frac{\gamma+s}{2}}f(t)\|^{2}_{L^{1}}+\alpha\lD_{s+\g}(t).\end{equation}
{Choosing now $\alpha >0$ such that $\widetilde{C}_{\g}(f_{\rm in})\alpha=\frac{K_{0}}{8}$ if $s \in [0,1]$ or $C_{\g}(f_{\mathrm{in}})\,\left(s^{2}+{s^{-\frac{\g}{2+\g}}}{+s+s^{\frac{2}{\g+3}}}\right)\alpha= \frac{K_{0}}{8}$ whenever $s >1$,} we end up with
$$\frac{1}{2}\dfrac{\d}{\d t}\lM_{s}(t) +  \frac{K_{0}}{8}\lD_{s+\g}(t) \leq C_{s}(f_{\mathrm{in}})\|\langle \cdot \rangle^{\frac{\gamma+s}{2}}f\|^{2}_{L^{1}}\,$$
where, according to estimate \eqref{eq:Ms+g} and \eqref{eq:MsgYo},
$${C_{s}(f_{\rm in})=C\alpha^{-\frac{3}{2}}C_{\g}(f_{\mathrm{in}})\,\left(s^{2}+{s^{-\frac{\g}{2+\g}}}{+s+s^{\frac{2}{\g+3}}}\right)},\qquad {s >1}$$
and the last choice of $\alpha=\frac{K_{0}}{8C_{\g}(f_{\rm in})}\left(s^{2}+{s^{-\frac{\g}{2+\g}}}{+s+s^{\frac{2}{\g+3}}}\right)^{-1}$ gives that
$$C_{s}(f_{\mathrm{in}})=C(f_{\mathrm{in}})\left(s^{2}+{s^{-\frac{\g}{2+\g}}}{+s+s^{\frac{2}{\g+3}}}\right)^{{\frac{5}{2}}} \leq \bar{C}(f_{\mathrm{in}}){\left(1+s^{\frac{10}{2+\g}}\right)}, $$
since $\max(1,2,\frac{-\g}{2+\g},\frac{2}{3+\g}) \leq \frac{4}{2+\g}$ for any $-2< \g<0$ and with $C(f_{\mathrm{in}})$ and $\bar{C}(f_{\mathrm{in}})$ depending only on $\|f_{\rm in}\|_{L^{1}_{2}}$ and $H(f_{\mathrm{in}})$ but not on $s$.  This shows \eqref{final-L2}.
\end{proof}

\begin{cor}\label{cor:L2Ms}
Assume that $-2 < \g < 0$  and let a nonnegative initial datum $f_{\mathrm{in}}$ satisfying \eqref{hypci}--\eqref{eq:Mass} for some $\dd_0 >0$ be given. For $\dd \in (0,\dd_0]$, let  $f(t,\cdot)$ be a weak-solution to \eqref{LFD}.  Given $s \in [0,4+|\g|]$ there exists some positive constant $C(f_{\mathrm{in}})$ depending  on $f_{\mathrm{in}}$ only through $\|f_{\mathrm{in}}\|_{L^{1}_{2}}$,  $H(f_{\rm in})$, and such that
$$\lM_{s}(t_{2}) + {\frac{K_{0}}{4}}\int^{t_{2}}_{t_{1}}\lD_{s+\g}(\tau)\d \tau \leq \lM_{s}(t_{1}) + C(f_{\mathrm{in}})\,(t_{2}-t_{1})\,,$$
holds for any $0 \leq t_{1} < t_{2}$.
\end{cor}
\begin{proof} When $\frac{\gamma+s}{2}\leq2$, it holds that $\|\langle \cdot \rangle^{\frac{\gamma+s}{2}}f\|_{L^1}\leq \|f_{{\rm in}}\|_{L^{1}_{2}}$,
 which gives the statement after integration of \eqref{final-L2}.\end{proof}

\subsection{Short-time estimates and appearance of $L^{2}$-moments}

Before trying to get global-in-time estimates for both $\lm_{s}(t)$ and $\lM_{s}(t)$ (with a growing rate independent of $s$), let us start with the following non-optimal growth that has to be interpreted here as a short-time estimate.
\begin{prop}\label{shortime}
Assume that $-2 < \g < 0$  and let a nonnegative initial datum $f_{\mathrm{in}}$ satisfying \eqref{hypci}--\eqref{eq:Mass} for some $\dd_0 >0$ be given. For $\dd \in (0,\dd_0]$, let  $f(t,\cdot)$ be a weak-solution to \eqref{LFD}.  Then, for any  $s \geq 3$,
\begin{equation}\label{eq:momlms-short}
\lm_{s}(t) \leq \Big[ \lm_{s}(0)^{\frac{|\gamma|}{s}} +  C(f_{\rm in})\,|\gamma|\,s \, t\Big]^{\frac{s}{|\gamma|}}\,,\qquad t\geq0\,,
\end{equation}
where the constant $C(f_{\mathrm{in}})$ depends on $f_{\mathrm{in}}$ only through $\|f_{\mathrm{in}}\|_{L^{1}_{2}}$ but does not depend on $s$.  If $s \in (2,3)$ \eqref{eq:momlms-short} still  holds for $\g \in [-1,0)$ whereas, for $\g \in (-2,-1)$,
\begin{equation}\label{eq:momlms-short0}
\lm_{s}(t) \leq \lm_{s}(0) +  C(f_{\rm in}) \, t\,,\qquad t\geq0\,,
\end{equation}
for a constant $C(f_{\rm in})$ depending only on $\|f_{\rm in}\|_{L^1_2}$.
\end{prop}
\begin{proof}
Recall that, according to \eqref{eq:mom-s}, $\dfrac{\d}{\d t}\lm_{s}(t)=\mathscr{J}_{s,1}(f,F)+\mathscr{J}_{s,2}(f,F)$
where, for any $s \geq 2$,
\begin{multline*}
\mathscr{J}_{s,1}(f,F)=\frac{1}{2}\Big(\mathscr{J}_{s,1}(F,F)+\mathscr{J}_{s,1}(f,f)\Big)\\
+\frac{\dd}{2}\Big(\mathscr{J}_{s,1}(f^{2},f)-\mathscr{J}_{s,1}(f,f^{2})\Big) -\frac{\dd^{2}}{2}\mathscr{J}_{s,1}(f^{2},f^{2})\,.
\end{multline*}
Recall (see Remark \ref{rmq:negaJs1}) that 
\begin{equation}\label{coulomb:e-1}
\mathscr{J}_{s,1}(F,F)\leq0,\,\,\qquad\mathscr{J}_{s,1}(f,f)\leq 0\,.
\end{equation}
We neglect such absorption terms since we are interested here in the short time propagation of moments. We also recall that, according to \eqref{eq:diff},
\begin{equation}\label{coulomb:e1}
\mathscr{J}_{s,1}(f^{2},f)-\mathscr{J}_{s,1}(f,f^{2}) \leq
\frac{{4}s}{\dd}\,\|f_{\rm in}\|_{L^{1}_{2}}\,\lm_{s+\gamma}(t)\,.
\end{equation} {Now, to deal with the term $\mathscr{J}_{s,1}(f^{2},f^{2})$, we observe that there is $c >0$ (independent of $s$) such that
\begin{equation}\label{coulomb:e3}
\Big| \langle v\rangle^{s-2}-\langle \vet\rangle^{s-2} \Big| \leq c(s-2)\Big(  \langle v\rangle^{s-3} + \langle \vet\rangle^{s-3}\Big)\big| v - \vet\big|\,,
\end{equation}
which implies, 
\begin{equation*}\begin{split}
\Big(\langle v\rangle^{s-2} - \langle \vet\rangle^{s-2}\Big) &\left(\langle v\rangle^{2}-\langle \vet\rangle^{2}\right) \le c(s-2) |v-\vet|^2 \Big(  \langle v\rangle^{s-3} + \langle \vet\rangle^{s-3}\Big) (\langle v\rangle  + \langle \vet \rangle) \\
& \le 3c(s-2) |v-\vet|^2 (\langle v\rangle^{s-2}  +  \langle \vet \rangle^{s-2}) \qquad \mbox{ for } s\geq 3.
\end{split}
\end{equation*}
where we used that $a^{s-3}b \leq \frac{s-3}{s-2}a^{s-2}+\frac{1}{s-2}b^{s-2} \leq a^{s-2}+b^{s-2}$ for any $a,b >0$, $s \geq 3$ in the last estimate.   Using then a symmetry argument, this yields
\begin{equation}\label{coulomb:e0}
\begin{split}
-\mathscr{J}_{s,1}(f^{2},f^{2})&=2s\int_{\R^{6}}f^{2}f^{2}_{\ast}|v-\vet|^{\g}\langle v\rangle^{s-2}\left(\langle v\rangle^{2}-\langle \vet\rangle^{2}\right)\d v\d\vet\\
&=s\int_{\R^{6}}f^{2}f^{2}_{\ast}|v-\vet|^{\g}\Big(\langle v\rangle^{s-2} - \langle \vet\rangle^{s-2}\Big) \left(\langle v\rangle^{2}-\langle \vet\rangle^{2}\right)\d v\d\vet\\
&\leq 3c\,\frac{s(s-2)}{\dd^2}\int_{\R^{6}}f_{\ast}\,f\left(\langle v\rangle^{s-2}+\langle \vet\rangle^{s-2}\right)|v-\vet|^{\g+2}\d v\d\vet\\
&\leq 6c\,\frac{s(s-2)}{\dd^2}\|f_{\rm in}\|_{L^{1}_{2}}\,\lm_{s+\gamma}(t)\,.
\end{split}
\end{equation}}
Therefore, adding estimate \eqref{coulomb:e-1}, \eqref{coulomb:e1}, and \eqref{coulomb:e0},  {there is some $C_{1} >0$ such that}
\begin{equation}
\mathscr{J}_{s,1}(f,F)\leq  {C_{1}}\,s(s-1) \,\lm_{s+\gamma}(t)\,,\qquad \forall \, s\geq 3\,.
\end{equation}
Furthermore, recall from Lemma \ref{lem:T2} that 
\begin{equation*}
\mathscr{J}_{s,2}(f,F) \leq  {C_{2}}\, s(s-2) \lm_{s-2}(t).
\end{equation*} 
Consequently, 	there exists  {$C(f_{\rm in}) >0$ depending only on $\|f_{\mathrm{in}}\|_{L^{1}_{2}}$ such that
\begin{equation*}
\frac{\d}{\d t}\lm_{s}(t) \leq C(f_{\rm in})\,s(s-1)\lm_{s+\gamma}(t)
\leq C(f_{\rm in})\,s(s-1)\big(\lm_{s}(t)\big)^{\frac{s+\gamma}{s}}\,,\end{equation*}
}for any $t \geq 0$ and any $s \geq 3.$ This leads to \eqref{eq:momlms-short} after integration.  Let us now investigate the case $s \in (2,3)$. If $\g \in [-1,0)$, one simply uses that
  $$\Big(\langle v\rangle^{s-2} - \langle \vet\rangle^{s-2}\Big) \left(\langle v\rangle^{2}-\langle \vet\rangle^{2}\right) \le  |v-\vet| \Big(  \langle v\rangle^{s-2} + \langle \vet\rangle^{s-2}\Big) (\langle v\rangle  + \langle \vet \rangle) $$
to obtain
  \begin{equation*}
\begin{split}
-\mathscr{J}_{s,1}(f^{2},f^{2})
&\leq \,\frac{3s}{\dd^2}\int_{\R^{6}}f_{\ast}\,f\left(\langle v\rangle^{s-1}+\langle \vet\rangle^{s-1}\right)|v-\vet|^{\g+1}\d v\d\vet \leq \frac{6s}{\dd^2}\|f_{\rm in}\|_{L^{1}_{2}}\,\lm_{s+\gamma}(t)\,.
\end{split}
\end{equation*}
This estimate is similar to \eqref{coulomb:e0} and yields again \eqref{eq:momlms-short}. For $s\in(2,3)$ and $\gamma\in(-2,-1)$, \eqref{coulomb:e3} implies
  $$ \Big(\langle v\rangle^{s-2} - \langle \vet\rangle^{s-2}\Big) \left(\langle v\rangle^{2}-\langle \vet\rangle^{2}\right) \le   2 c(s-2) |v-\vet|^2  (\langle v\rangle  + \langle \vet \rangle) , $$
  which yields
  \begin{equation}\label{coulomb:e4}
   \begin{split}
-\mathscr{J}_{s,1}(f^{2},f^{2})
& \leq 2c\,\frac{s(s-2)}{\dd^2}\int_{\R^{6}}f_{\ast}\,f\left(\langle v\rangle+\langle \vet\rangle\right)|v-\vet|^{\g+2}\d v\d\vet\\
&\leq 4c\,\frac{s(s-2)}{\dd^2}\|f_{\rm in}\|_{L^{1}_{2}}\,\lm_{3+\gamma}(t)\leq \frac{12c}{\dd^2}\|f_{\rm in}\|_{L^1_2}\,\lm_{3+\g}(t)\,.
\end{split}
\end{equation}
We have $\lm_{s+\gamma}(t)\le \lm_{2}(t)$ and $\lm_{3+\gamma}(t)\le \lm_{2}(t)$ since $s \in (2,3)$ and $\g \in (-2,-1)$. Consequently,  adding estimate \eqref{coulomb:e-1}, \eqref{coulomb:e1} and \eqref{coulomb:e4}   leads to
  \begin{equation*}
\mathscr{J}_{s,1}(f,F)\leq  {C_{1}} \,\lm_{2}(t)\,,
\end{equation*}
for some $C_1 >0$ depending on $\|f_{\rm in}\|_{L^1_2}$ (recall $s \in (2,3)$). Then,
\begin{equation*}\begin{split}
\dfrac{\d}{\d t}\lm_{s}(t) &\leq C_1\,\lm_{2}(t)+\mathscr{J}_{s,2}(f,F) \leq C_1\,\lm_2(t) + {C_{2}}\, s(s-2) \lm_{s-2}(t)\\
&\leq C_1\lm_2(0) + 3C_2\lm_{1}(t) \leq \left(C_1+3C_2\right)\lm_{2}(0)=:C(f_{\rm in})\,.
\end{split}\end{equation*}
This yields the desired estimate after integration for $s \in (2,3)$ and $\g \in (-2,-1)$.
\end{proof} 
Notice that, besides the above Corollary \ref{cor:L2Ms}, one can also provide the following \emph{appearance and short-time bounds} for $\lM_{s}(\cdot)$. 

\begin{prop}\label{theo:boundedL2}
Assume that $-2 < \g < 0$ and let a nonnegative initial datum $f_{\mathrm{in}}$ satisfying \eqref{hypci}--\eqref{eq:Mass} for some $\dd_0 >0$ be given. For $\dd \in (0,\dd_0]$, let  $f(t,\cdot)$ be a weak-solution to \eqref{LFD}. Assume additionally that
$$\lm_{\frac{2s-3\g}{4}}(0) < \infty, \qquad  {s > 4+|\gamma|}.$$
Then, there exists a constant $c_s(f_{\mathrm{in}})$  such that
\begin{equation}\label{appearLMs}
\lM_{s}(t) \leq c_s(f_{\rm in})\,t^{-\frac32}\,,\qquad t\in\big(0,\tfrac{1}{1+s}\big]\,,
\end{equation}
with moreover,
\begin{equation}\label{rmq:barC}
c_{s}(f_{\rm in}) \leq C(f_{\mathrm{in}})2^{\frac{s}{|\g|}}\left[\lm_{\frac{2s-3\gamma}{4}}^{2}(0)+\bm{C}_{0}^{\frac{s}{|\g|}}\right] \,(1 + s^{\frac6{2+\gamma}}) \qquad \forall  {s > 6+|\g|,} 
\end{equation}
for some positive constants $C(f_{\rm in}), \bm{C}_{0}$ depending only on $\|f_{\rm in}\|_{L^{1}_{2}}$ and $H(f_{\rm in})$ (but not on $s$).
\end{prop}
\begin{proof} Let us pick $s \geq 0$ and set $T_{s}:=\frac{1}{1+s}$. Recall estimate \eqref{final-L2} 
\begin{equation*}
\frac{1}{2}\dfrac{\d}{\d t}\lM_{s}(t) + {\frac{K_{0}}{8}}\lD_{s+\g}(t) \leq C(f_{\mathrm{in}})\, {\left(1+s^{\frac{10}{2+\g}}\right)}\,\lm^2_{\frac{s+\gamma}{2}}(t) \,,\qquad t>0\,,
\end{equation*}
for some positive constant $C(f_{\mathrm{in}})$ depending only on  $\|f_{\mathrm{in}}\|_{L^{1}_{2}}$ and $H(f_{\rm in})$.  Using a classical interpolation inequality (see \eqref{int-ineq} in the next section), one has
\begin{equation*}
\|\langle \cdot \rangle^{\frac{s}{2}}f(t)\|_{L^{2}} \leq \| \langle \cdot \rangle^{\frac{2s-3\g}{4}} f(t) \|^{\frac25}_{L^{1}} \| \langle \cdot \rangle^{\frac{s+\gamma}{2} } f(t) \|^{\frac35}_{L^{6}}\leq  C_{\mathrm{Sob}}^{\frac{3}{5}}\,\lm^{\frac25}_{\frac{2s-3\g}{4}}(t)\,\lD_{s+\g}(t)^{\frac{3}{10}}\,,
\end{equation*}
where we used  Sobolev's inequality \eqref{eq:sobC}. Thus,
$$\lD_{s+\g}(t) \geq C_{\mathrm{Sob}}^{-2}\,\lm_{\frac{2s-3\g}{4}}^{-\frac{4}{3}}(t)\,\lM_{s}^{\frac{5}{3}}(t).$$
For  $s > 4-\g >4+\frac{3\g}{2}$,  we estimate $\lm_{\frac{2s-3\g}{4}}(t)$ and $\lm_{\frac{s+\g}{2}}(t)$  using Proposition \ref{shortime}. We assume for simplicity that both $\frac{2s-3\g}{4}$ and $\frac{s+\g}{2}$ are larger than $3$ to use \eqref{eq:momlms-short} only.  {This amounts to pick $s > 6+|\g|$}. One has 
\begin{equation}\label{eq:msg4}\begin{split}
\lm_{\frac{2s-3\g}{4}}(t) &\leq 2^{\frac{s}{2|\gamma|}}\Big(\lm_{\frac{2s-3\gamma}{4}}(0)+\big(C(f_{\rm in})|\gamma|\tfrac{2s-3\gamma}{4}t\big)^{\frac{s}{2|\gamma|}+\frac34}\Big)\\
&\leq 2^{\frac{s}{2|\gamma|}}\Big(\lm_{\frac{2s-3\gamma}{4}}(0)+\big(C(f_{\rm in})|\gamma|\big)^{\frac{s}{2|\gamma|}+\frac34}\Big)\,,\qquad t \in \left(0,T_{s}\right],
\end{split}\end{equation}
and, in the same way,  for $s>4-\gamma$,
$$\lm^2_{\frac{s+\gamma}{2}}(t) \leq\,2^{\frac{s}{|\gamma|}-2}\Big( \lm^{2}_{\frac{s+\gamma}{2}}(0)+\big(C(f_{\rm in})|\gamma|\big)^{\frac{s}{|\gamma|}-1}\Big)\,, \qquad t \in \left(0,T_{s}\right]\,$$
(note that $\lm_{\frac{s+\gamma}{2}}(0)<\infty$ because $\frac{s+\gamma}{2}<\frac{2s-3\gamma}{4} $).
Therefore, 
\begin{equation}\label{eq:xtyt}
\dfrac{\d}{\d t}\lM_{s}(t) + \bm{a}_{s}(f_{\rm in})\lM_{s}(t)^{\frac{5}{3}}  \leq  \bm{k}_s(f_{\rm in})\,,\qquad t \in \left(0,T_{s}\right],
\end{equation}
where
\begin{align}\label{constantcs0}
\begin{split}
\frac{1}{\bm{a}_{s}(f_{\rm in})}&=2^{\frac{2s}{3|\gamma|}{+  {\frac73} }}\,C^2_{\mathrm{Sob}}\,K^{-1}_0\,\Big(\lm^{\frac43}_{\frac{2s-3\gamma}{4}}(0)+\big(C(f_{\rm in})|\gamma|\big)^{\frac{2s}{3|\gamma|}+1}\Big)\,,\\
\bm{k}_{s}(f_{\rm in})&=C(f_{\mathrm{in}})\left(1+s^{\frac{10}{2+\g}}\right)\,2^{\frac{s}{|\gamma|}  {-1} }\Big( \lm^{2}_{\frac{s+\gamma}{2}}(0)+\big(C(f_{\rm in})|\gamma|\big)^{\frac{s}{|\gamma|}-1}\Big)\,.
\end{split}
\end{align} 
The conclusion then follows by a comparison argument. Namely,
introducing 
$$\bm{\zeta}(x)={\bm{k}_{s}}(f_{\mathrm{in}})-\bm{a}_{s}(f_{\mathrm{in}})x^{\frac{5}{3}}, \qquad x >0\,,$$ and $\bar{x}=\left(\frac{{2\bm{k}_{s}}(f_{\mathrm{in}})}{\bm{a}_{s}(f_{\mathrm{in}})}\right)^{\frac{3}{5}}$, one has $\bm{\zeta}(x) \leq -\frac{\bm{a}_{s}(f_{\rm in})}{2}x^{\frac{5}{3}}$ for $x \geq \bar{x}$. Fixing $t_{\star}\in \left(0,T_{s}\right]$, one has according to \eqref{eq:xtyt} that
$$\dfrac{\d}{\d t}\lM_{s}(t) \leq \bm{\zeta}\left(\lM_{s}(t)\right), \qquad t \in \left(t_{\star},T_{s}\right).$$
Three cases may occur: 
\begin{enumerate}[$i)$]
\item if $\lM_{s}(t_{\star}) < \bar{x}$, then since $\bm{\zeta}(\bar{x}) <0$, one has $\lM_{s}(t) \leq \bar{x}$ for any $t \geq t_{\star},$ 
\item if $\lM_{s}(t_{\star}) > \bar{x}$ and $\lM_{s}(t) > \bar{x}$ for any $t \in (t_{\star},T_{s})$, then one has
$$\dfrac{\d}{\d t}\lM_{s}(t) \leq \bm{\zeta}\left(\lM_{s}(t)\right) \leq -\frac{\bm{a}_{s}(f_{\rm in})}{2}\lM_{s}(t)^{\frac{5}{3}}\,, \qquad t \in (t_{\star},T_{s})$$
which, after integration, yields
$$\lM_{s}(t) \leq \left(\frac{3}{\bm{a}_{s}(f_{\rm in})(t-t_{\star})}\right)^{\frac{3}{2}}\,, \qquad t \in (t_{\star},T_{s}).$$
\item if $\lM_{s}(t_{\star}) > \bar{x}$ and $\lM_{s}(\overline{t}) \le \bar{x}$ for some $\overline{t}\in (t_{\star},T_{s})$ then, setting $$T_{\star}:=\inf\{t \in (t_{\star},T_{s})\;:\;\lM_{s}(t) \leq \bar{x}\}\,,$$ one has, as in the second point, that
$$\lM_{s}(t) \leq \left(\frac{3}{\bm{a}_{s}(f_{\rm in})(t-t_{\star})}\right)^{\frac{3}{2}}\,, \qquad t \in (t_{\star},T_{\star}).$$
Since $\lM_{s}(T_{\star})=\bar{x}$ by continuity, we deduce that $\lM_{s}(t) \leq \bar{x}$ for all $t \geq T_{\star}$ from the first point.
\end{enumerate}
\noindent
In all the cases it holds that
$$\lM_{s}(t) \leq \max\left(\bar{x},\left(\frac{3}{\bm{a}_{s}(f_{\rm in})(t-t_{\star})}\right)^{\frac{3}{2}}\right)\,, \qquad t > t_{\star}\,,$$
from which the result follows by letting $t_{\star} \to 0$ and with
$$c_{s}(f_{\mathrm{in}})= \max\left\{\left(\frac{3}{\bm{a}_{s}(f_{\rm in})}\right)^{\frac32}\,,\,{2^{\frac35}} \left(\frac{\bm{k}_{s}(f_{\rm in})}{\bm{a}_{s}(f_{\rm in})}\right)^{\frac35} \right\}$$
with $\bm{a}_{s}(f_{\mathrm{in}})$ and $\bm{k}_{s}(f_{\mathrm{in}})$ defined in \eqref{constantcs0} with constant $C(f_{\rm in})$ depending only on  {$\|f_{\rm in}\|_{L^{1}_{2}}$} and $H(f_{\rm in})$. In particular, as far as the dependence with respect to $s$ is concerned, we easily derive \eqref{rmq:barC}. 

 {If $\min\left(\frac{2s-3\g}{4},\frac{s+\g}{2}\right) < 3$, then one has to estimate $\lm_{\frac{2s-3\g}{4}}(t)$ and/or $\lm_{\frac{s+\g}{2}}(t)$ using \eqref{eq:momlms-short0}. The same computations as before allows then to  end up again with \eqref{eq:xtyt} (with slightly different expression for $\bm{k}_{s}(f_{\rm in})$ and $\bm{a}_{s}(f_{\rm in})$) and get the result as in the previous case. Details are left to the reader.}
\end{proof} 

\subsection{Combined estimates}

We now introduce  
$$\bm{E}_{s}(t)=\lm_{s}(t)+\frac{1}{2}\lM_{s}(t).$$
Combining Proposition \ref{prop:L2} with Proposition \ref{lem:ms1}, one gets the following lemma.
\begin{lem}\label{lem:C1sC2s}
Assume that $-2 < \g < 0$  and let a nonnegative initial datum $f_{\mathrm{in}}$ satisfying \eqref{hypci}--\eqref{eq:Mass} for some $\dd_0 >0$ be given. For $\dd \in (0,\dd_0]$, let  $f(t,\cdot)$ be a weak-solution to \eqref{LFD}. Then, for any $s >2$, there are positive constants $\bm{K}_{s}, C_{1,s}$ which depend on $s$ and $f_{\mathrm{in}}$ (through   $H(f_{\rm in})$ and $\|f_{\mathrm{in}}\|_{L^{1}_{2}}$)  such that
\begin{multline}\label{eq:Est}
  \dfrac{\d}{\d t}\bm{E}_{s}(t)+s\frac{\eta_{\star}}{4}\lm_{s+\g}(t) + {\frac{K_{0}}{16}}\lD_{s+\g}(t) \\
  \leq {2s}\bm{K}_{s} + C_{1,s}\lM_{s+\g}(t)+\bar{C}(f_{\rm in}){\left(1+s^{\frac{10}{2+\g}}\right)}\lm_{\frac{s+\g}{2}}^{2}(t),
\end{multline}
where $\bar{C}(f_{\rm in})$ is the constant in inequality \eqref{final-L2}, $\bm{K}_{s}$ was estimated in \eqref{rmq:Ks}, and $$C_{1,s}=\bar{C}_{1}\big(s^{\frac{2}{2+\g}}+s\big),$$
 for some positive constant $\bar{C}_{1}$ depending only on $f_{\mathrm{in}}$ through $K_{0}$ and $\|f_{\mathrm{in}}\|_{L^{1}_{2}}$.
\end{lem}

\begin{proof} We simply apply \eqref{eq:mom-s-1} with 	 {$\delta=\min\left({\frac{K_{0}}{16s}},1\right)$} and add the obtained inequality with \eqref{final-L2} to obtain the result.  {We derive easily the estimate for $C_{1,s}$ since, for $s$ large enough, $s\delta=\frac{K_{0}}{16}.$}\end{proof}

We have all in hands to prove Theorem \ref{theo:main-moments} in the introduction.
\begin{proof}[Proof of Theorem \ref{theo:main-moments}] {Let $s > 4+|\g|$. Since $\frac{2s-3\g}{4} \leq s$, one has $\max\left(\lm_{s}(0),\lm_{\frac{2s-3\g}{4}}(0)\right)=\lm_{s}(0) < \infty$}, and one deduces from Propositions~\ref{shortime} and \ref{theo:boundedL2}  that
$$\bm{E}_{s}(t) \leq \bar{C}_{s}t^{-\frac{3}{2}}, \qquad t \in \left(0,\tfrac{1}{1+s}\right]\,,$$
with
\begin{equation*}
\bar{C}_{s}=\left[\lm_{s}(0)^{\frac{|\g|}{s}}+C(f_{\rm in})|\g|\right]^{\frac{s}{|\gamma|}} + {\frac{1}{2}c_{s}(f_{\rm in})},
\end{equation*}
where  $c_{s}(f_{\rm in})$ and $C(f_{\rm in})$ are defined in  Propositions \ref{theo:boundedL2} and  \ref{shortime}.  Since $s> 4 +|\gamma|$, we use \eqref{eq:momlms-short} and not \eqref{eq:momlms-short0}.   In particular, using that 
$$\lm_{\frac{2s-3\g}{4}}^{2}(0) \leq \lm_{\frac{3|\g|}{2}}(0)\lm_{s}(0)$$
thanks to Cauchy-Schwarz inequality, we deduce from \eqref{rmq:barC} that, for $ s>6+|\gamma|$, 
\begin{multline*}
\bar{C}_{s} \leq 2^{\frac{s+\g}{|\g|}}\left(\lm_{s}(0)+\left({C(f_{\rm in})}|\g|\right)^{\frac{s}{|\g|}}\right) 
+ {C(f_{\rm in})2^{\frac{s}{|\g|}}\lm_{\frac{3|\g|}{2}}(0)\lm_{s}(0)}
\,{(1 + s^{\frac{6}{2+\gamma}})}  \\
+{C(f_{\rm in})\left({2} \bm{C}_{0}\right)^{\frac{s}{|\g|}}}\, {(1 + s^{\frac{6}{2+\gamma}})}.
\end{multline*}
Consequently, there are positive constants $C_{0},C_{1} >0$  depending only on $\|f_{\rm in}\|_{L^{1}_{2}}$ and $H(f_{\rm in})$ such that
\begin{equation}\label{eq:barC}
\bar{C}_{s} \leq C_{0}2^{\frac{s}{|\g|}}\lm_{\frac{3|\g|}{2}}(0)\lm_{s}(0) + C_{1}^{\frac{s}{|\g|}}, \qquad s >6+|\gamma|,
\end{equation}
where we used  that $1\le \lm_{\frac{3|\gamma|}{2}}(0)$. Let us then focus on the evolution of $\bm{E}_{s}(t)$ for $t \geq \tfrac{1}{1+s}.$ We start with \eqref{eq:Est} and  estimate  $\lM_{s+\g}(t)$ and $\lm_{\frac{s+\g}{2}}^{2}(t)$  as in the proof of Proposition \ref{prop:L2} (see also \cite[Lemma 3.5]{ABL}). Namely, as seen at the end of the proof of Prop. \ref{prop:L2}, there is a universal constant $C >0$ independent of $s$ such that, for any $\delta >0$,
$$\,\lM_{s+\g}(t) \leq C {\delta}^{-\frac{3}{2}} \|\langle \cdot \rangle^{\frac{\gamma+s}{2}}f(t)\|^{2}_{L^{1}}+\delta\,\lD_{s+\g}(t).$$
Inserting this in \eqref{eq:Est} and choosing $\delta >0$ such that $C_{1,s}\delta={\frac{K_{0}}{32}}$, one has 
$$\frac{\d}{\d t}\bm{E}_{s}(t)+s\frac{\eta_{\star}}{4}\lm_{s+\g}(t)+{\frac{K_{0}}{32}}\lD_{s+\g}(t) \leq  {2s}\bm{K}_{s} +\bar{C}_{3,s}\lm_{\frac{s+\g}{2}}(t)^{2}\,,$$
where
\begin{equation}\label{eq:C3s}
\bar{C}_{3,s}=C\,\left({\frac{K_{0}}{32}}\right)^{-\frac{3}{2}}C_{1,s}^{\frac{5}{2}}{+\bar{C}(f_{\mathrm{in}})\left(1+s^{\frac{10}{2+\gamma}}\right)}.\end{equation}
Now, for $s \geq 4 -\g$, 
\begin{equation}\label{eq:lmsg2}
\lm_{\frac{s+\gamma}{2}}(t)^{2}\leq \lm_{2}(t)^{\frac{s+\gamma}{s+\gamma-2}}\,\lm_{s+\gamma}(t)^{\frac{s+\gamma-4}{s+\gamma-2}} \leq \|f_{\mathrm{in}}\|_{L^{1}_{2}}^{\frac{s+\g}{s+\g-2}}\lm_{s+\g}(t)^{\frac{s+\g-4}{s+\g-2}}\,,\qquad t \geq 0\,,
\end{equation}
where we used the conservation of mass and energy. From Young's inequality, for any $\delta_{0} >0$, one has then
$$\lm_{\frac{s+\g}{2}}(t)^{2} \leq \|f_{\mathrm{in}}\|_{L^{1}_{2}}^{\frac{s+\g}{2}}\,\delta_{0}^{-\frac{s+\g-4}{2}} + \delta_{0}\,\lm_{s+\gamma}(t)\,,\qquad t >0\,.$$
Choosing now $\delta_{0} >0$ so that $\bar{C}_{3,s}\delta_{0}=s\frac{\eta_{\star}}{8}$, we end up with
\begin{equation}\label{eq:Es}
\frac{\d}{\d t}\bm{E}_{s}(t)+s\frac{\eta_{\star}}{8}\lm_{s+\g}(t)+{\frac{K_{0}}{32}}\lD_{s+\g}(t) \leq \overline{\bm{C}}_{s},
\end{equation}
where
\begin{equation}\label{eq:CCs}
\overline{\bm{C}}_{s}= {2s}\bm{K}_{s} +  {\bar{C}_{3,s}}\|f_{\mathrm{in}}\|_{L^{1}_{2}}^{\frac{s+\g}{2}}\,\delta_{0}^{-\frac{s+\g-4}{2}}= {2s}\bm{K}_{s}+   {\bar{C}_{3,s}}\|f_{\mathrm{in}}\|_{L^{1}_{2}}^{\frac{s+\g}{2}}\left(\frac{8}{\eta_{\star}}\frac{\bar{C}_{3,s}}{s}\right)^{\frac{s+\g-4}{2}}.
\end{equation}
Integrating this inequality on $\left(\tfrac{1}{1+s},t\right)$ gives
$$\bm{E}_{s}(t) \leq \bm{E}_{s}\left(\tfrac{1}{1+s}\right)+\overline{\bm{C}}_{s}\left(t-\tfrac{1}{1+s}\right), \qquad t\geq \tfrac{1}{1+s}\,,$$
so that
$$ \bm{E}_{s}(t)  \leq \bar{C}_{s}\left(\tfrac{1}{1+s}\right)^{-\frac32}+\overline{\bm{C}}_{s}t \leq \bar{C}_{s}(1+s)^{\frac52}t+\overline{\bm{C}}_{s}t, \qquad t\geq \tfrac{1}{1+s}\,.$$
Proposition  \ref{shortime} gives now the second part of \eqref{res19} for small times whereas Proposition \ref{theo:boundedL2} and \eqref{eq:Es} lead to the second part of  \eqref{res19} for large times with $$\bm{C}_{s} :=\max\left(\bar{C}_{s},\bar{C}_{s}(1+s)^{\frac52}+\overline{\bm{C}}_{s}\right)=\bar{C}_{s}(1+s)^{\frac52}+\overline{\bm{C}}_{s}.$$
Combining \eqref{eq:C3s} with Lemma \ref{lem:C1sC2s}, one sees first that there is $C_{3} >0$ depending only on $\|f_{\mathrm{in}}\|_{L^{1}_{2}}$ and $H(f_{\rm in})$ such that,
$$\bar{C}_{3,s} \leq C_{3} {\left(s^{\frac52}+s^{\frac{5}{2+\g}}+s^{\frac{10}{2+\g}}\right) \leq 2C_{3}s^{\frac{10}{2+\g}}}, \qquad s >2.$$
Then, using \eqref{rmq:Ks} and \eqref{eq:CCs}, one deduces that there exists $\beta_{0} >0$ depending only on $\|f_{\mathrm{in}}\|_{L^{1}_{2}}$ and $H(f_{\rm in})$ such that
$$\overline{\bm{C}}_{s} \leq \beta_{0}\left[\left(\beta_{0}\,s\right)^{ {\frac{s-2}{\g+2}}+1}+\left(\beta_{0}\,s\right)^{{\frac{8-\g}{4+2\g}(s+\g-2)+1}}\right], \qquad s \geq 4-\g$$
 {where we roughly estimate $\|f_{\rm in}\|_{L^{1}_{2}}^{\frac{s+\g}{2}}$ as $\beta_{0}^{\frac{8-\g}{4+2\g}(s+\g-2)+1}$ once we notice that $\frac{s+\g}{2} \leq {\frac{8-\g}{4+2\g}(s+\g-2)+1}$ for $s \geq 4-\g$.} 
Combining these estimates with \eqref{eq:barC} {and because $\frac{s-2}{\g+2} < \frac{8-\g}{4+2\g}(s+\g-2)$ for $s > 4-\g$, one deduces easily the estimate \eqref{rmq:Csfinal}.}  
\end{proof}

\begin{rmq} Of course, if $f_{\mathrm{in}}$ is actually belonging to $L^{1}_{s}(\R^{3}) \cap L^{2}_{s}(\R^{3})$, then $\bm{E}_{s}(0) <\infty$ and one deduces after integration of \eqref{eq:Es} that
$$\bm{E}_{s}(t) \leq \bm{E}_{s}(0)+ \overline{\bm{C}}_{s} \, t, \qquad t \geq0.$$
\end{rmq}
\noindent
The above result shows the linear growth of both the $L^{1}$-moments and $L^{2}$-moments. Actually, it is possible to sharpen the above estimates  {(for small $s$)} with the following proposition.

\begin{prop}\label{prop:boundedL2}
Assume that $-2 < \g < 0$  and let a nonnegative initial datum $f_{\mathrm{in}}$ satisfying \eqref{hypci}--\eqref{eq:Mass} for some $\dd_0 >0$ be given. For $\dd \in (0,\dd_0]$, let  $f(t,\cdot)$ be a weak-solution to \eqref{LFD}. Then, for any $s \in [0,\frac{8+3\g}{2}]$,
\begin{enumerate}
\item if $f_{\mathrm{in}} \in L^{2}_{s}(\R^{3})$, there is a positive constant  $\bm{C}_{\mathrm{in}}$ depending only on  $\|f_{\mathrm{in}}\|_{L^{1}_{2}}$, $H(f_{\rm in})$ such that
  \begin{equation}\label{propLMs}
  \sup_{t\geq0}\|f(t)\|_{L^{2}_{s}}=\sup_{t\geq0}\left\|\langle \cdot \rangle^{\frac{s}{2}}f(t,\cdot)\right\|_{L^{2}} \leq \bm{C}_{\mathrm{in}}.\end{equation}
\item There are constants $\bm{C}_{0,\mathrm{in}}$ and $\tilde{\bm{C}}_{\mathrm{in}}$ depending only on $\|f_{\mathrm{in}}\|_{L^{1}_{2}}$, and $H(f_{\rm in})$  such that for any $t>0$,
\begin{equation}\label{appearLMs}\lM_{s}(t) \leq \max\left(\tilde{\bm{C}}_{\mathrm{in}}\,,\,\frac{\bm{C}_{0,\mathrm{in}}}{t^{\frac{3}{2}}}\right).
\end{equation}
\end{enumerate}

\end{prop}

\begin{proof} Let us pick $s \in [0,\frac{8+3\g}{2}]$. In light of estimate \eqref{final-L2}, since $\frac{s+\g}{2} \leq 2,$ we have  that
\begin{equation*}
\frac{1}{2}\dfrac{\d}{\d t}\lM_{s}(t) +  {\frac{K_{0}}{8}}\lD_{s+\g}(t) \leq \bar{C}(f_{\mathrm{in}}) {\left(1+s^{\frac{10}{2+\g}}\right)}\|f_{\rm in}\|^{2}_{L^{1}_{2}}\leq \bm{K}(f_{\rm in})\,,\qquad t>0\,,
\end{equation*}
for some positive constant $\bm{K}(f_{\mathrm{in}})$ depending only on $f_{\mathrm{in}}$ through $\|f_{\mathrm{in}}\|_{L^{1}_{2}}$ and $H(f_{\rm in})$. Arguing exactly as in the proof of Proposition \ref{theo:boundedL2}, Eq. \eqref{eq:xtyt} {but with $\bm{m}_{\frac{2s-3\gamma}4}\le \|f_{\mathrm{in}}\|_{L^1_2}$}, one deduces that
\begin{equation}\label{eq:xtyt0}
\dfrac{\d}{\d t}\lM_{s}(t) + \kappa_{\mathrm{in}}\lM_{s}(t)^{\frac{5}{3}} \leq  {2 \bm{K}(f_{\mathrm{in}})}\,,\qquad t>0,
\end{equation}
where we set $\kappa_{\mathrm{in}}=K_{0}\left({4\,C_{\mathrm{Sob}}^{2}\| f_{{\rm in}} \|^{\frac{4}{3}}_{ L^{1}_{2}}}\right)^{-1}$. 
The first point  follows then by a simple comparison argument choosing $\bm{C}_{\mathrm{in}}^{2}={\max\left(\left(\frac{2\bm{K}(f_{\mathrm{in}})}{\kappa_{\mathrm{in}}}\right)^{\frac{3}{5}},\lM_{s}(0)\right)}$,
whereas the second point is obtained exactly as in the proof of Proposition \ref{theo:boundedL2}.
\end{proof}
The following corollary is a simple consequence of Proposition \ref{prop:boundedL2}. 

\begin{cor} 
Assume that $-2 < \g < 0$  and let a nonnegative initial datum $f_{\mathrm{in}}$ satisfying \eqref{hypci}--\eqref{eq:Mass} for some $\dd_0 >0$ be given. For $\dd \in (0,\dd_0]$, let  $f(t,\cdot)$ be a weak-solution to \eqref{LFD}. Then,  there exists a positive constant $C(f_{\mathrm{in}})$ depending only on $\|f_{\mathrm{in}}\|_{L^{1}_{2}}$ and $H(f_{\rm in})$   such that
\begin{equation}\label{eq:bm1}
\left|\bm{b}[f(t)](v)\right| \leq C(f_{\mathrm{in}})\,\langle v\rangle^{\max(0,1+\g)}{\left(1+t^{-\frac{3}{2}}\right), \qquad \forall \,v \in \R^{3},\quad t > 0.}\end{equation}\
\end{cor}

\begin{proof} Recall that
$$\left|\bm{b}[f(t)](v)\right|=2\left|\int_{\R^{3}}(v-\vet)|v-\vet|^{\g}f(t,\vet)\d\vet\right| \leq 2\int_{\R^{3}}|v-\vet|^{1+\g}f(t,\vet)\d \vet\,.$$
If $1+\g \geq0$, one notices that $|v-\vet|^{1+\g} \leq 2^{\frac{1+\g}{2}}\langle v\rangle^{1+\g}\langle \vet\rangle^{1+\g}$, so that
\begin{equation}\label{eq:bm2}
\left|\bm{b}[f(t)](v)\right| \leq 2^{\frac{3+\g}{2}}\langle v\rangle^{1+\g}\lm_{1+\g}(t) \leq 2^{\frac{3+\g}{2}}\|f_{\mathrm{in}}\|_{L^{1}_{2}}\langle v\rangle^{1+\g}, 
\end{equation}
since $1+\g <2.$ Let us now  investigate  the case $1+\g < 0.$ One splits the integral defining $\bm{b}[f(t)](v)$ according to $|v-\vet| \leq 1$ and $|v-\vet|>1$. One has then, thanks to Cauchy-Schwarz inequality,
\begin{equation*}\begin{split}
|\bm{b}[f(t)](v)| &\leq 2\int_{\R^{3}}f(t,\vet)\d \vet + 2\int_{|v-\vet| <1}|v-\vet|^{\g+1}f(t,\vet)\d \vet \\
&\leq 2\|f(t)\|_{L^{1}}+2\|f(t)\|_{L^{2}}\,\left(\int_{|v-\vet|<1}|v-\vet|^{2\g+2}\d \vet\right)^{\frac{1}{2}}.
\end{split}\end{equation*}
Noticing that 
$$\int_{|v-\vet|<1}|v-\vet|^{2\g+2}\d \vet=|\S^{2}|\int_{0}^{1}r^{2(\g+2)}\d r < \infty, \qquad 2+\g >0,$$
is independent of $v$, one gets
\begin{equation}\label{eq:bm3}
\left|\bm{b}[f(t)](v)\right| \leq C\,\big(\|f(t)\|_{L^{1}}+\|f(t)\|_{L^{2}}\big)
\end{equation}
for some universal positive constant $C >0.$ {We deduce then \eqref{eq:bm1} from the conservation of mass and Proposition \ref{prop:boundedL2}. }
\end{proof}
\noindent
Estimate \eqref{eq:bm1} implies of course that ${\sup_{t\geq1}}\left|\bm{b}[f(t)] \cdot v\right| \leq C(f_{\mathrm{in}})\,\langle v\rangle^{\max(1,2+\g)}$.  Additionally, we need the following estimate.

\begin{lem}
Assume $-2 < \g <0$. There exist two  constants $c_{0},c_{1} >0$ {depending only on $\g$} such that, for any nonnegative $g \in L^{1}_{2+\g}(\R^{3})$ 
\begin{equation}\label{eq:bm}
\big|\bm{b}[g]\cdot v\big| \leq c_{0}\langle v\rangle^{\g+2}\|g\|_{L^{1}_{2+\g}}-c_{1}\,\langle v\rangle^{2}\bm{c}_{\g}[g](v)\,,
\end{equation}
where we recall that $-\bm{c}_{\g}[g](v)=2(\g+3)\int_{\R^{3}}|v-\vet|^{\g}g(\vet)\d\vet \geq0.$
\end{lem}
\begin{proof} Let $\delta >0$ be given. One writes
\begin{equation*}\begin{split}
\bm{b}[g] \cdot v&=-2\int_{\R^{3}}v \cdot (v-\vet)\,|v-\vet|^{\g}g(\vet)\d\vet\\
&=-\int_{\R^{3}}|v-\vet|^{\g}g(\vet)\left(|v-\vet|^{2}+|v|^{2}-|\vet|^{2}\right)\d\vet.\end{split}\end{equation*}
Since 
$\left|\,|v|^{2}-|\vet|^{2}\right| \leq \frac{1}{2}|v-\vet|^{2}+\frac{1}{2}|v+\vet|^{2} \leq \frac{3}{2}|v-\vet|^{2}+4|v|^{2}$, we get
$$|\bm{b}[g]\cdot v| \leq \frac{5}{2}\int_{\R^{3}}|v-\vet|^{\g+2}g(\vet)\d\vet+4\langle v\rangle^{2}\int_{\R^{3}}|v-\vet|^{\g}g(\vet)\d\vet, $$
which gives the result using that
$$  |v-\vet|^{\g+2}\leq 2^{\frac{\g+2}{2}}\langle v\rangle^{\g+2}\langle \vet\rangle^{\g+2}\,,\quad \text{for}\; -2 < \g <0\,,$$ and recalling the definition of $\bm{c}_{\g}[g]$.
\end{proof}

\section{De Giorgi's approach to pointwise bounds}\label{sec:level}

We introduce, as in \cite{ricardo}, for any \emph{fixed} $\l \in (0,\frac{1}{2\dd})$,
$$f_{\l}(t,v) :=(f(t,v)-\l), \qquad f_{\l}^{+}(t,v) :=f_{\l}(t,v)\ind_{\{f\geq\l\}}.$$
To prove an $L^{\infty}$ bound for $f(t,v)$, one looks for an $L^{2}$-bound for $f_{\l}$. We start with the following estimate.
\begin{lem}\label{lem:fl+} 
Assume that $-2 < \g < 0$  and let a nonnegative initial datum $f_{\mathrm{in}}$ satisfying \eqref{hypci}--\eqref{eq:Mass} for some $\dd_0 >0$ be given. For $\dd \in (0,\dd_0]$, let  $f(t,\cdot)$ be a weak-solution to \eqref{LFD}.
There exist $c_{0},C_{0} >0$ depending only on  $\|f_{\rm in}\|_{L^1_2}$ and $H(f_{\rm in})$  such that, for any $\l \in (0,\frac{1}{2\dd})$,
\begin{equation}\label{eq:13Ric}\begin{split}
\frac{1}{2}\frac{\d}{\d t}\|f_{\l}^{+}(t)\|_{L^{2}}^{2} &+ c_{0}\int_{\R^{3}}\left|\nabla \left(\langle v \rangle^{\frac{\g}{2}}f_{\l}^{+}(t,v)\right)\right|^{2}\d v \\
&\leq C_{0}\|\langle \cdot \rangle^{\frac{\g}{2}}f_{\l}^{+}(t)\|_{L^{2}}^{2}-\l \int_{\R^{3}}\bm{c}_{\g}[f](t,v)\,f_{\l}^{+}(t,v)\d v.\end{split}\end{equation}
\end{lem}

\begin{proof} Given $\l \in (0,\frac{1}{2\dd})$, one has
$\partial_{t}\left(f_{\l}^{+}\right)^{2}=2f_{\l}^{+}\partial_{t}f_{\l}^{+}=2f_{\l}^{+}\partial_{t}f$
and $\nabla f_{\l}^{+}=\ind_{\{f\geq\l\}}\nabla f$, so that
\begin{multline*}
\frac{1}{2}\frac{\d}{\d t}\|f_{\l}^{+}(t)\|_{L^2}^{2}=-\int_{\R^{3}}\bm{\Sigma}\nabla f\cdot \nabla f_{\l}^{+} \d v
+\int_{\R^{3}}f(1-\dd f)\bm{b}[f]\cdot \nabla f_{\l}^{+}\d v\\
=-\int_{\R^{3}}\bm{\Sigma}\nabla f^{+}_{\l}\cdot \nabla f_{\l}^{+}\d v
+\int_{\R^{3}}f(1-\dd f)\bm{b}[f]\cdot \nabla f_{\l}^{+}\d v.\end{multline*}
Now, one easily checks that
\begin{equation*}\begin{split}
f(1-\dd f)\nabla f^{+}_{\l}&=\ind_{\{f\geq\l\}}\left(f_{\l}(1-\dd f)\nabla f_{\l}^{+} + \l(1-\dd f)\nabla f_{\l}^{+}\right)\\
&=(1-2 \dd \l)f^{+}_{\l} \nabla f^{+}_{\l} -\dd  (f^{+}_{\l})^{2}  \nabla f^{+}_{\l} +\l(1-\dd \l)\nabla f^{+}_{\l}\\
&=(\tfrac{1}{2}-\dd \l)\nabla (f^{+}_{\l})^{2}-\frac{\dd}{3}\nabla (f^{+}_{\l})^{3} + \l(1-\dd\l)\nabla f^{+}_{\l}.
\end{split}\end{equation*}
Therefore
\begin{multline*}
\frac{1}{2}\frac{\d}{\d t}\|f_{\l}^{+}(t)\|_{L^2}^{2}+\int_{\R^{3}}\bm{\Sigma}\nabla f^{+}_{\l}\cdot \nabla f^{+}_{\l}\d v
= (\tfrac{1}{2}-\dd\l)\int_{\R^{3}}\bm{b}[f] \cdot \nabla (f^{+}_{\l})^{2} \d v  \\+ \l(1-\dd\l)\int_{\R^{3}}\bm{b}[f]\cdot \nabla f^{+}_{\l}\d v
-\frac{\dd}{3}\int_{\R^{3}}\bm{b}[f]\cdot \nabla (f^{+}_{\l})^{3}\d v\\
= -(\tfrac{1}{2}-\dd \l)\int_{\R^{3}}\bm{c}_{\g}[f](f^{+}_{\l})^{2}\d v - \l(1-\dd \l)\int_{\R^{3}}\bm{c}_{\g}[f] f^{+}_{\l}\d v+\frac{\dd}{3}\int_{\R^{3}}\bm{c}_{\g}[f]\,(f^{+}_{\l})^{3}\d v.\end{multline*}
Using that $\tfrac{1}{2}-\dd \l >0$ and $-\bm{c}_{\g}[f] \geq 0$, we deduce from Proposition \ref{diffusion} and Proposition \ref{prop:GG} with $\phi=f_{\l}^{+}$ and $\delta >0$ small enough that there exist positive constants $c_{0}, C_{0}$ depending on $\|f_{\mathrm{in}}\|_{L^{1}_{2}}$ and $H(f_{\rm in})$ such that
$$
\frac{1}{2}\frac{\d}{\d t}\|f_{\l}^{+}(t)\|^{2}_{L^2} + c_{0}\int_{\R^{3}}\langle v\rangle^{\g}\left|\nabla f^{+}_{\l}(t,v)\right|^{2}\d v 
\leq C_{0}\|\langle \cdot\rangle^{\frac{\g}{2}}f^{+}_{\l}(t)\|_{L^2}^{2} -\l\int_{\R^{3}}\bm{c}_{\g}[f]\,f_{\l}^{+}\d v.$$
Notice that, using again inequality \eqref{eq:Gradient}, we can replace easily the above with \eqref{eq:13Ric} with a different (but irrelevant) choice of $c_{0},C_{0}$. 
\end{proof} 

Inspired by De Giorgi's iteration method introduced for elliptic equations, the crucial point in the level set approach of \cite{ricardo} is to compare some suitable energy functional associated to $f_{\l}^{+}$ with the same energy functional at some different level $f_{k}^{+}.$ The key observation being that, if 
$0 \leq k <\ell$, then
\begin{equation*}
0 \leq f^{+}_{\l}\leq f^{+}_{k},\quad\text{and}\quad \mathbf{1}_{\{f_{\l} \geq 0\}} \leq \frac{f^{+}_{k}}{\ell - k}\,.
\end{equation*}
Indeed, on the set $\{f_{\l}\geq0\}$, one has $f \geq \l >k$, so that $f^{+}_{k}=f-k=f^{+}_{\l}+(\l-k)$ and $\frac{f^{+}_{k}}{\l-k}=\frac{f^{+}_{\l}}{\l-k}+1 \geq 1.$ In particular, for any $\alpha \geq0$, we deduce that
$$\mathbf{1}_{\{f_{\l}\geq0\}}=\left(\mathbf{1}_{\{f_{\l}\geq0\}}\right)^{\alpha} \leq \left(\frac{f^{+}_{k}}{\ell-k}\right)^{\alpha},$$ 
which, since $f_{\l}^{+}\leq f^{+}_{k}$, gives
\begin{equation}\label{eq:alphaf}
f_{\l}^{+} \leq \left(\ell-k\right)^{-\alpha}\,\left(f^{+}_{k}\right)^{1+\alpha} \qquad \forall \,\alpha \geq 0, \qquad 0 \leq k < \l.\end{equation}
On this basis, we need the following interpolation inequality where the dependence of $f_{\l}^{+}$ with respect to time is omitted hereafter.
\begin{lem}\label{lem:changelevel}
Assume that $-2 < \g < 0$  and let a nonnegative initial datum $f_{\mathrm{in}}$ satisfying \eqref{hypci}--\eqref{eq:Mass} for some $\dd_0 >0$ be given. For $\dd \in (0,\dd_0]$, let  $f(t,\cdot)$ be a weak-solution to \eqref{LFD}. There exists $C >0$  (independent of $\dd$ and $f_{\rm in}$) such that, for any $0 \leq k < \l$, one has 
\begin{equation}\label{eq:flL2}
\|\langle \cdot \rangle^{\frac{\g}{2}}f_{\l}^{+}\|_{L^{2}}^{2} \leq C\,(\l-k)^{-\frac{4}{3}}\left\|\nabla \left(\langle \cdot \rangle^{\frac{\g}{2}}f_{k}^{+}\right)\right\|_{L^{2}}^{2}\,\left\|f^{+}_{k}\right\|_{L^{2}}^{\frac{4}{3}}.\end{equation}
Moreover, for $p \in [1,3)$, there is $C_{p} >0$ such that,
\begin{equation}\label{eq:flLp}
\|\langle \cdot \rangle^{\g}f_{\l}^{+}\|_{L^{p}} \leq C_{p}(\l-k)^{-(\frac{2}{p}+\frac{1}{3})}\,\left\|\nabla\left(\langle \cdot \rangle^{\frac{\g}{2}}f_{k}^{+}\right)\right\|_{L^{2}}^{2}\,\|f_{k}^{+}\|_{L^{2}}^{\frac{2}{p}-\frac{2}{3}}, \qquad 0 \leq k < \l.\end{equation}
Finally, for any  $q \in \left(\frac{8}{3},\frac{10}{3}\right)$, there is $c_{q} >0$ such that
\begin{equation}\label{eq:flLq}
\|f_{\l}^{+}\|_{L^{2}}^{2} \leq \frac{c_{q}}{(\l-k)^{q-2}}\,\left\|\langle \cdot \rangle^{s}f_{k}^{+}\right\|_{L^{1}}^{\frac{10}{3}-q}\,\|f^{+}_{k}\|_{L^{2}}^{2(q-\frac{8}{3})}\,\left\|\nabla \left(\langle \cdot \rangle^{\frac{\g}{2}}\,f^{+}_{k}\right)\right\|_{L^{2}}^{2}, \quad \quad 0 \leq k < \l\,,
\end{equation}
with $s=-\frac{3\g}{10-3q} > -\frac{3}{2}\g$.
\end{lem}

\begin{proof} The proof is based on the interpolation inequality
\begin{equation}\label{int-ineq}
\|\langle \cdot \rangle^{a}g\|_{L^{r}} \leq \|\langle \cdot \rangle^{a_{1}}g\|_{L^{r_{1}}}^{\theta}\,\|\langle \cdot \rangle^{a_{2}}g\|_{L^{r_{2}}}^{1-\theta}\,,
\end{equation}
with 
$$\frac{1}{r}=\frac{\theta}{r_{1}}+\frac{1-\theta}{r_{2}}, \quad a=\theta\,a_{1}+(1-\theta)a_{2},  \quad \theta \in (0,1).$$
Moreover, for the special case $r_{1}=6$, $r_{2}=2$, $a_{1}=\frac{\g}{2}$ and $r \in (2,6)$, thanks to Sobolev embedding, the identity will become
\begin{equation}\begin{split}\label{eq:interpol}
\|\langle \cdot \rangle^{a}g\|_{L^{r}} &\leq C_{\theta}\,\left\|\nabla\left(\langle \cdot \rangle^{\frac{\g}{2}}g\right)\right\|_{L^{2}}^{\theta}\,\|\langle \cdot \rangle^{a_{2}}g\|_{L^{2}}^{1-\theta},\\
\frac{1}{r}&=\frac{3-2\theta}{6}, \qquad a=\theta\,\frac{\g}{2}+(1-\theta)a_{2}, \qquad \theta \in (0,1),\qquad r \in (2,6).\end{split}\end{equation}
With these tools at hands, one has for $0 \leq k < \l$ and $r >2$, writing $r=2+2\alpha$ with \eqref{eq:alphaf},
\begin{multline*}
\|\langle \cdot \rangle^{\frac{\g}{2}}f_{\l}^{+}\|_{L^{2}}^{2}=\int_{\R^{3}}\langle v\rangle^{\g}(f^{+}_{\l}(t,v))^{2}\d v
\\\leq (\l-k)^{-2\alpha}\int_{\R^{3}}\langle v\rangle^{\g}(f^{+}_{k}(t,v))^{2+2\alpha}\d v
=(\l-k)^{-(r-2)}\left\|\langle \cdot \rangle^{\frac{\g}{r}}f^{+}_{k}(t)\right\|_{L^{r}}^{r}\,,
\end{multline*}
so that \eqref{eq:interpol} gives, with $a=\frac{\g}{r}$, 
$$\|\langle \cdot \rangle^{\frac{\g}{2}}f_{\l}^{+}\|_{L^{2}}^{2} \leq C(\l-k)^{-(r-2)}\left\|\nabla \left(\langle \cdot \rangle^{\frac{\g}{2}}f_{k}^{+}(t)\right)\right\|_{L^{2}}^{r\theta}\,\left\|\langle \cdot \rangle^{a_{2}} f^{+}_{k}\right\|_{L^{2}}^{r(1-\theta)}, $$
with $\theta=\frac{3r-6}{2r}$ and $a_{2}=\frac{\g}{2}\frac{10-3r}{6-r}$. One picks then $r=\frac{10}{3}$ so that $a_{2}=0$ and $r\theta=2$, to obtain \eqref{eq:flL2}. One proceeds in the same way to estimate $\|\langle \cdot\rangle^{\g}f_{\l}^{+}\|_{L^{p}}^{p}$. Namely, for $r >p$,
$$
\|\langle \cdot \rangle^{\g}f_{\l}^{+}\|_{L^{p}}^{p} \leq (\l-k)^{-(r-p)}\left\|\langle \cdot\rangle^{\frac{\g\,p}{r}}f^{+}_{k}\right\|_{L^{r}}^{r}
$$
and, with $r >2p$, imposing in \eqref{eq:interpol}  $a_{2}=0$ and $a=\frac{\g\,p}{r}$, we get $\theta=\frac{2p}{r}$ and 
$$\|\langle \cdot \rangle^{\g}f_{\l}^{+}\|_{L^{p}}^{p} \leq C(\l-k)^{-(r-p)}\left\|\nabla \left(\langle \cdot \rangle^{\frac{\g}{2}}f_{k}^{+}\right)\right\|_{L^{2}}^{2p}\,\left\|f_{k}^{+}\right\|_{L^{2}}^{r-2p}, $$ 
which gives \eqref{eq:flLp} when $r=2+\frac{4p}{3}$ (notice that $r >2p$ since $p <3$). 

Let us now prove \eqref{eq:flLq}. Let us consider first $q >2$ and use \eqref{int-ineq}. One has
$$\|g\|_{L^{q}}\leq \|\langle \cdot \rangle^{s}\,g\|_{L^{1}}^{\theta_{1}}\,\|g\|_{L^{2}}^{\theta_{2}}\,\|\langle \cdot \rangle^{\frac{\g}{2}}g\|_{L^{6}}^{\theta_{3}}, $$
with $\theta_{i} \geq 0$ $(i=1,2,3)$ such that
$$\theta_{1}+\theta_{2}+\theta_{3}=1, \qquad s\,\theta_{1}+0\cdot \theta_{2}+\frac{\g}{2}\theta_{3}=0, \qquad \frac{\theta_{1}}{1}+\frac{\theta_{2}}{2}+\frac{\theta_{3}}{6}=\frac{1}{q}.$$
Imposing $q\theta_{3}=2$, this easily yields
$$q\theta_{1}=\frac{10}{3}-q, \qquad q\theta_{2}=2\left(q-\frac{8}{3}\right), \qquad s=-\frac{3\g}{10-3q}, \qquad q \in \left(\frac{8}{3},\frac{10}{3}\right).$$
Using Sobolev inequality, this gives, for any $q \in \left(\frac{8}{3},\frac{10}{3}\right)$,
 the existence of a positive constant $C >0$ such that
$$\|g\|_{L^{q}}^{q} \leq C\,\|\langle \cdot \rangle^{s}g\|_{L^{1}}^{\frac{10}{3}-q}\,\|g\|_{L^{2}}^{2(q-\frac{8}{3})}\,\left\|\nabla \left(\langle \cdot \rangle^{\frac{\g}{2}}\,g\right)\right\|_{L^{2}}^{2}, \qquad s=-\frac{3\g}{10-3q}.$$
Using then \eqref{eq:alphaf}, for any $q >2$, one has $\|f_{\l}^{+}\|_{L^{2}}^{2} \leq (\l-k)^{2-q}\,\|f_{k}^{+}\|_{L^{q}}^{q}$ for $0 \leq k < \l$,
 and the above inequality gives the result.\end{proof}

Let us now introduce, for any measurable $f:=f(t,v) \ge 0$ and  $\ell \geq 0$, the energy functional
$$\mathscr{E}_{\ell}(T_{1},T_{2})=\sup_{t \in [T_{1},T_{2})}\left(\frac{1}{2} \left\|f_{\l}^{+}(t)\right\|_{L^{2}}^{2} + c_{0}\int_{T_{1}}^{t}\left\|\nabla \left(\langle \cdot \rangle^{\frac{\g}{2}}\,f^{+}_{\ell}(\tau)\right)\right\|_{L^{2}}^{2}\d \tau\right), \qquad 0 \leq T_{1} \leq T_{2}.$$
We have then the fundamental result for the implementation of the level set analysis.

\begin{prop}\label{main-energy-functional}
Assume that $-2 < \g < 0$  and let a nonnegative initial datum $f_{\mathrm{in}}$ satisfying \eqref{hypci}--\eqref{eq:Mass} for some $\dd_0 >0$ be given. For $\dd \in (0,\dd_0]$, let  $f(t,\cdot)$ be a weak-solution to \eqref{LFD}. Then, for any $p_{\g} \in (1,3)$ and any  $s > \frac{3}{2}|\g|$, there exist some positive constants $ C_{1},C_{2}$ depending only on $s$, $\|f_{\rm in}\|_{L^1_2}$ and $H(f_{\rm in})$ such that, for any times $0 \leq T_{1} < T_{2} \leq T_{3}$ and $0 \leq k < \ell$, 
\begin{multline}\label{eq:ElT2T3}
\mathscr{E}_{\l}(T_{2},T_{3}) \leq \frac{{C}_{2}}{T_{2}-T_{1}}(\l-k)^{{-\frac{4s+3\g}{3s}}}\left[\sup_{\tau \in [T_{1},T_{3}]}\lm_{s}(\tau)\right]^{{\frac{|\g|}{s}}}\,\mathscr{E}_{k}(T_{1},T_{3})^{{\frac{5s+3\g}{3s}}}
\\
+{C}_{1}\left(\mathscr{E}_{k}(T_{1},T_{3})\right)^{\frac{1}{p_{\g}}+\frac{2}{3}}\left(\l-k\right)^{-\frac{2}{p_{\g}}-\frac{1}{3}}
\\
\times \left(\l+\left[{(\l-k)^{\frac{2}{p_{\g}}-1}+\l(\l-k)^{\frac{2}{p_{\g}}-2}}\right]\mathscr{E}_{k}(T_{1},T_{3})^{1-\frac{1}{p_{\g}}}\right)\,. 
\end{multline} \end{prop}

\begin{proof}
Let us fix $0 \leq T_{1} < T_{2} \leq T_{3}$. Integrating inequality \eqref{eq:13Ric} over $(t_{1},t_{2})$, we obtain that
\begin{multline*}
\frac{1}{2}\|f_{\l}^{+}(t_{2})\|_{L^{2}}^{2} + c_{0}\int_{t_{1}}^{t_{2}} \big\| \nabla \big(\langle \cdot \rangle^{\frac{\g}{2}}f_{\l}^{+}(\tau)\big) \big\|^{2}_{L^{2}} \d\tau \leq \frac{1}{2}\|f_{\l}^{+}(t_{1})\|_{L^{2}}^{2} \\
+ C_{0}\int_{t_{1}}^{t_{2}}\|\langle \cdot \rangle^{\frac{\g}{2}}f_{\l}^{+}(\tau)\|_{L^{2}}^{2}\d \tau - \l \int_{t_{1}}^{t_{2}}\d\tau\int_{\R^{3}}\bm{c}_{\g}[f](\tau,v)\,f_{\l}^{+}(\tau,v)\d v.
\end{multline*}
Thus, if $T_{1}\leq t_{1} \leq T_{2} \leq t_{2}\leq T_{3}$, one first notices that the above inequality implies that
\begin{multline*}
\frac{1}{2}\|f_{\l}^{+}(t_{2})\|_{L^{2}}^{2} + c_{0}\int_{T_{2}}^{t_{2}} \big\| \nabla \big(\langle \cdot \rangle^{\frac{\g}{2}}f_{\l}^{+}(\tau)\big) \big\|_{L^{2}}^{2} \d\tau \leq \frac{1}{2}\|f_{\l}^{+}(t_{1})\|_{L^{2}}^{2} \\
+ C_{0}\int_{T_{1}}^{t_{2}}\|\langle \cdot \rangle^{\frac{\g}{2}}f_{\l}^{+}(\tau)\|_{L^{2}}^{2}\d \tau - \l \int_{T_{1}}^{t_{2}}\d\tau\int_{\R^{3}}\bm{c}_{\g}[f](\tau,v)\,f_{\l}^{+}(\tau,v)\d v,
\end{multline*}
and, taking the supremum over $t_{2} \in [T_{2},T_{3}]$, one gets
\begin{multline*}
\mathscr{E}_{\l}(T_{2},T_{3}) \leq \frac{1}{2}\|f_{\l}^{+}(t_{1})\|_{L^{2}}^{2}+ C_{0}\int_{T_{1}}^{T_{3}}\|\langle \cdot \rangle^{\frac{\g}{2}}f_{\l}^{+}(\tau)\|_{L^{2}}^{2}\d \tau\\
 - \l \int_{T_{1}}^{T_{3}}\d\tau\int_{\R^{3}}\bm{c}_{\g}[f](\tau,v)\,f_{\l}^{+}(\tau,v)\d v, \qquad \forall t_{1} \in [T_{1},T_{2}].\end{multline*}
Integrating now this inequality with respect to $t_{1} \in [T_{1},T_{2}]$, one obtains
\begin{multline*}
\mathscr{E}_{\l}(T_{2},T_{3}) \leq \frac{1}{2(T_{2}-T_{1})}\int_{T_{1}}^{T_{2}}\|f_{\l}^{+}(t_{1})\|_{L^{2}}^{2}\d t_{1}+ C_{0}\int_{T_{1}}^{T_{3}}\|\langle \cdot \rangle^{\frac{\g}{2}}f_{\l}^{+}(\tau)\|_{L^{2}}^{2}\d \tau\\
 - \l \int_{T_{1}}^{T_{3}}\d\tau\int_{\R^{3}}\bm{c}_{\g}[f](\tau,v)\,f_{\l}^{+}(\tau,v)\d v.\end{multline*}
Therefore, applying Proposition \ref{Lemma-LS-1} with $\lambda=\g <0$, $g=f$ and $\varphi = f^{+}_{\l}$, we see that
\begin{multline}\label{ef-f1}
\mathscr{E}_{\l}(T_{2},T_{3}) \leq \frac{1}{2(T_{2}-T_{1})}\int_{T_{1}}^{T_{3}}\|f_{\l}^{+}(\tau)\|_{L^{2}}^{2}\d \tau+ C_{0}\int_{T_{1}}^{T_{3}}\|\langle \cdot \rangle^{\frac{\g}{2}}f_{\l}^{+}(\tau)\|_{L^{2}}^{2}\d \tau\\
+ \ell \, C_{\g,p_{\g}}(f_{\mathrm{in}})\,\int_{T_{1}}^{T_{3}} \|\langle \cdot \rangle^{\g}f_{\l}^{+}(\tau)\|_{L^{1}} \d\tau + \ell \, C_{\g,p_{\g}}(f_{\mathrm{in}})\,\int_{T_{1}}^{T_{3}} \|\langle \cdot \rangle^{\g}f_{\l}^{+}(\tau)\|_{L^{p_{\g}}} \d\tau,
\end{multline}
for $p_{\g} >1$ such that $-\g\,q_{\g} < 3$, where $\frac{1}{p_{\g}}+\frac{1}{q_{\g}}=1.$ Notice that, since $\gamma \in(-2,0)$, any $p_{\g} \in (1,3)$ is admissible. We resort now to Lemma \ref{lem:changelevel} to estimate the last three terms on the right-hand side of \eqref{ef-f1}. Applying \eqref{eq:flL2}, one first has
\begin{equation*}
\begin{split}
\int_{T_{1}}^{T_{3}}\|\langle \cdot \rangle^{\frac{\g}{2}}f_{\l}^{+}(\tau)\|_{L^{2}}^{2}\d \tau &\leq C\,(\l-k)^{-\frac{4}{3}}\int_{T_{1}}^{T_{3}}\left\|\nabla \left(\langle \cdot \rangle^{\frac{\g}{2}}f_{k}^{+}(\tau)\right)\right\|_{L^{2}}^{2}\,\left\|f^{+}_{k}(\tau)\right\|_{L^{2}}^{\frac{4}{3}}\d\tau \\ 
&\leq \frac{C}{(\ell - k)^{\frac{4}{3}}}\sup_{t\in[T_{1},T_{3}]}\| f^{+}_{k}(t) \|^{\frac{4}{3}}_{L^{2}}\int_{T_{1}}^{T_{3}}\left\|\nabla\left(\langle \cdot \rangle^{\frac{\g}{2}}f^{+}_{k}(\tau)\right)\right\|^{2}_{L^{2}}\d \tau .\\
\end{split}
\end{equation*}
Since 
$$\sup_{t\in [T_{1},T_{3}]}\|f_{k}^{+}(t)\|_{L^{2}}^{\frac{4}{3}} \leq \left(2\mathscr{E}_{k}(T_{1},T_{3})\right)^{\frac{2}{3}} \quad \text{ and } \quad \int_{T_{1}}^{T_{3}}\left\|\nabla\left(\langle \cdot \rangle^{\frac{\g}{2}}f^{+}_{k}(\tau)\right)\right\|^{2}_{L^{2}}\d \tau \leq c_{0}^{-1}\mathscr{E}_{k}(T_{1},T_{3}),$$ by definition of the energy functional, we get
\begin{equation}\label{eq:intl2}
C_{0}\int_{T_{1}}^{T_{3}}\|\langle \cdot \rangle^{\frac{\g}{2}}f_{\l}^{+}(\tau)\|_{L^{2}}^{2}\d \tau \leq \bar{C}_{0}\,(\l-k)^{-\frac{4}{3}}\mathscr{E}_{k}(T_{1},T_{3})^{\frac{5}{3}}, 
\end{equation}
for some positive constant $\bar{C}_{0}$ depending only on $\|f_{\mathrm{in}}\|_{L^{1}_{2}}$ and $H(f_{\rm in})$. 
Similarly, using \eqref{eq:flLp} first with $p=1$ and then with $p=p_{\g}>1$, one deduces that
\begin{equation}\label{eq:intl1}\begin{split}
C_{\g,p_{\g}}(f_{\mathrm{in}})\,\int_{T_{1}}^{T_{3}} \|\langle \cdot \rangle^{\g}f_{\l}^{+}(\tau)\|_{L^{1}} \d\tau  &\leq \bar{C}_{0}(\ell -k)^{-\frac{7}{3}}\mathcal{E}_{k}(T_{1},T_{3})^{\frac{5}{3}}\,,\\
C_{\g,p_{\g}}(f_{\mathrm{in}})\,\int_{T_{1}}^{T_{3}} \|\langle \cdot \rangle^{\g}f_{\l}^{+}(\tau)\|_{L^{p_{\g}}} \d\tau  &\leq \bar{C}_{0}(\ell -k)^{-(\frac{2}{p_{\g}}+\frac{1}{3})}\mathcal{E}_{k}(T_{1},T_{3})^{\frac{1}{p_{\g}}+\frac{2}{3}}\,.
\end{split}
\end{equation}
Regarding the first term in the right-hand side of  \eqref{ef-f1}, one uses \eqref{eq:flLq} {with $q=\frac{10}{3}+\frac{\g}{s} \in \left(\frac{8}{3},\frac{10}{3}\right)$, where $s >\frac{3}{2}|\g|$ is given}, to get
\begin{multline*}
\int_{T_{1}}^{T_{3}}\|f^{+}_{\l}(\tau)\|_{L^{2}}^{2}\d \tau \leq \frac{c_{q}}{(\l-k)^{q-2}}\sup_{\tau\in [T_{1},T_{3}]}\|\langle \cdot \rangle^{s}\,f_{k}^{+}(\tau)\|_{L^{1}}^{\frac{10}{3}-q}\,\times\\
\times \int_{T_{1}}^{T_{3}}\|f_{k}^{+}(\tau)\|_{L^{2}}^{2(q-\frac{8}{3})}\left\|\nabla \left(\langle \cdot \rangle^{\frac{\g}{2}}\,f_{k}^{+}(\tau)\right)\right\|_{L^{2}}^{2}\d \tau\\
\leq \frac{\bm{c}_{q}}{(\l-k)^{q-2}}\sup_{\tau\in [T_{1},T_{3}]}\|\langle \cdot \rangle^{s}\, f_{k}^{+}(\tau)\|_{L^{1}}^{\frac{10}{3}-q}\mathscr{E}_{k}(T_{1},T_{3})^{q-\frac{5}{3}},
\end{multline*}
for some positive constant $\bm{c}_{q} >0.$ Thus
\begin{equation}\label{eq:f+lL2}
\int_{T_{1}}^{T_{3}}\|f^{+}_{\l}(\tau)\|_{L^{2}}^{2}\d \tau \leq \frac{\bm{c}_{q}}{(\l-k)^{q-2}}\left(\sup_{\tau \in [T_{1},T_{3}]}\lm_{s}(\tau)\right)^{\frac{10}{3}-q}\,\mathscr{E}_{k}(T_{1},T_{3})^{q-\frac{5}{3}}.
\end{equation}
Gathering \eqref{ef-f1}--\eqref{eq:intl1}--\eqref{eq:intl2}--\eqref{eq:f+lL2} gives the result {recalling that $q=\frac{10}{3}+\frac{\g}{s}$.}
\end{proof}
\begin{rmq}\label{nb:4third} Notice that, {for $-\frac{4}{3} < \gamma < 0$,
then, one can choose $s=2 > \frac{3}{2}|\g|$ in \eqref{eq:ElT2T3} to get $\sup_{\tau \in [T_{1},T_{3}]}\lm_{s}(\tau)=\|f_{\mathrm{in}}\|_{L^{1}_{2}}$}. For $\g \leq -\frac{4}{3}$, we will rather use \eqref{eq:ElT2T3} with the choice $s=3$.\end{rmq}

With this, we can implement the level set iteration to deduce Theorem \ref{Linfinito*}.

\begin{proof}[Proof of Theorem \ref{Linfinito*}] We first start with short times, that is, we are concerned at this point with the appearance of the norm. In all the proof, $C(f_{\mathrm{in}})$ will denote a generic constant depending only on $f_{\mathrm{in}}$ through its $L^{1}_{2}$-norm and entropy  $H(f_{\rm in})$. \ Let us fix $T>t_{*}>0$ and let $K>0$ (to be chosen sufficiently large). We consider the sequence of levels and times 
$$\l_{n}=K\,\left(1-\frac{1}{2^{n}}\right), \qquad t_{n}:=t_{*}\left(1-\frac{1}{2^{n+1}}\right), \qquad T > t_{*} >0, \qquad n \in \N.$$
We apply Proposition \ref{main-energy-functional} with $T_{3}=T$ and the choices
$$k= \ell_{n} < \ell_{n+1}=\ell\,,\quad \quad T_{1}=t_{n}<t_{n+1} =T_{2}\,,\qquad {E}_{n}:=\mathscr{E}_{\l_{n}}(t_{n},T),$$
to conclude that
\begin{multline*}{E}_{n+1} \leq 2^{{\frac{10s+3\g}{3s}}}C_{2}\frac{\bm{y}_{s}^{{\frac{|\g|}{s}}}}{K^{{\frac{4s+3\g}{3s}}}\,t_{*}}\,2^{n{\frac{7s+3\g}{3s}}}{E}_{n}^{{\frac{5s+3\g}{3s}}}+\frac{C_{1}}{K^{\frac{2}{p_{\g}}-\frac{2}{3}}}\,{E}_{n}^{\frac{1}{p_{\g}}+\frac{2}{3}}\,2^{\frac{6+p_{\g}}{3p_{\g}}(n+1)}\\
+{\frac{C_{1}}{K^{\frac{4}{3}}}\,E_{n}^{\frac{5}{3}}\,2^{\frac{7}{3}(n+1)}(1+2^{-(n+1)})}\,,\end{multline*}
that is,
\begin{multline}\label{DeG-ineq}
{E}_{n+1}\leq 2^{{\frac{10s+3\g}{3s}}}C_{2}\frac{\bm{y}_{s}^{{\frac{|\g|}{s}}}}{K^{{\frac{4s+3\g}{3s}}}\,t_{*}}\,2^{n{\frac{7s+3\g}{3s}}}{E}_{n}^{{\frac{5s+3\g}{3s}}}+\frac{C_{1}}{K^{\frac{2}{p_{\g}}-\frac{2}{3}}}\,E_{n}^{\frac{1}{p_{\g}}+\frac{2}{3}}\,2^{\frac{6+p_{\g}}{3p_{\g}}(n+1)} + \frac{2C_{1}}{K^{\frac{4}{3}}}\,E_{n}^{\frac{5}{3}}\,2^{\frac{7}{3}(n+1)},
\end{multline}
for some positive constants $C_{1},C_{2}$ depending only on $\|f_{\mathrm{in}}\|_{L^{1}_{2}}$ and $H(f_{\rm in})$ (but not on $n$), where
$$\bm{y}_{s}=\sup_{t \in [0,T)}\lm_{s}(t).$$
Notice that
\begin{equation*}\begin{split}
{E}_{0}=\mathscr{E}_{0}\left(\frac{t_{*}}{2},T\right)&\leq \frac{1}{2}\sup_{t \in [\frac{t_{*}}{2},T)}\|f(t)\|_{L^{2}}^{2}+c_{0}\int_{\frac{t_{*}}{2}}^{T}\left\|\nabla \left(\langle \cdot\rangle^{\frac{\g}{2}}f(\tau)\right)\right\|_{L^{2}}^{2}\d\tau\\
&\leq \frac{1}{2}\sup_{t\in [\frac{t_{*}}{2},T)}\lM_{0}(t)+c_{0}\int_{\frac{t_{*}}{2}}^{T}\lD_{\g}(\tau)\d\tau\,,
\end{split}\end{equation*}
so that Proposition \ref{prop:boundedL2} together with Corollary \ref{cor:L2Ms} ensure that 
$${E}_0 \leq C(f_{\mathrm{in}})\big( T-\tfrac{ t_{*} }{2} + {t_{*}^{-\frac{3}{2}}}\big)\,.$$ 
We look now for a choice of the parameters  {$K$ and $Q >0$} ensuring that  the sequence $(E_{n}^{\star})_{n}$ defined by
$${E}^{\star}_{n}:={E}_0\,Q^{-n}, \qquad n \in \N\,,$$
satisfies \eqref{DeG-ineq} with the reversed inequality. Notice that 
\begin{multline}\label{deG-ineqstar}
E_{n+1}^{\star} 
\geq 2^{{\frac{10s+3\g}{3s}}}C_{2}\frac{\bm{y}_{s}^{{\frac{|\g|}{s}}}}{K^{{\frac{4s+3\g}{3s}}}\,t_{*}}\,2^{n{\frac{7s+3\g}{3s}}}\left(E_{n}^{\star}\right)^{{\frac{5s+3\g}{3s}}}\\
+\frac{C_{1}}{K^{\frac{2}{p_{\g}}-\frac{2}{3}}}\,\left(E_{n}^{\star}\right)^{\frac{1}{p_{\g}}+\frac{2}{3}}\,2^{\frac{6+p_{\g}}{3p_{\g}}(n+1)} + \frac{2C_{1}}{K^{\frac{4}{3}}}\,\left(E_{n}^{\star}\right)^{\frac{5}{3}}\,2^{\frac{7}{3}(n+1)}\end{multline}
is equivalent to 
\begin{multline*}
1 \geq \frac{2^{{\frac{10s+3\g}{3s}}}\,C_{2}}{K^{ {\frac{4s+3\g}{3s}}}\,t_{*}}\bm{y}_{s}^{ {\frac{|\g|}{s}}}\,Q {E}_{0}^{ {\frac{2s+3\g}{3s}}}\,\left[Q^{ {-\frac{2s+3\g}{3s}}}\,2^{ {\frac{7s+3\g}{3s}}}\right]^{n} 
+\frac{2^{\frac{6+p_{\g}}{3p_{\g}}}C_{1}}{K^{\frac{2}{p_{\g}}-\frac{2}{3}}}\,Q\,{E}_{0}^{\frac{1}{p_{\g}}-\frac{1}{3}}\left[Q^{\frac{1}{3}-\frac{1}{p_{\g}}}\,2^{\frac{6+p_{\g}}{3p_{\g}}}\right]^{n}\\
+\frac{2^{\frac{10}{3}}\,C_{1}}{K^{\frac{4}{3}}}\,Q\,{E}_{0}^{\frac{2}{3}}\left[2^{\frac{7}{3}}Q^{-\frac{2}{3}}\right]^{n}.
\end{multline*}
We first choose $Q$ in a such a way that all the terms $\left[\cdots\right]^{n}$ are smaller than one, i.e.
$$Q=\max\left(2^{\frac{7}{2}},2^{\frac{6+p_{\g}}{3-p_{\g}}},2^{{\frac{7s+3\g}{2s+3\g}}}\right), $$ 
where we recall that ${s>\frac{3}{2}|\g|}$ and $p_{\g} <3.$ With such a choice, \eqref{deG-ineqstar} would hold as soon as 
\begin{equation}\label{deG-ine1}
1 \geq \frac{2^{{\frac{10s+3\g}{3s}}}\,C_{2}}{K^{{\frac{4s+3\g}{3s}}}\,t_{*}}\bm{y}_{s}^{{\frac{|\g|}{s}}}\,Q {E}_{0}^{{\frac{2s+3\g}{3s}}}\,+\frac{2^{\frac{6+p_{\g}}{3p_{\g}}}C_{1}}{K^{\frac{2}{p_{\g}}-\frac{2}{3}}}\,Q\,{E}_{0}^{\frac{1}{p_{\g}}-\frac{1}{3}}+\frac{2^{\frac{10}{3}}\,C_{1}}{K^{\frac{4}{3}}}\,Q\,{E}_{0}^{\frac{2}{3}}.
\end{equation}
This would hold for instance if each term of the sum is smaller than $\frac{1}{3}$, and  a direct computation shows that this amounts to choose
\begin{multline}\label{eq:Kt_{*}}
K \geq K(t_{*},T)=\max\left\{1,\left(3C_{2}{E}_{0}^{{\frac{2s+3\g}{3s}}}2^{{\frac{10s+3\g}{3s}}}\,Q\,t_{*}^{-1}\bm{y}_{s}^{-{\frac{\g}{s}}}\right)^{{\frac{3s}{4s+3\g}}},\right.\\
\left.\left(3C_{1} {E}_{0}^{\frac{3-p_{\g}}{3p_{\g}}}2^{\frac{6+p_{\g}}{3p_{\g}}}Q\right)^{\frac{3p_{\g}}{6-2p_{\g}}},\left(3C_{1} {E}_{0}^{\frac{2}{3}}2^{\frac{10}{3}}Q\right)^{\frac{3}{4}}\right\}.
\end{multline}
By a comparison principle  (because ${E}_{0} ={E}^{\star}_{0}$), one concludes that
\begin{equation*}
E_{n} \leq E_{n}^{\star}\,, \qquad n \in \N\,,
\end{equation*}
and in particular, since $Q >1$,
$$\lim_{n}E_{n}=0.$$
Since $\lim_{n}t_{n}=t_{*}$ and $\lim_{n}\ell_{n}=K$,  this implies that 
\begin{equation*}
\sup_{t\in[t_*,T)}\|f^{+}_{K}(t)\|_{L^{2}}= 0\,,
\end{equation*}
for $K \geq K(t_{*},T)$ and, in particular, 
\begin{equation*}\label{Linfinito}
\| f(t)\|_{L^{\infty}} \leq K(t_*,T)\,,\qquad  0<t_{*}\leq t < T\,.
\end{equation*}
Recall that $K(t_{*},T)$ as defined in \eqref{eq:Kt_{*}} is the maximum of four terms that we denote here $K_{i}(t_{*},T)$ $(i=1,2,3,4)$ with of course $K_{1}(t_{*},T)=1$. We estimate it roughly by the sum of these four terms, i.e.
$$K(t_{*},T) \leq 1+\sum_{i=2}^{4}K_{i}(t_{*},T)\,,$$
and notice that the dependence with respect to $T,t_{*}$ and $t$ is encapsulated in the term $E_{0}$ and $K_{2}(t_{*},T)$ (through  $t_{*}^{-1}$ and $\bm{y}_{s}$). One has easily
$$K_{2}(t_{*},T) \leq c_{2} E_{0}^{{\frac{2s+3\g}{4s+3\g}}}t_{*}^{-{\frac{3s}{4s+3\g}}}\bm{y}_{s}^{{\frac{3|\g|}{4s+3\g}}}, \qquad K_{3}(t_{*},T) + K_{4}(t_{*},T) \leq c_{3} E_{0}^{\frac{1}{2}}\,,$$
for some positive constants $c_{2},c_{3} >0$ depending on $Q,s,p_{\g}$ and $\|f_{\rm in}\|_{L^{1}_{2}}, H(f_{\rm in})$ (through $C_{1},C_{2}$). 

{Noticing that $E_{0}$ is bounded away from zero (by some constant independent of $t_{*},T$)\footnote{
Indeed, for any $t \geq 0$ and any $R >0$, a simple use of Cauchy-Schwarz inequality yields
\begin{multline*}
\lM_{0}(t) \geq \int_{\{|v| \leq R\}}f^{2}(t,v)\d v \geq \frac{1}{|B(0,R)|}\left(\int_{B(0,R)}f(t,v)\d v\right)^{2} 
\geq \frac{1}{|B(0,R)|}\left(1-\int_{|v|\geq R}f(t,v)\d v\right)^{2}
\end{multline*}
where $|B(0,R)|$ is the volume of the euclidean ball centred in $0$ and radius $R >0.$ Since moreover $\int_{|v|\geq R}f(t,v)\d v \leq R^{-2}\int_{\R^{3}}f(t,v)|v|^{2}\d v=3R^{-2}$ one sees that, picking $R >0$ large enough, $\lM_{0}(t) \geq c_{R} > 0$ for any $t \geq0$. In turn, $E_{0} >\tfrac{1}{2}c_{R} >0$.} and that ${E}_0\leq C(f_{\mathrm{in}})\big(T-\frac{t_{*}}{2} +  {t_{*}^{-\frac{3}{2}}}\big)$, since ${\frac{2s+3\g}{4s+3\g}} <\frac{1}{2}$,}
we can derive the estimate
$$K(t_{*},T) \leq C\left(1+t_{*}^{-{\frac{3s}{4s+3\g}}}\right)\sqrt{ T - \tfrac{t_{*}}{2} + {t_{*}^{-\frac{3}{2}}} }\,\bm{y}_{s}^{ {\frac{3|\g|}{4s+3\g}}}\,,$$
for some positive constant $C$ depending on $Q,s,p_{\g}$ and the constants $C_{1},C_{2}$ appearing in \eqref{eq:Kt_{*}}.
Thus, taking $0<t_{*}<T=2$, we obtain the result in the time interval $(0,2]$.

\smallskip
\noindent
For $T\geq2$, we copycat the previous argument with the increasing sequence of times
\begin{equation*}
0 <  T - \tfrac{3}{2}=t_{0} \leq t_{n} = T - 1 - \frac{1}{2^{n+1}}\,,\qquad n\in\mathbb{N}\,.
\end{equation*}
In this case the first term in the right-hand side of \eqref{DeG-ineq} can be replaced with (since no dependence upon $t_{*}$ appears)
$$ 2^{{\frac{10s+3\g}{3s}}}C_{2}\frac{\bm{y}_{s}^{{\frac{|\g|}{s}}}}{K^{{\frac{4s+3\g}{3s}}}}\,2^{n{\frac{7s+3\g}{3s}}}{E}_{n}^{{\frac{5s+3\g}{3s}}}\,.$$
Furthermore, $\lim_{n}t_{n}=T-1$ and, by Corollary  \ref{cor:L2Ms},
$${E}_0\leq C(f_{\mathrm{in}})\big(T - t_{0} +1)=C(f_{\mathrm{in}})\big( T - \big( T - \tfrac{3}{2} \big)+1\big) = \tfrac{5}{2}C(f_{\mathrm{in}})\,.$$
Consequently,
$$\sup_{\tau\in[T-1,T]}\| f(\tau)\|_{L^{\infty}} \leq K\leq C(f_{\mathrm{in}}) \bm{y}_{s}^{{\frac{3|\g|}{4s+3\g}}}
\,.$$
The result follows since $T\geq2$ is arbitrary. 
\end{proof}
A simple consequence of the above is the following:

\begin{cor}\label{cor:Linfty} Assume that  $-\frac{4}{3} < \g < 0$ and let a nonnegative initial datum $f_{\mathrm{in}}$ satisfying \eqref{hypci}--\eqref{eq:Mass} for some $\dd_0 >0$ be given. For $\dd \in (0,\dd_0]$, let  $f(t,\cdot)$ be a weak-solution to \eqref{LFD}.  
Then, there is a constant $C>0$ depending only on  $\|f_{\rm in}\|_{L^{1}_{2}}$ and $H(f_{\rm in})$  such that, for any  $t_{*} >0$,
\begin{equation}\label{eq:Linfty}
\sup_{t \geq t_{*}}\left\|f(t)\right\|_{L^{\infty}} \leq C\,\big( 1+ {t_{*}^{-\frac{6}{8+3\g}-\frac{3}{4}}} \big)\,.\end{equation}
In particular, there exists some explicit $\dd^{\dagger}$ and $\kappa_{0}$ both depending only on  $\|f_{\rm in}\|_{L^{1}_{2}}$ and $H(f_{\rm in})$  such that, for any $\dd \in [0,\dd^{\dagger}]$, 
\begin{equation}\label{eq:kap1}
\inf_{v \in \R^{3}}\left(1-\dd f(t,v)\right) \geq \kappa_{0} >0, \qquad t \geq 1.\end{equation}
\end{cor}
\begin{proof} The proof is a direct consequence of Theorem \ref{Linfinito*} (cf. also Remark \ref{nb:4third}) since, for $-\frac{4}{3} < \g <0$, {we can pick $s=2$ and $\sup_{t \in [0,T]}\lm_{s}(t)=\|f_{\mathrm{in}}\|_{L^{1}_{2}}$ is independent of $T$}.
\end{proof}

\section{Long-time behaviour: algebraic convergence result}\label{sec:converge}
We investigate now the long-time behaviour of solutions to \eqref{LFD}. Our approach is based upon the entropy/entropy dissipation method. 

\subsection{General strategy and estimates}\label{sec:gene6} In this section, for any $\eta \in \R$, we will denote by $\mathscr{D}_{\dd}^{(\eta)}(g)$ the entropy  {production} associated to the interaction kernel $\Psi(z)=|z|^{\eta+2}$, i.e.
\begin{equation}\label{defdeta}
\mathscr{D}_{\dd}^{(\eta)}(g):=\frac{1}{2}\int\int_{\R^{3}\times\R^{3}}\,|v-\vet|^{\eta+2}\bm{\Xi}_{\dd}[g](v,\vet)\d v\d\vet\,,
\end{equation}
where $\bm{\Xi}_{\dd}[g](v,\vet)$ is defined by \eqref{eq:Xidd}. We recall the following result from a previous contribution \cite{ABDL-entro}.
\begin{theo}\label{theo:main-entro} Assume that $0 \leq g \leq \dd^{-1}$ is such that
\begin{equation} \label{eq0}
 \int_{\R^{3}} g(v)\,\d v = 1, \quad \int_{\R^{3}} g(v)\,v_i\,\d v = 0 \quad (i=1,2,3)\,, \quad \int_{\R^{3}} g(v)\,|v|^2\,\d v = 3\,, 
\end{equation}
and let
\begin{equation}\label{eq1}
\kappa_{0} := \kappa_{0}(g)=\inf_{v\in\R^{3}}(1-\dd\,g(v)) >0.\end{equation}
For any $\eta\geq 0$,
$$\mathscr{D}_{\dd}^{(\eta)}(g) \geq 2\lambda_{\eta}(g)\left[b_{\dd}-\frac{12\dd^{2}}{\kappa_{0}^{4}}\max(\|g\|_{\infty}^{2},\|\M_{\dd}\|_{\infty}^{2})\right]
\mathcal{H}_{\dd}(g|\M_{\dd}), $$
where $\lambda_{\eta}(g) >0$ is given by 
\begin{equation}\label{eq:lambdag}
\frac{1}{\lambda_{\eta}(g)}:= {510}\frac{\bm{e}_{g}^{3}\,}{\kappa_{0}^{2}}\,\max(1,B_{g})\,\max\left(1,\lm_{2+\eta}(g)\right)\mathscr{I}_{\eta}(g)\,,\end{equation}
with
$$\mathscr{I}_{\eta}(g)=\sup_{v \in \R^{3}}\langle v\rangle^{\eta}\int_{\R^{3}}g(w)|w-v|^{-\eta}\langle w\rangle^{2}\d w , $$
and
$$\frac{1}{B_{g}}:=\min_{i\neq j}\inf_{\sigma \in \S^{1}}\int_{\R^{3}}\left|\sigma_{1}\frac{v_{i}}{\langle v\rangle}-\sigma_{2}\frac{v_{j}}{\langle v\rangle}\right|^{2}g(v)\d v, \qquad \frac{1}{\bm{e}_{g}}=\min_{i}\tfrac{1}{3}\int_{\R^{3}} g(v)\,v_{i}^2\,\d v\,.$$
Recall that $\M_{\dd}$ and $b_{\dd}$ are introduced in Definition \ref{defi:FDstats}.\end{theo}

Our approach is based on the interpolation between the entropy production
 with parameter $\g$ and the entropy production with parameter $\eta \geq0.$ Namely, for a given $g$ satisfying \eqref{eq0}, a simple consequence of H\"older's inequality is that
$$\mathscr{D}^{(0)}_{\dd}(g) \leq \Big( \mathscr{D}^{(\gamma)}_{\dd}(g) \Big)^{ \frac{\eta}{\eta-\gamma} }\;  \Big( \mathscr{D}^{(\eta)}_{\dd}(g) \Big)^{ \frac{-\gamma}{\eta-\gamma} }\,, \qquad \eta >0\,, \quad \g < 0\,,$$
or equivalently,
\begin{equation}\label{eq:interpolDe}
\mathscr{D}^{(\g)}_{\dd}(g) \geq \left(\mathscr{D}_{\dd}^{(0)}(g)\right)^{1-\frac{\g}{\eta}}\,\left(\mathscr{D}^{(\eta)}_{\dd}(g)\right)^{\frac{\g}{\eta}}.
\end{equation}
Noticing that $1-\frac{\g}{\eta} >0$, we can invoke Theorem \ref{theo:main-entro} to bound from below $\mathscr{D}_{\dd}^{(0)}(g)$ in terms of $\mathcal{H}_{\dd}(g|\M_{\dd})$, and we need to deduce an upper bound for $\mathscr{D}_{\dd}^{(\eta)}(g)$. 
We begin with the lower bound of $\mathscr{D}_{\dd}^{(0)}(f(t))$ for solutions to \eqref{LFD}, which can be deduced from Theorem \ref{theo:main-entro}.

\begin{prop}\label{prop:D0ddgen} Assume that $-2 < \g < 0$  and let a nonnegative initial datum $f_{\mathrm{in}}$ satisfying \eqref{hypci}--\eqref{eq:Mass} for some $\dd_0 >0$ be given. For $\dd \in (0,\dd_0]$, let  $f(t,\cdot)$ be a weak-solution to \eqref{LFD}. Then, there exist {$\dd_{1} \in (0,\dd_{0}]$ and} a positive constant $\bar{C}_{1} >0$ \emph{depending only on}  $\|f_{\mathrm{in}}\|_{L^{1}_{2}}$ and $H(f_{\rm in})$ such that 
$$\mathscr{D}_{\dd}^{(0)}(f(t)) \geq \bar{C}_{1}\left(1-98\dd \chi(t)\right)\,\mathcal{H}_{\dd}(f(t)|\M_{\dd})\,, \qquad t\geq0\, ,\qquad {\dd \in (0,\dd_{1}]},$$
where
$$\chi(t):=\max\left(\|f(t)\|_{L^{\infty}},\|\M_{\dd}\|_{L^{\infty}}\right) \in \left(0,\dd^{-1}\right), \qquad t \geq 0.$$
\end{prop} 

\begin{proof} From Theorem \ref{theo:main-entro}, there is some universal constant $c >0$ such that
$$\mathscr{D}_{\dd}^{(0)}(f(t)) \geq 2c\lambda_{0}(t)\left[b_{\dd}-\frac{12\dd^{2}}{\kappa_{0}(t)^{4}}\max\left(\|f(t)\|_{L^{\infty}}^{2},\|\M_{\dd}\|_{L^{\infty}}^{2}\right)\right]\mathcal{H}_{\dd}(f(t)|\M_{\dd})\,, \qquad t\geq0\,,$$
with
$$\kappa_{0}(t)=\inf_{v \in \R^{3}}(1-\dd f(t,v)), \qquad t \geq 0\,,$$
and
$$\lambda_{0}(t)^{-1}=\|f_{\mathrm{in}}\|_{L^{1}_{2}}^{2}\frac{\bm{e}(t)^{3}}{\kappa_{0}^{2}(t)}\max(1,B(t))\,.$$ 
Here,
$$\frac{1}{B(t)}:=\min_{i\neq j}\inf_{\sigma \in \S^{1}}\int_{\R^{3}}\left|\sigma_{1}\frac{v_{i}}{\langle v\rangle}-\sigma_{2}\frac{v_{j}}{\langle v\rangle}\right|^{2}f(t,v) \, \d v, \qquad \frac{1}{\bm{e}(t)}=\min_{i}\tfrac{1}{3}\int_{\R^{3}} f(t,v)\,v_{i}^2\,\d v, $$
since 
$$\max(1,\lm_{2}(f(t)))\mathscr{I}_{0}(f(t))=\max(1,\lm_{2}(t))\, \lm_{2}(t) \leq \|f_{\mathrm{in}}\|_{L^{1}_{2}}^{2}, \qquad t\geq 0\,,$$ by conservation of energy and because $\|f_{\mathrm{in}}\|_{L^{1}_{2}} \geq 1$.   As shown in \cite[{Remarks 2.10 \& 2.11}]{ABDL-entro}, there is a positive constant $C_{0} >0$ depending only on $\|f_{\mathrm{in}}\|_{L^{1}_{2}}$ and $H(f_{\mathrm{in}})$ such that $\min\left(\frac{1}{B(t)},\frac{1}{\bm{e}^{3}(t)}\right) \geq C_{0}$ for any $t\geq0.$ Therefore, there is a positive constant $\bar{C}_{0} >0$ depending only on $\|f_{\mathrm{in}}\|_{L^{1}_{2}}$ and $H(f_{\mathrm{in}})$ such that
$$2c\lambda_{0}(t) \geq C_{0}^{2}\frac{\kappa_{0}(t)^{2}}{\|f_{\mathrm{in}}\|_{L^{1}_{2}}^{2}} \geq \bar{C}_{0}\kappa_{0}(t)^{2}, \qquad t\geq0\,,$$
and, since $\kappa_{0}(t) \leq 1$ and $b_{\dd} \geq \frac{1}{8}$ for $\dd$ small enough (see \cite[Lemma A.1]{ABL}), we easily deduce that
\begin{equation}\label{eq:D0ddka}
\mathscr{D}_{\dd}^{(0)}(f(t)) \geq \bar{C}_{1}\,\left[\kappa_{0}(t)^{4}- 96\dd^{2}\,\max\left(\|f(t)\|_{L^{\infty}}^{2},\|\M_{\dd}\|_{L^{\infty}}^{2}\right)\right]\mathcal{H}_{\dd}(f(t)|\M_{\dd}),
\end{equation}
for any $t\geq0$ with $\bar{C}_{1}=\tfrac{1}{8}\bar{C}_{0}.$ Since $\kappa_{0}(t)=1-\dd \|f(t)\|_{L^{\infty}}$,
 one has $\kappa_{0}(t) \geq 1-\dd \chi(t)$ for any $t\geq0$ and \eqref{eq:D0ddka} becomes
$$\mathscr{D}_{\dd}^{(0)}(f(t)) \geq \bar{C}_{1}\,\left[\left(1-\dd \chi(t)\right)^{4}-96\dd^{2}\chi(t)^{2}\right]\mathcal{H}_{\dd}(f(t)|\M_{\dd}), \qquad t\geq0.$$
Expanding $(1-\dd \chi(t))^{4}$ and noticing that $-\dd^{3}\chi^{3}(t) \geq -\dd^{2}\chi^{2}(t) \geq -\dd\chi(t)$ because $\dd \chi(t)\leq1$, one gets the result.
\end{proof}
We now derive an upper bound for $\mathscr{D}_{\dd}^{(\eta)}(g)$.  A first observation is the following technical estimate.

\begin{lem}\label{lem:Ds-est}
For any $0 \leq g \leq \dd^{-1}$ satisfying \eqref{eq1} and any $\eta \geq -2$, one has
\begin{equation}\label{Ds:est}
\mathscr{D}^{(\eta)}_{\dd}(g) \leq \frac{2^{\frac{\eta+8}{2}}}{\kappa_{0}(g)}\,\|g\|_{L^{1}_{\eta+2}}\int_{\R^{3}}\langle v \rangle^{\eta+2}\big| \nabla \sqrt{g}\big|^{2}\d v\,,
\end{equation}
where we recall that $\kappa_{0}(g)=\inf_{v \in \R^{3}}\left(1-\dd g(v)\right)=1-\dd \|g\|_{L^{\infty}}.$
\end{lem}

\begin{proof} Using definition \eqref{eq:Xidd}, one has
$$\mathscr{D}_{\dd}^{(\eta)}(g)=\frac{1}{2}\int_{\R^{6}}|v-\vet|^{\eta+2}\,g\,g_{\ast}(1-\dd g)\,(1-\dd g_{\ast})\left|\Pi(v-\vet)\left[\nabla h-\nabla h_{\ast}\right]\right|^{2}\d v\d\vet , $$
where $h(v)=\log(g(v))-\log(1-\dd g(v)).$ Using the obvious estimate 
$$\left|\Pi(v-\vet)\left[\nabla h-\nabla h_{\ast}\right]\right|^{2} \leq 2|\nabla h|^{2}+2|\nabla h_{\ast}|^{2}\,,$$ 
one has
\begin{equation*}\begin{split}
\mathscr{D}_{\dd}^{(\eta)}(g) &\leq 2 \int_{\R^{6}}|v-\vet|^{\eta+2}\,g\,g_{\ast}(1-\dd g)\,(1-\dd g_{\ast})\left|\frac{\nabla g(v)}{g(v)(1-\dd g(v))}\right|^{2}\d v\d\vet\\
&\leq 2\int_{\R^{3}}\frac{|\nabla g(v)|^{2}}{g(1-\dd g)}\d v\int_{\R^{3}}|v-\vet|^{\eta+2}\,g_{\ast}\d \vet.
\end{split}\end{equation*}
Using the fact that $|v-\vet|^{\eta+2} \leq 2^{\frac{\eta+2}{2}}\langle v\rangle^{\eta+2}\langle \vet\rangle^{\eta+2}$,
 we get the desired result thanks to \eqref{eq1}.
\end{proof}
On the basis of estimates \eqref{Ds:est} and \eqref{eq:interpolDe},  we need to provide a uniform in time upper bound of the above weighted Fisher information  along solutions to \eqref{LFD}. {We follow the approach of \cite{fisher} and  begin with a technical Lemma:}

\begin{lem}\label{lem:flogf} Assume that $-2 < \g < 0$  and let a nonnegative initial datum $f_{\mathrm{in}}$ satisfying \eqref{hypci}--\eqref{eq:Mass} for some $\dd_0 >0$ be given. For $\dd \in (0,\dd_0]$, let  $f(t,\cdot)$ be a weak-solution to \eqref{LFD}. Then, for any $s >\frac32$, there exists $C_{s}(f_{\mathrm{in}}) >0$ depending on  $s$, $\|f_{\mathrm{in}}\|_{L^{1}_{2}}$ and $H(f_{\rm in})$ such that, for any $t\ge 0$ and $k \geq 0$
\begin{multline}\label{eq:cgflogf}
-\int_{\R^{3}}\langle v\rangle^{k}\bm{c}_{\g}[f(t)]\,f(t,v)\left(1+|\log f(t,v)|\right)\d v \\
\leq C_{s}(f_{\mathrm{in}})\left(\sqrt{\lm_{2k+2s}(t)}+\lM_{k}(t)+\bm{E}_{\frac{3k}{2}}(t)^{\frac{2}{3}}{\left(1+\frac{1}{t}\right)}\right)\,, 
\end{multline}
and
\begin{equation}\label{eq:flogf}
\int_{\R^{3}}\langle v\rangle^{k+\g}f(t,v)\left(1+\left|\log f(t,v)\right|\right)\d v \leq C_{s}(f_{\mathrm{in}})\left(\sqrt{\lm_{2(k+s+\g)}(t)}+\lM_{k+\g}(t)\right)\,.
\end{equation}
\end{lem}
\begin{proof} We use the following obvious estimate: for any $p,r >1$, there is $C_{p,r} >0$ such that
  \begin{equation}\label{obest}
  x\left(1+|\log x|\right) \leq C_{p,r}\left(x^{\frac{1}{r}}+x^{p}\right)\,, \qquad \forall\, x >0.
  \end{equation}
For notational simplicity, in several places we omit the dependence of $f$ with respect to $t$. Splitting $\bm{c}_{\g}[f]$ as
$$\bm{c}_{\g}[f]=-2(\g+3)\left[\left(|\cdot|^{\g}\ind_{|\cdot| \leq 1}\ast f \right) + \left(|\cdot|^{\g}\ind_{|\cdot| >1} \ast f\right)\right]=\bm{c}_{\g}^{(1)}[f]+\bm{c}_{\g}^{(2)}[f]\,,$$
one has that
\begin{multline*}
-\int_{\R^{3}}\langle v\rangle^{k}\bm{c}_{\g}[f]\,f\left(1+|\log f|\right)\d v
\leq -C_{\frac43,\frac32}\int_{\R^{3}}\langle v\rangle^{k}\bm{c}_{\g}^{(1)}[f]\left(f^{\frac{2}{3}}+f^{\frac{4}{3}}\right)\d v \\
-C_{2,2}\int_{\R^{3}}\langle v\rangle^{k}\bm{c}_{\g}^{(2)}[f]\left(\sqrt{f}+f^{2}\right)\d v\,.
\end{multline*}
Clearly, 
$$-\bm{c}_{\g}^{(2)}[f]=2(\g+3)\int_{|v-\vet| > 1}|v-\vet|^{\g}f(\vet)\d \vet \leq 2(\g+3)\|f_{\rm in}\|_{L^{1}}\,,$$
so that
$$-\int_{\R^{3}}\langle v\rangle^{k}\bm{c}_{\g}^{(2)}[f]\left(\sqrt{f}+f^{2}\right)\d v \leq 2(\g+3)\|f_{\rm in}\|_{L^{1}} \left(\int_{\R^{3}}\langle v\rangle^{k}\sqrt{f}\d v+\int_{\R^{3}}\langle v\rangle^{k}f^{2}(v)\d v\right).$$
According to Cauchy-Schwarz inequality, for any $s >\frac32$
$$\int_{\R^{3}}\langle v\rangle^{k}\sqrt{f}\d v \leq \sqrt{\lm_{2(k+s)}(f)}\|\langle \cdot \rangle^{-s}\|_{L^{2}}=C_{s}\sqrt{\lm_{2(k+s)}(f)}\,,$$
and consequently
$$-\int_{\R^{3}}\langle v\rangle^{k}\bm{c}_{\g}^{(2)}[f(t)]\left(\sqrt{f(t,v)}+f^{2}(t,v)\right)\d v \leq C_{s}(f_{\mathrm{in}})\left(\sqrt{\lm_{2(k+s)}(t)}+\lM_{k}(t)\right), $$
for some positive constant depending only on $s $ and $\|f_{\rm in}\|_{L^{1}}$.  Using H\"older's inequality,
$$-\int_{\R^{3}}\langle v\rangle^{k}\bm{c}_{\g}^{(1)}[f]\left(f^{\frac{2}{3}}+f^{\frac{4}{3}}\right)\d v \leq \left\|\langle \cdot \rangle^{k}\left(f^{\frac{2}{3}}+f^{\frac{4}{3}}\right)\right\|_{L^{\frac{3}{2}}}\,\left\|\bm{c}_{\g}^{(1)}[f]\right\|_{L^{3}}, $$
and, according to Young's convolution inequality,
$$\left\|\bm{c}_{\g}^{(1)}[f]\right\|_{L^{3}} \leq 2(\g+3)\left\|\,|\cdot|^{\g}\ind_{|\cdot|\leq 1}\right\|_{L^{\frac{3}{2}}}\,\|f\|_{L^{\frac{3}{2}}} , $$
where $\left\|\,|\cdot|^{\g}\ind_{|\cdot|\leq 1}\right\|_{L^{\frac{3}{2}}} < \infty$ since $\frac{3}{2}\g+3>0$ (recall that $\g \in (-2,0))$.   Since 
$$\|f\|_{L^{\frac{3}{2}}} \leq \left(\lm_{0}(f)+\|f\|_{L^{2}}^{2}\right)^{\frac{2}{3}} , $$
we deduce that
\begin{equation*}\begin{split}
-\int_{\R^{3}}\langle v\rangle^{k}\bm{c}_{\g}^{(1)}[f(t)]&\left(f(t,v)^{\frac{2}{3}}+f(t,v)^{\frac{4}{3}}\right)\d v \\
&\leq C_{\gamma} \left\|\langle \cdot \rangle^{k}\left(f^{\frac{2}{3}}+f^{\frac{4}{3}}\right)\right\|_{L^{\frac{3}{2}}}\,\left(\lm_{0}(t)+\lM_{0}(t)\right)^{\frac{2}{3}}\\
&\le 4C_{\gamma}\bm{E}_{\frac{3k}{2}}(t)^{\frac{2}{3}} \left(1+\lM_{0}(t)\right)^{\frac{2}{3}}.\end{split}\end{equation*}
Using now  Proposition \ref{prop:boundedL2},  this proves \eqref{eq:cgflogf}. Now, by \eqref{obest}, one has
$$\int_{\R^{3}}\langle v\rangle^{k+\g}\,f\left(1+|\log f|\right)\d v
\leq C_{2,2}\int_{\R^{3}}\langle v\rangle^{k+\g}\left(\sqrt{f}+f^{2}\right)\d v, $$
and, proceeding as above, one obtains \eqref{eq:flogf}.
\end{proof}
{We can state now the following Proposition which is inspired by \cite{fisher}.}
\begin{prop}\label{prop:weightFish}
Assume that $-2 < \g < 0$  and let a nonnegative initial datum $f_{\mathrm{in}}$ satisfying \eqref{hypci}--\eqref{eq:Mass} for some $\dd_0 >0$ be given. For $\dd \in (0,\dd_0]$, let  $f(t,\cdot)$ be a weak-solution to \eqref{LFD}. Let  $\eta \geq \g-2.$ Assume moreover that
$$ {f_{\mathrm{in}} \in L^{1}_{2{\eta+8-2\g}}(\R^{3})}.$$
Then, for any $t_{0} >0$, there exists $C >0$ depending on {$\eta,t_{0}$} and on  $f_{\mathrm{in}}$ through $\|f_{\mathrm{in}}\|_{L^1_{2{\eta+8-2\g}}}$   such that 
\begin{equation}\label{eq:weighFish}
\int_{t_{0}}^{t}\d \tau \int_{\R^{3}}\langle v\rangle^{ \eta+2}\left|\nabla \sqrt{f(\tau,v)}\right|^{2}\d v \leq C(1+t)^{2}\,,  \qquad 0\leq t_{0} < t.\end{equation}
In particular, for {$\eta \geq 0$, there is $C_{\eta}(f_{\mathrm{in}})$ depending only on $\|f_{\mathrm{in}}\|_{L^{1}_{2}}$, $H(f_{\rm in})$  and $\eta$} and such that
\begin{equation}\label{eq:Dddeta}
\int_{t_{0}}^{t}\mathscr{D}_{\dd}^{(\eta)}(f(\tau))\d\tau \leq C_{\eta}(f_{\mathrm{in}})\left[\sup_{t_0\le \tau\le t} \frac{\bm{m}_{\eta+2}(\tau)}{\kappa_{0}(\tau)}\right]\,(1+t)^{2} \,,\qquad 0\leq t_{0} < t\,, 
\end{equation}
where we recall that $\kappa_{0}(\tau)=1-\dd\,\|f(\tau)\|_{L^{\infty}}$, $\tau \geq0.$
\end{prop}
\begin{proof} {Let $\eta \geq \g-2$ be fixed. Since we aim to use Lemma \ref{lem:flogf}, it will be convenient here to introduce  $k =\eta+2-\g$.  We compute, as in \cite[Proposition 1]{fisher} the evolution of 
$$S_{k}(t):=\int_{\R^{3}}\langle v\rangle^{k}f(t,v)\log f(t,v)\d v, $$
for a solution $f=f(t,v)$ to \eqref{LFD}. To simplify notations, we write $F=F(t,v)=f(1-\dd f)$. 
One can check that
\begin{equation}\label{evol:Sk}\begin{split}
\dfrac{\d}{\d t}S_{k}(t)&=\dfrac{\d}{\d t}\lm_{k}(t)+\int_{\R^{3}}\langle v\rangle^{k} \Q(f)\log f\d v \\
&=\dfrac{\d}{\d t}\lm_{k}(t)+\int_{\R^{3}}\langle v\rangle^{k}\nabla\cdot\left(\bm{\Sigma}[f]\nabla f\right)\log f\d v-\int_{\R^{3}}\langle v\rangle^{k}\nabla \cdot \left(\bm{b}[f]F\right)\log f\d v.\end{split}\end{equation}
One computes, using that $\log f \,\nabla f=\nabla \left[f\log f-f\right]$, that
\begin{multline*}
\int_{\R^{3}}\langle v\rangle^{k}\nabla\cdot\left(\bm{\Sigma}[f]\nabla f\right)\log f\d v=-\int_{\R^{3}}\langle v\rangle^{k}\bm{\Sigma}[f]\nabla f\cdot \frac{\nabla f}{f}\d v\\
+k\int_{\R^{3}}\nabla \cdot \left(\bm{\Sigma}[f]v \langle v\rangle^{k-2}\right)\,\left[f\log f-f\right]\d v.\end{multline*}
Similarly,
\begin{multline*}
\int_{\R^{3}}\langle v\rangle^{k}\nabla \cdot \left(\bm{b}[f]F\right)\log f\d v=-k\int_{\R^{3}}\langle v\rangle^{k-2}\left(\bm{b}[f]\cdot v\right)\,F\log f\d v\\
+\int_{\R^{3}}\left(f-\frac{\dd}{2}f^{2}\right)\nabla \cdot \left(\langle v\rangle^{k}\bm{b}[f]\right)\d v.\end{multline*}
 As in the proof of Lemma \ref{lem:L2-Ms}, 
$$\nabla \cdot \left(\bm{\Sigma}[f]v \langle v\rangle^{k-2}\right)=\langle v\rangle^{k-2}\bm{B}[f]\cdot v +\langle v\rangle^{k-4}\mathrm{Trace}\left(\bm{\Sigma}[f]\cdot \bm{A}(v)\right), $$
with $\bm{A}(v)=\langle v\rangle^{2}\mathbf{Id}+(k-2)\,v\otimes v,$ whereas
$$\nabla \cdot \left(\langle v\rangle^{k}\bm{b}[f]\right)=k\langle v\rangle^{k-2} \left(\bm{b}[f]\cdot v\right) + \langle v\rangle^{k}\bm{c}_{\g}[f], $$
resulting in
\begin{align}\label{eq:trueevol}
\begin{split}
\dfrac{\d}{\d t}S_{k}(t)=\dfrac{\d}{\d t}\lm_{k}(t)&-\int_{\R^{3}}\langle v\rangle^{k}\bm{\Sigma}[f] \nabla f \cdot \frac{\nabla f}{f}\d v
-\int_{\R^{3}}\langle v\rangle^{k}\bm{c}_{\g}[f]\left(f-\frac{\dd}{2}f^{2}\right)\d v\\
&+k\int_{\R^{3}}\langle v\rangle^{k-4}\mathrm{Trace}\left(\bm{\Sigma}[f]\cdot \bm{A}(v)\right)\left[f\log f-f\right]\d v\\
&+k\int_{\R^{3}}\langle v\rangle^{k-2}\left(\bm{b}[f]\cdot v\right)\left[F\log f-f+\frac{\dd}{2}f^{2}\right]\d v\\
&+ k \int_{\R^{3}}\langle v\rangle^{k-2}\left(\bm{B}[f]\cdot v\right)\left[f\log f-f\right]\d v.
\end{split}
\end{align}
From \eqref{eq:trueevol}, using the coercivity of $\bm{\Sigma}[f]$ {and the fact that $-\bm{c}_{\g}[f] \geq0$}, we get
\begin{multline}\label{eq:SkUe}
\frac{\d}{\d t}S_{k}(t)-\dfrac{\d}{\d t}\lm_{k}(t) +K_{0}\int_{\R^{3}}\langle v\rangle^{k+\g}\frac{|\nabla f|^{2}}{f}\d v \\
\leq {- \int_{\R^{3}}\langle v\rangle^{k}\bm{c}_{\g}[f]\,f\d v}
+k\int_{\R^{3}}\langle v\rangle^{k-4}\Big|\mathrm{Trace}\left(\bm{\Sigma}[f]\cdot \bm{A}(v)\right)\Big|\left|f\log f-f\right|\d v\\
+k\int_{\R^{3}}\langle v\rangle^{k-2}\left|\bm{b}[f]\cdot v\right|\left|F\log f-f+\frac{\dd}{2}f^{2}\right|\d v\\
+k\int_{\R^{3}}\langle v\rangle^{k-2}\left|\bm{B}[f]\cdot v\right|\left|f\log f-f\right|\d v.
\end{multline}
As in the proof of Proposition \ref{prop:L2}, we see that
$$\left|\mathrm{Trace}\left(\bm{\Sigma}[f]\cdot \bm{A}(v)\right)\right| \leq 9\cdot 2k\langle v\rangle^{2} \left(|\cdot|^{\g+2}\ast f\right) \leq 9\cdot 2^{\frac{\g+4}{2}}\,k\,\|f_{\mathrm{in}}\|_{L^{1}_{2}}\langle v\rangle^{\g+4}, $$ 
and, since $|\bm{B}[f]\cdot v| \leq |\bm{b}[f]\cdot v| +\dd |\bm{b}[f^2]\cdot v|$ with $\dd f^{2} \leq f$, one can check that $\frac{1}{2}|\bm{B}[f]\cdot v|$ also satisfies \eqref{eq:bm}. We deduce then from \eqref{eq:SkUe} that there exists a constant $C_{k}(f_{\mathrm{in}})>0$ depending only on $\|f_{\mathrm{in}}\|_{L^{1}_{2}}$ and $k$, such that
\begin{multline}\label{eq:SkUe1}
\frac{\d}{\d t}S_{k}(t)-\dfrac{\d}{\d t}\lm_{k}(t) +K_{0}\int_{\R^{3}}\langle v\rangle^{k+\g}\frac{|\nabla f|^{2}}{f}\d v \\
\leq  C_{k}(f_{\mathrm{in}})\int_{\R^{3}}\langle v\rangle^{k+\g}\left(\left|f\log f-f\right|+\left|F\log f-f+\frac{\dd}{2}f^{2}\right|\right)\d v\\
-C_{k}(f_{\mathrm{in}})\int_{\R^{3}}\langle v\rangle^{k}\bm{c}_{\g}[f]\,\left(\left|f\log f-f\right|+\left|F\log f-f+\frac{\dd}{2}f^{2}\right|{+f}\right)\d v.\end{multline}
Since
\begin{equation}\label{eq:FLogF}
|f \log f-f|+|F\log f-f+\frac{\dd}{2}f^{2}| \leq 2f|\log f|+ \frac{5}{2}f,
\end{equation}
we have that
\begin{multline}\label{eq:SkUe1}
\frac{\d}{\d t}S_{k}(t)-\dfrac{\d}{\d t}\lm_{k}(t) +K_{0}\int_{\R^{3}}\langle v\rangle^{k+\g}\frac{|\nabla f|^{2}}{f}\d v \\
\leq  \frac{5}{2}C_{k}(f_{\mathrm{in}})\int_{\R^{3}}\langle v\rangle^{k+\g}f\left(1+\left|\log f\right|\right)\d v\\
-\frac{7}{2}C_{k}(f_{\mathrm{in}})\int_{\R^{3}}\langle v\rangle^{k}\bm{c}_{\g}[f]\,f\left(1+\left|\log f\right|\right)\d v.\end{multline}
Using  Lemma \ref{lem:flogf} with  $s=2$, we deduce then that, for any $t_{0} >0$,
\begin{multline*}
\frac{\d}{\d t}S_{k}(t) +K_{0}\int_{\R^{3}}\langle v\rangle^{\g+k}\frac{|\nabla {f}|^{2}}{f}\d v\\ 
\leq \frac{\d}{\d t}\lm_{k}(t) +{C_{k}(f_{\mathrm{in}})\left(\sqrt{\lm_{2k+4}(t)}+\lM_{k}(t)+\bm{E}_{\frac{3k}{2}}(t)^{\frac{2}{3}}\left(1+t^{-1}\right)\right)},\\
\leq \frac{\d}{\d t}\lm_{k}(t) + C_{k}(f_{\mathrm{in}})\left(1+t\right)\,,
\qquad t >t_{0}\,,
\end{multline*} 
 where we used Theorem \ref{theo:main-moments} for the last estimate and where $C_{k}(f_{\rm in})$ now depends on $t_{0}$.  {Notice that, for $s=2k+4 >4+|\gamma|$,   Theorem \ref{theo:main-moments} provides the growth of $\lm_{2k+4}(t)$, $\lM_{k}(t)$ and $\bm{E}_{\frac{3k}{2}}(t)$ whenever $\lm_{s}(0) < \infty$. Our assumption on $f_{\rm in}$ exactly means that $\lm_{2k+4}(0)< \infty.$} Integrating this inequality over $(t_{0},t)$ yields  
$$
K_{0}\int_{t_{0}}^{t}\d\tau\int_{\R^{3}}\langle v\rangle^{\g+k}\frac{|\nabla {f(\tau,v)}|^{2}}{f(\tau,v)}\d v 
\leq 
S_{k}(t_{0})-S_{k}(t)+\lm_{k}(t) + \frac{1}{2}C_{k}(f_{\mathrm{in}})(1+t)^{2}.$$
Clearly, $S_{k}(t)$ has no sign but, according to \cite[Eq. (B.3), Lemma B.4]{fisher}, for any $\sigma >0$ there exists $C_{\sigma} >0$ such that  
$$-S_{k}(t) \leq -\int_{\R^{3}}\langle v\rangle^{k}f(t,v)|\log f(t,v)|\d v+2\lm_{k+\sigma}(t)+C_{\sigma} , $$
yielding,
{for $\sigma = 2$},
$$ K_{0} \int_{{t_0}}^{t}\d\tau\int_{\R^{3}}\langle v\rangle^{\g+k}|\nabla \sqrt{f(\tau,v)}|^{2}\d v 
\leq C_k(f_{\rm in})\,(1+t)^{2}  {+} S_{k}(t_{0}).$$ 
Let us note here that with our assumptions, one can deduce from \eqref{eq:flogf} and Theorem \ref{theo:main-moments} that $S_{k}(t_{0})$ is actually finite.  {Indeed, \eqref{eq:flogf} implies that $S_{k}(t_{0}) < \infty$ if $\lm_{2k+2r}(t_{0}) < \infty$ and $\lM_{k}(t_{0}) < \infty$ for some $r > \frac{3}{2}$. According to Theorem \ref{theo:main-moments}, this holds if $\lm_{{s}}(0) < \infty$ for $s=2k+2r > 4-\g$. As already observed, one has $\lm_{2k+2r}(0) < \infty$ with $r=2 >\frac{3}{2}$}. Recalling that $k+\g=\eta+2$, this} proves \eqref{eq:weighFish} with a positive constant $C$ depending in particular on $t_{0}$ (with $C \lesssim t_{0}^{-2}$).
We deduce then \eqref{eq:Dddeta} from \eqref{Ds:est} and \eqref{eq:weighFish}.
\end{proof}

\subsection{The case $-\frac{4}{3} < \gamma < 0$} 
We apply the result established here above to the case $\gamma \in \left(-\frac{4}{3},0\right)$. In that case Proposition \ref{prop:D0ddgen} can be stated as:
\begin{prop}\label{prop:D0dd} Assume that $-\frac{4}{3} < \g <0$  and let a nonnegative initial datum $f_{\mathrm{in}}$ satisfying \eqref{hypci}--\eqref{eq:Mass} for some $\dd_0 >0$ be given. For $\dd \in (0,\dd_0]$, let  $f(t,\cdot)$ be a weak-solution to \eqref{LFD}. Then, there exists $\dd^{\star} \in (0,\dd_{0}]$ and $\bar{\lambda}_{0} >0$ \emph{depending only on}  $\|f_{\rm in}\|_{L^{1}_{2}}$ and $H(f_{\rm in})$  such that, for $\dd \in (0,\dd^{\star}]$, 
\begin{equation}\label{eq:D0lamb}
\mathscr{D}_{\dd}^{(0)}(f(t)) \geq \bar{\lambda}_{0}\, \mathcal{H}_{\dd}(f(t)|\M_{\dd})\,, \qquad t\geq1.\end{equation}
\end{prop}

\begin{proof}  {The proof is a direct consequence of Proposition \ref{prop:D0ddgen} and Corollary \ref{cor:Linfty} since, recalling that $\sup_{\dd\in (0,\dd_{0}]}\|\M_{\dd}\|_{L^{\infty}}<\infty$ by \cite[Lemma A.1]{ABL}, one has
    $$\chi(t)=\max\left(\|f(t)\|_{L^{\infty}},\|\M_{\dd}\|_{L^{\infty}}\right) \leq C\,, \qquad t\geq1 , $$
 with $C>0$ independent of $\dd$.}
Thus, there exists $\dd^{\star} \in (0,\dd^{\dagger})$ such that $\inf_{t\geq1}\left(1-98\dd\chi(t)\right) > 0$ for any $\dd \in (0,\dd^{\star}).$
\end{proof}
 
\begin{rmq}\label{rmq:D0Lamb} The restriction $-\frac{4}{3} < \g <0$ is needed here only to ensure that $\left(1-98\dd\chi(t)\right) > 0$. Thus, the above estimate \eqref{eq:D0lamb} holds in any situation for which $ \bar{\lambda}_0=\bar{C}_1 \inf_{t\geq 1}\left(1-98\dd\chi(t)\right) >0$.\end{rmq}
This gives the following  {version of Theorem \ref{theo:main} where the assumptions on the initial datum are slightly relaxed with respect to the statement of  Theorem \ref{theo:main}:}

\begin{prop}\label{theo:case4/3} Let $-\frac{4}{3}< \g < 0$. Let $\eta >2 |\g|$ and   {$0\leq f_{\mathrm{in}} \in L^{1}_{2\eta+8+2|\g|}(\R^{3})$}  satisfying \eqref{hypci}--\eqref{eq:Mass}  for some $\dd_0 >0$. For $\dd \in (0,\dd_0]$, let  $f(t,\cdot)$ be a weak-solution to \eqref{LFD}. Then, there exists $C_{\eta}$ depending on  $\|f_{\mathrm{in}}\|_{L^{1}_{2}}$, $H(f_{\rm in})$ and $\eta >0$, and there exists $\dd^{\ddagger} \in (0,\dd_{0}]$ such that for any $\dd \in (0,\dd^{\ddagger})$,
\begin{equation}\label{eq:ratecase43}
\left\|f(t)-\M_{\dd}\right\|_{L^{1}} \leq C_{\eta}\,\left(1+t\right)^{{-\frac{\eta-2|\g|}{2|\g|}}} , \qquad \forall t \geq 1\,.\end{equation}
As a consequence, given $s > 2|\g|$, one has
$$\sup_{t\geq 1}\bm{E}_{s}(t) < \infty ,$$
provided that {$f_{\rm in} \in L^{1}_{r}$ with $r >  {\max(2s+8+2|\g|,\frac{s^{2}}{s-2|\g|})}$.} 
\end{prop}

\begin{proof} Using Proposition \ref{prop:D0dd} and \eqref{eq:interpolDe}, for any $\eta >0$,
$$\mathscr{D}^{(\g)}_{\dd}(f(t)) \geq \bar{\lambda}_{0}^{1-\frac{\g}{\eta}}\,\mathscr{D}_{\dd}^{(\eta)}(f(t))^{\frac{\g}{\eta}}\,\mathcal{H}_{\dd}(f(t)|\M_{\dd})^{1-\frac{\g}{\eta}}, \qquad t\geq1.$$
For simplicity, we set
$$A_{\eta}(t) :=\bar{\lambda}_{0}^{1-\frac{\g}{\eta}}\,\mathscr{D}_{\dd}^{(\eta)}(f(t))^{\frac{\g}{\eta}} \geq 0, \qquad \qquad \bm{y}(t) :=\mathcal{H}_{\dd}(f(t)|\M_{\dd})\,,\qquad t\geq0.$$
Since $\dfrac{\d}{\d t}\mathcal{H}_{\dd}(f(t)|\M_{\dd})=-\mathscr{D}^{(\g)}_{\dd}(f(t))$, one has
$$\frac{\d}{\d t}\bm{y}(t) + A_{\eta}(t)\bm{y}(t)^{1-\frac{\g}{\eta}} \leq 0, \qquad t\geq 1.$$
Integrating this inequality, we deduce that
$$\bm{y}(t)^{\frac{\g}{\eta}}\geq \bm{y}(1)^{\frac{\g}{\eta}}-\frac{\g}{\eta}\int_{1}^{t}A_{\eta}(\tau)\d\tau\geq \bm{y}(0)^{\frac{\g}{\eta}}-\frac{\g}{\eta}\int_{1}^{t}A_{\eta}(\tau)\d\tau, $$
i.e.
$$\mathcal{H}_{\dd}(f(t)|\M_{\dd}) \leq \left(\mathcal{H}_{\dd}(f_{\mathrm{in}}|\M_{\dd})^{\frac{\g}{\eta}}-\frac{\g}{\eta}\int_{1}^{t}A_{\eta}(\tau)\d \tau\right)^{\frac{\eta}{\g}}, \qquad t\geq1.$$
Let us estimate from below the integral of $A_{\eta}(\tau)$. One has
\begin{equation*}\begin{split}
\int_{1}^{t}A_{\eta}(\tau)\,\d \tau&=\bar{\lambda}_{0}^{1-\frac{\g}{\eta}}\int_{1}^{t}\mathscr{D}_{\dd}^{(\eta)}(f(\tau))^{\frac{\g}{\eta}}\d\tau=\bar{\lambda}_{0}^{1-\frac{\g}{\eta}}(t-1)\int_{1}^{t}\mathscr{D}_{\dd}^{(\eta)}(f(\tau))^{\frac{\g}{\eta}}\frac{\d\tau}{t-1}\\
&\geq \bar{\lambda_{0}}^{1-\frac{\g}{\eta}}(t-1)\left(\int_{1}^{t}\mathscr{D}_{\dd}^{(\eta)}(f(\tau))\frac{\d\tau}{t-1}\right)^{\frac{\g}{\eta}},
\end{split}\end{equation*}
where we used Jensen's inequality and the convexity of the mapping $x >0 \mapsto x^{\frac{\g}{\eta}}.$ Therefore, 
$$\int_{1}^{t}A_{\eta}(\tau)\d \tau \geq \left(\bar{\lambda}_{0}(t-1)\right)^{1-\frac{\g}{\eta}}\left(\int_{1}^{t}\mathscr{D}_{\dd}^{(\eta)}(f(\tau))\d\tau\right)^{\frac{\g}{\eta}},$$
and,  using \eqref{eq:Dddeta}, which holds  since  {$f_{\rm in} \in L^{1}_{2\eta + 8 - 2 \gamma}(\R^{3})$}, one gets
$$\int_{1}^{t}A_{\eta}(\tau)\, \d\tau {\ge C}\, (t-1)^{1-\frac{\g}{\eta}}{\left(1+t\right)^{\frac{3\g}{\eta}}}
 {\ge} C {t^{1+\frac{2\g}{\eta}}}, \qquad t  \ge 2,$$
for some positive constant $C$ depending on $\eta,\|f_{\mathrm{in}}\|_{L^{1}_{2}}$ and $H(f_{\rm in})$ {where we used \eqref{res19} and the fact that $\inf_{\tau \geq 1}\kappa_{0}(\tau)  >0$}. Choosing {$\eta > -2\g$}, this gives a decay rate
$$\mathcal{H}_{\dd}(f(t)|\M_{\dd}) \leq C_{\eta}(f_{\mathrm{in}})\,{t^{2+\frac{\eta}{\g}}}, $$
for all $t\geq 2$ with $C_{\eta}(f_{\rm in})$ depending on $\eta,\|f_{\rm in}\|_{L^{1}_{2}}$ and $H(f_{\rm in})$. We conclude then with   Csisz\'ar-Kullback inequality for Fermi-Dirac relative entropy  \eqref{eq:czisz}. Let us prove now the bound on $\bm{E}_{s}(t)$ for $s \geq0$. It follows from some standard arguments (see \cite{desvmou}). Namely, let $s > -2\g$ be given and let $p=\frac{s^{2}}{s+2\g} >s$. If $f_{\rm in} \in L^{1}_{r}$ with $r > \max(2s+8-2\g,p)$, {the bound \eqref{res19} in Theorem \ref{theo:main-moments} holds as well as the above \eqref{eq:ratecase43}  with $\eta=s$}. Then, for some positive $C_{s}$ depending only on $s$, $\|f_{\rm in}\|_{L^{1}_{2}}$, $H(f_{\rm in})$  and $\lm_{r}(0)$, one has 
\begin{equation*}\begin{split}
\lm_{s}(t) &\leq \|\M_{\dd}\|_{L^{1}_{s}}+\|f(t)-\M_{\dd}(t)\|_{L^{1}_{s}} \leq \|\M_{\dd}\|_{L^{1}_{s}}+\|f(t)-\M_{\dd}\|_{L^{1}}^{1-\theta}\,\|f(t)-\M_{\dd}\|_{L^{1}_{p}}^{\theta}\\
&\leq \|\M_{\dd}\|_{L^{1}_{s}}+C_{s}(1+t)^{-(1-\theta)\frac{s+2\g}{2|\g|}}\,\left(\lm_{p}(t)^{\theta}+\|\M_{\dd}\|_{L^{1}_{p}}^{\theta}\right)\\
&\leq \|\M_{\dd}\|_{L^{1}_{s}}+C_{s}(1+t)^{-(1-\theta)\frac{s+2\g}{2|\g|}}\left(\bm{C}_{p}^{\theta}\left(1+t\right)^{\theta}+\|\M_{\dd}\|_{L^{1}_{p}}^{\theta}\right),
\qquad \theta=\frac{s+2\g}{s} \in (0,1)\,,\end{split}\end{equation*}
{for any $t\geq1$ (so that \eqref{eq:ratecase43} holds)}. Using that $\|\M_{\dd}\|_{L^{1}_{s}}$ and $\|\M_{\dd}\|_{L^{1}_{p}}$ are bounded uniformly with respect to $\dd$ (see \cite[Lemma A.1]{ABL}), we deduce that  there is $c_{s} >0$ depending only on $s$,  $\|f_{\mathrm{in}}\|_{L^{1}_{2}}$, $\lm_{r}(0)$ and $H(f_{\rm in})$ but not $\dd$ such that
$$\lm_{s}(t) \leq c_s \left(1+\left(1+t\right)^{-(1-\theta)\frac{s+2\gamma}{2|\gamma|}+\theta}\right)\,, \qquad t \geq 1.$$
Since $-(1-\theta)\frac{s+2\g}{2|\g|}+\theta=0$, this proves that $\sup_{t\geq0}\lm_{s}(t) \leq 2c_{s}$. The proof is similar for the estimate of $\sup_{t\geq 1}\lM_{s}(t)$ {where we notice that $\max(p,\frac{2p-3\g}{4})=p$ which ensures the appearance of the $L^{2}$-moment $\lM_{p}(t)$ thanks to Propositions \ref{theo:boundedL2} and \ref{prop:boundedL2}.  The result follows.}
\end{proof} 
\subsection{Proof of Theorem \ref{theo:main}: the case $-2 < \gamma\leq -\frac{4}{3}$} We are in position to give here the full proof of Theorem \ref{theo:main}.  It suffices to consider the case $-2 < \gamma \leq -\frac{4}{3}$ since the case $-\frac{4}{3} < \g < 0$ has been covered by Proposition \ref{theo:case4/3}  {where in that result, $\eta=\frac{1}{2}\left(s-8+2\g\right)$}.  With respect to the proof of Proposition \ref{theo:case4/3}, we no longer have a direct control of the norm $\sup_{t\geq 1}\|f(t)\|_{\infty}$.

\smallskip
\noindent
Recall that, according to Proposition \ref{prop:D0ddgen}, there is $\bar{C}_{1} >0$ depending on $\|f_{\rm in}\|_{L^{1}_{2}}$ and $H(f_{\rm in})$ such that
$$\mathscr{D}_{\dd}^{(0)}(f(t)) \geq \bar{C}_{1}\left[1-98\,\dd \chi(t)\right]\,\mathcal{H}_{\dd}(f(t)|\M_{\dd}), \qquad  t \geq0\,,$$
with $\chi(t)=\max\Big(\|f(t)\|_{L^{\infty}}\,,\,\sup_{\dd>0}\|\M_{\dd}\|_{L^{\infty}}\Big)$.   Let us fix $T>2$ and define
\begin{equation*}
\chi^{\star}=\chi^{\star}(T):=98\sup_{t\in[1,T)}\chi(t)\,,
\end{equation*} 
so that
\begin{equation}\label{eq:D0dT-T}
\mathscr{D}_{\dd}^{(0)}(f(t)) \geq \bar{C}_{1}\left(1-\dd \chi^{\star}\right)\mathcal{H}_{\dd}(f(t)|\M_{\dd})\,, \qquad t \in [1,T).
\end{equation}
Pick $\dd:=\dd(T)$ such that 
$$1-\dd\chi^{\star} >0.$$ 
Note that the existence of such $\dd$ follows from Theorems \ref{Linfinito*} and \ref{theo:main-moments}, since ${s >\max(4-\g,-\frac{3}{2}\g)}$ with our assumptions.
The idea is to couple the \textit{a priori} estimates with the entropy method to be able to conclude that in fact these quantities are independent of $T>2$ as long as $\dd>0$ is less than some threshold depending only on the initial distribution $f_{\mathrm{in}}$.  The interpretation of this condition is that the initial distribution is not too saturated for the argument to hold.  It is an open problem to prove that the relaxation to thermal equilibrium happens with a \textit{specific rate} when $f_{\mathrm{in}}$ is very close to a saturated state even in the hard potential case, see \cite{ABL}. \smallskip

\noindent
As in the proof of Proposition \ref{theo:case4/3}, we write
$$\bm{y}(t)=\mathcal{H}_{\dd}(f(t)|\M_{\dd}), \qquad t\geq 0.$$
{Recall that we assume here that $f_{\rm in} \in L^{1}_{s}(\R^{3})$ with  {$s > 14+6|\g|.$} For notational simplicity, we 
write  {$s=2\eta+8-2\g$ with $\eta > 3+2|\g|.$} In all the sequel, we have then  
\begin{equation}\label{eq:ad}
f_{\rm in}\in L^1_{2\eta +8-2\gamma}(\R^3) \qquad \text{ with } \quad  \eta>3+2|\gamma|.\end{equation}
Using  \eqref{eq:interpolDe}, for {such a choice of $\eta$}, we deduce from \eqref{eq:D0dT-T} that
$$
\mathscr{D}^{(\g)}_{\dd}(f(t)) \geq \bar{C}_{1}^{1-\frac{\g}{\eta}}\left(1-\dd\chi^{\star}\right)^{1-\frac{\g}{\eta}}\,\left[\mathscr{D}^{(\eta)}_{\dd}(f(t))\right]^{\frac{\g}{\eta}}\,\bm{y}(t)^{1-\frac{\g}{\eta}}, \qquad t \in [1,T).$$
Recalling that $\dfrac{\d}{\d t}\bm{y}(t)=-\mathscr{D}_{\dd}^{(\g)}(f(t))$, we deduce after integration of the above inequality that 
\begin{equation}\label{eq:ytgeta}
\bm{y}(t)^{\frac{\g}{\eta}}-\bm{y}(1)^{\frac{\g}{\eta}} \geq \bar{C}(\g,\eta)\left(1-\dd\chi^{\star}\right)^{1-\frac{\g}{\eta}}
\int_{1}^{t}\left[\mathscr{D}^{(\eta)}_{\dd}(f(\tau))\right]^{\frac{\g}{\eta}}\d\tau, \qquad t \in [1,T)\,,
\end{equation}
where we set $\bar{C}(\g,\eta)=\frac{|\g|}{\eta}\bar{C}_{1}^{1-\frac{\g}{\eta}}$.  Similar to the proof of Proposition \ref{theo:case4/3}, using the convexity of the mapping $r >0 \mapsto r^{\frac{\g}{\eta}}$, we have
\begin{align*}
\int_{1}^{t}\left[\mathscr{D}^{(\eta)}_{\dd}(f(\tau))\right]^{\frac{\g}{\eta}}\d\tau&\geq (t-1)\left(\int_{1}^{t}\mathscr{D}_{\dd}^{(\eta)}(f(\tau))\frac{\d\tau}{t-1}\right)^{\frac{\g}{\eta}}\\
&=(t-1)^{1-\frac{\g}{\eta}}\left(\int_{1}^{t}\mathscr{D}_{\dd}^{(\eta)}(f(\tau))\d\tau\right)^{\frac{\g}{\eta}}, \qquad t \in (1,T).
\end{align*}
Therefore, in light of \eqref{eq:Dddeta} and using \eqref{eq:ad},  there exists $C_{\eta} >0$ depending only on $f_{\mathrm{in}}$ such that
$$
\int_{1}^{t}\left[\mathscr{D}^{(\eta)}_{\dd}(f(\tau))\right]^{\frac{\g}{\eta}}\d\tau 
\geq C_{\eta}\left(1-\dd \chi^{\star}\right)^{-\frac{\g}{\eta}}(t-1)^{1-\frac{\g}{\eta}} {\left(1+t \right)^{\frac{3\g}{\eta}}}, \qquad t \in (1,T)\,,
$$
where we used \eqref{res19} and the fact that $\kappa_{0}(t) \geq 1-\dd\chi^{\star}$ for any $t \in [1,T)$. Inserting this into \eqref{eq:ytgeta},
$$
\bm{y}(t)^{\frac{\g}{\eta}}-\bm{y}(1)^{\frac{\g}{\eta}} \geq C_{\g,\eta}\left(1-\dd\chi^{\star}\right)^{1-2\frac{\g}{\eta}}(t-1)^{1-\frac{\g}{\eta}} {(1+t)^{3\frac{\g}{\eta}}}\,, \qquad t \in (1,T)\,,$$
for some positive constant $C_{\g,\eta} >0$ depending only on  $\eta,\|f_{\mathrm{in}}\|_{L^{1}_{2}}$ and $H(f_{\rm in})$. In other words
\begin{equation}\label{eq:estimYt}
\bm{y}(t) \leq \left(\bm{y}(1)^{\frac{\g}{\eta}} + C_{\g,\eta}\left(1-\dd\chi^{\star}\right)^{1-2\frac{\g}{\eta}}(t-1)^{1-\frac{\g}{\eta}} {(1+t)^{3\frac{\g}{\eta}}}\right)^{\frac{\eta}{\g}}, \end{equation}
for any $t \in (1,T)$. In particular
$$\bm{y}(t) \leq C_{\g,\eta} \,\left(1-\dd \chi^{\star}\right)^{\frac{\eta}{\g}-2}(t-1)^{\frac{\eta}{\g}-1}\, {(1+t)^{3}}\,, \qquad t \in  {(2,T)}.$$
We use this last estimate to sharpen the control of the third moment of $f(t,v)$.
\begin{lem} {For $\eta > 3+2|\g|$, one has}
\begin{equation}\label{eq:chilm3}
\sup_{t \in [2,T)}\lm_{3}(t) \leq C_{\eta,\g}(f_{\mathrm{in}}) \left(1-\dd\chi^{\star}\right)^{\frac{(\eta-3)(\eta-2\g)}{2\g\eta}}  {+ \| \M_{\dd} \|_{L^{1}_{3}}} .
\end{equation}
\end{lem}
\begin{proof}  We use ideas similar to those introduced at the end of Proposition  \ref{theo:case4/3}. For {$\eta >3+2|\g|$}, observe that
\begin{equation*}\begin{split}
\| f(t) \|_{L^{1}_{3}} &\leq \| \M_{\dd} \|_{L^{1}_{3}} + \| f(t) - \M_{\dd} \|_{L^{1}_{3}}  \\
&\leq \| \M_{\dd} \|_{L^{1}_{3}} + \| f(t) - \M_{\dd} \|^{1-\frac3\eta}_{L^{1}}\Big(\| f(t) \|_{L^{1}_{\eta}} + \| \M_{\dd} \|_{L^{1}_{\eta}}\Big)^{\frac3\eta}\\
& \leq  \| \M_{\dd} \|_{L^{1}_{3}} + C_{\eta}\big( 1+ t\big)^{\frac3\eta} \| f(t) - \M_{\dd} \|^{1-\frac3\eta}_{L^{1}}\\
& \leq \| \M_{\dd} \|_{L^{1}_{3}} + C_{\eta}\big( 1+ t\big)^{\frac3\eta} \mathcal{H}_{\dd}(f(t)|\M_{\dd})^{\frac12-\frac{3}{2\eta}}  \,,\qquad t\geq1\,,
\end{split}\end{equation*}
where, in the last inequality, we used one side of the Csisz\'ar-Kullback inequality \eqref{eq:czisz}.  {Let us note that $C_{\eta}$ does not depend on $\dd$ since $\| \M_{\dd} \|_{L^{1}_{\eta}}$ is uniformly bounded thanks to  \cite[Lemma A.7]{ABL}.} Plugging into the aforementioned estimation for $\bm{y}(t)=\mathcal{H}_{\dd}(f(t)|\M_{\dd})$, we obtain
\begin{multline*}
\lm_{3}(t) \leq \| \M_{\dd} \|_{L^{1}_{3}} + C_{\eta}\left(1-\dd\chi^{\star}\right)^{\frac{(\eta-3)(\eta-2\g)}{2\g\eta}}\times\\
\times (1+t)^{\frac{3}{\eta}}(t-1)^{\frac{(\eta-3)(\eta-\g)}{2\g\eta}}\, {(1+t)^{ \frac{3(\eta-3)}{2\eta}}}\,, \qquad t\in ({2},T)\,.\end{multline*}
Since {$\eta > 3+2|\g|$}, the function 
$$t \geq 2 \longmapsto (1+t)^{\frac{3}{\eta}}(t-1)^{\frac{(\eta-3)(\eta-\g)}{2\g\eta}}\, {(1+t)^{ \frac{3(\eta-3)}{2\eta}}}$$
is bounded by some positive constant $C_{\eta,\g}$.  We obtain then \eqref{eq:chilm3}.
\end{proof}
A simple consequence of the aforementioned Lemma is the following estimate on $\chi^{\star}$.
\begin{lem}
Assume that {$\eta >3+2|\g|$}, then there is a constant  $ C_{1}:=C_{1}(\gamma,\eta,f_{\rm in})$ independent of $\dd$ and $T$ such that
\begin{equation}\label{boundChi}
\chi^{\star}(T)\left(1-\dd\chi^{\star}(T)\right)^{\alpha} \leq  C_{1}, \qquad \qquad \alpha=\frac{(\eta-3)(\eta-2\g)}{2\eta(4+\g)} >0.\end{equation}
\end{lem}

\begin{proof} Using Theorem \ref{Linfinito*} (with $s=3$) and the fact that $\sup_{\tau \in[0,2) }\lm_{3}(\tau) \leq C(f_{\mathrm{in}})$ {thanks to Proposition \ref{shortime} (recall that \eqref{eq:ad} holds)} we can use 
the previous estimate to conclude that
\begin{equation*}
\chi^{\star}=98\sup_{t \in [1,T)}\max\bigg(\|f(t)\|_{L^{\infty}}\,,\,\sup_{\dd>0}\|\M_{\dd}\|_{L^{\infty}}\bigg) \leq  {\tilde{C}_{0}}\left(1+\sup_{\tau \in [2,T)}\lm_{3}(\tau)\right)^{-\frac{\g}{4+\g}},
\end{equation*}
which, with \eqref{eq:chilm3}, gives
$$\chi^{\star} \leq  {C_{0}}\,\left(1+\left(1-\dd\chi^{\star}\right)^{\frac{(\eta-3)(\eta-2\g)}{2\g\eta}}\right)^{-\frac{\g}{4+\g}} \leq {C_{1}}\,\left(1-\dd\chi^{\star}\right)^{-\alpha} , $$
where we used that $1-\dd\chi^{\star} \leq 1$. This gives \eqref{boundChi}. 
\end{proof} We introduce the mapping
$$\phi(x)=x\left(1-\dd\,x\right)^{\alpha}, \qquad x \in(0, \dd^{-1}).$$
One has 
$$\sup_{x\in(0, \dd^{-1})}\phi(x)=\phi(\bar{x})=\frac{\alpha^{\alpha}}{\dd(1+\alpha)^{1+\alpha}}, \qquad \bar{x}=\frac{1}{\dd(1+\alpha)}.$$
{We define 
\begin{equation}\label{trap-1}
\dd_\star=\frac{\alpha^\alpha}{2M(1+\alpha)^{\alpha+1}} \,,
%\dd_{\star}=\frac{\alpha^{\alpha}}{2M(1+\alpha)^{\alpha}}\,,
\end{equation}
where $M >0$ is a (large) constant to be determined. We consider values $\dd\in(0,\dd_{\star}]$ which ensure in particular that $M < \phi(\bar{x}).$  
Now, in such a case, the equation $\phi(x)=M$ has two roots $x_{1}<\bar{x}<x_{2}$ in the interval $(0,\dd^{-1})$. In particular, $\phi(x_{1})=M$ implies 
$$x_{1}=\frac{M}{(1-\dd x_{1})^{\alpha}} < \frac{M}{(1-\dd \bar{x})^{\alpha}}=\left(1+\frac{1}{\alpha}\right)^{\alpha}M < {\frac{1}{2\dd_{\star}}}.$$ 
Therefore,  the inequality $\phi(x) < M$ }holds
 in the following two cases:
\begin{equation}\label{dichotomy}
(i) \;\; \textrm{ either $x \leq x_{1} <  {\frac{1}{2\dd_{\star}}}\,$,} \qquad \qquad (ii) \;\;
\textrm{ or $x \geq x_{2} > \bar{x} = \frac{1}{\dd(1+\alpha)}\,.$}\end{equation}
Let us now show that,  choosing $M$ large enough, the second case $(ii)$ is an impossibility.

\begin{lem}\label{lem:T>3} Besides \eqref{trap-1}, assume that  $M\geq C_{1}$ and \begin{equation}\label{trap}
M\geq \frac{\chi^{\star}(3)}{2} \left(\frac{\alpha}{1+\alpha}\right)^{\alpha}.
\end{equation}
Then, for $\dd \in (0,\dd_{\star})$, it holds
$$\chi^{\star}(T)\leq \chi^{\star}(3) \leq x_{1} <   {\frac{1}{2\dd_{\star}}}, \qquad T\in(2,3].$$
\end{lem} 
\begin{proof}  Notice that \eqref{trap} means that
$\dd_{\star} \leq \frac{1}{\chi^\star(3)(1+\alpha)}.$
Applying Theorem \ref{Linfinito*} {with $s=3$} on the interval $[1,3)$, one has
$$\sup_{t \in [1,3)}\left\|f(t)\right\|_{L^{\infty}} \leq C\,\bigg(1 + \sup_{t \in [1,3)}\lm_{3}(t)\bigg)^{-\frac{\g}{4+\g}}$$
for some positive $C$ depending on $\|f_{\rm in}\|_{L^{1}_{2}}$ and $H(f_{\rm in})$ and this last quantity is finite and independent of $\dd$ thanks to Proposition \ref{shortime} {since $\lm_{3}(0) <\infty.$}  Therefore,
$$\chi^\star(3)=98\sup_{t \in [1,3)}\max\bigg(\|f(t)\|_{L^{\infty}}\,,\,\sup_{\dd >0}\|\M_{\dd}\|_{L^{\infty}}\bigg) < \infty\,,$$
depends only on $\|f_{\mathrm{in}}\|_{L^{1}_{2}}$ and $H(f_{\rm in})$. Under the additional  constraint \eqref{trap}, we observe that for any $\dd \in (0,\dd_{\star}]$ it holds $\bar{x}\geq \chi^{\star}(3)$, which excludes the case $(ii)$.  By the aforementioned binary option, one gets the desired conclusion.
\end{proof}

This argument shows the existence of a trapping region which can be extended to any $T>3$. 

\begin{lem} Assume  \eqref{trap-1} and \eqref{trap} are in force. Then, defining
\begin{align*}
T^{\star} := \sup\Big\{ T>2 \; \big| \chi^{\star}(T)\leq x_{1}\Big\}\,,
\end{align*} 
one can choose $M$ large enough (explicit) such that $T^{\star}=\infty$ for any $\dd \in (0,\dd_{\star})$.
\end{lem}

\begin{proof} We already saw in Lemma \ref{lem:T>3} that $T^{\star}\geq 3$. We argue by contradiction considering that $T^{\star}<\infty$. {In all the sequel, we will denote by $u(t)$ a function of $t \geq 0$ which is such that $\lim_{t\to0^{+}}u(t)=0$ (i.e. $u(t)=o(1)$) and that may change from line to line.} Recalling and integrating the moment inequality \eqref{eq:mom-s-1} (with $\delta=1$ {and $s=3$}) in the time interval $(T^\star,T^\star+t )$, it follows that 
\begin{align*}
\lm_{3}(T^\star + t ) &\leq \lm_{3}(T^\star) +  {6}\bm{K}_{3}\,t  +  {\frac32}\int^{T^\star+t }_{T^{\star}} \lD_{3+\g}(\tau)\d\tau +   {3C} \int^{T^\star+t }_{T^{\star}}\lM_{3+\g}(\tau)\d\tau \\
&= \lm_{3}(T^\star) +   {u(t)}\,,\qquad t \in(0,1]\,,
\end{align*}
since the latter three terms in the right-hand side can be made as small as desired when $t \rightarrow0$.  In other words,
\begin{equation}\label{cg3}
\sup_{\tau \in [2,T^\star +t )}\lm_{3}(\tau) = \sup_{\tau \in [2,T^\star )}\lm_{3}(\tau) +  {u(t)}\,,\qquad t \in(0,1]\,.
\end{equation}
Using Theorem \ref{Linfinito*} applied on the interval $[1,T^\star + t )$, the fact that $\sup_{\tau \in[0,2) }\lm_{3}(\tau) \leq C(f_{\mathrm{in}})\,,$
and the continuous growth of the third moment \eqref{cg3}, one is led to
\begin{align*}
\chi^\star(T^\star +t ) &\leq C\,\left(1 + \sup_{\tau \in [2,T^\star + t )}\lm_{3}(\tau)\right)^{-\frac{\g}{4+\g}} = C\,\left(1 + \sup_{\tau \in [2,T^\star)}\lm_{3}(\tau)\right)^{-\frac{\g}{4+\g}} +  {u(t)}\,,
\end{align*}
for some positive $C$ depending on $\|f_{\rm in}\|_{L^{1}_{2}}$ and $H(f_{\rm in})$. Consequently, one can use \eqref{eq:chilm3} with $T=T^\star$  {to get
\begin{multline*}
\left(1 + \sup_{\tau \in [2,T^\star)}\lm_{3}(\tau)\right)^{-\frac{\g}{4+\g}} \leq \left(1+C_{\eta,\g}(f_{\mathrm{in}}) \left(1-\dd\chi^{\star}(T^{\star})\right)^{\frac{(\eta-3)(\eta-2\g)}{2\g\eta}}+ \| \M_{\dd} \|_{L^{1}_{3}}\right)^{-\frac{\g}{4+\g}}\\
\leq 2^{-\frac{\g}{4+\g}-1}C_{\eta,\g}^{-\frac{\g}{4+\g}} \left(1-\dd\chi^{\star}(T^{\star})\right)^{\frac{(\eta-3)(\eta-2\g)}{2\g\eta}\frac{-\g}{4+\g}}+2^{-\frac{\g}{4+\g}-1}\left(1+\|\M_{\dd}\|_{L^{1}_{3}}\right)^{-\frac{\g}{4+\g}}.\end{multline*}
We deduce from this that} there is some $\overline{C} >0$ independent of $\dd$ and $M$ such that
\begin{equation}\label{eq:11}
\chi^{\star}(T^{\star}+t ) \leq \overline{C}\left(1+\big(1-\dd \chi^{\star}(T^{\star})\big)^{-\alpha}\right)+u(t) \leq 2\overline{C}\big(1-\dd \chi^{\star}(T^{\star})\big)^{-\alpha} + {u(t)}\,,\end{equation}
where we used again that $1-\dd\chi^{\star}(T^{\star}) \leq 1.$  Notice that, by definition of $T^{\star}$, 
$$\big(1-\dd \chi^{\star}(T^{\star})\big)^{-\alpha} {\leq} \big(1-\dd x_{1}\big)^{-\alpha}\,,$$
where $ {1-\dd \chi^{\star}(T^{\star}) \geq } 1-\dd x_{1} \geq 1-\dd_{\star}x_{1}>\tfrac{1}{2}$.  Thus, $\left(1-\dd \chi^{\star}(T^{\star})\right)^{-\alpha} \in [1,2^{\alpha}]$ and, for $t $ small enough, \eqref{eq:11} implies that
\begin{equation}\label{eq:12}
\chi^{\star}(T^{\star}+t ) \leq 3\overline{C}\big(1-\dd\chi^{\star}(T^{\star})\big)^{-\alpha} \leq 3\overline{C}\left(1-\dd x_{1}\right)^{-\alpha}.\end{equation}
Set now 
$$M:=\max\left(\frac{\chi^{\star}(3)}{2}\left(\frac{\alpha}{1+\alpha}\right)^{\alpha},3\overline{C} , C_{1} \right).$$
One deduces from \eqref{eq:12} that 
$$\chi^{\star}(T^{\star}+t ) \leq M(1-\dd x_{1})^{-\alpha}=x_{1}\,,$$
which is a contradiction since, by definition of $T^{\star}$, $\chi^{\star}(T^{\star}+t ) > x_{1}.$
Thus, for the above choice of $M$, we must have that $T^{\star}=\infty$.\end{proof} 
 We have all in hands to conclude.

\begin{proof}[Proof of Theorem \ref{theo:main}] The previous Lemma exactly means that,  {for some explicit $\dd_{\star} >0$ (associated to the above choice of $M$)}, one has
\begin{equation*}
\chi^{\star}(T) \leq x_{1} <  {\frac{1}{2\dd_{\star}}}\,, \qquad\qquad \forall\;T > 2\,, \qquad \forall \dd \in (0,\dd_{\star}).
\end{equation*}
This proves in particular that
$$\sup_{t\geq 1}\|f(t)\|_{L^{\infty}} \leq  {\frac{1}{196\dd_{\star}}},$$ which is independent of $\dd$. This gives the no saturation property
$$\kappa_{0} = 1 - \dd\,\sup_{t\geq1}\| f(t) \|_{\infty}>0, \qquad \forall \,\dd \in (0,\dd_{\star}].$$
At this stage, we  can resume the proof of Proposition \ref{theo:case4/3} to get the desired result.
\end{proof}
\section{Upgrading the convergence}\label{strexp}
We explain in this section how the rate of convergence can be upgraded to a \emph{stretched exponential} rate whenever the initial datum satisfies a more stringent decay in terms of large velocities decay. As before, our strategy is based upon suitable interpolations. Notations are those of Section~\ref{sec:gene6} and we follow at first the interpolation procedure developed in \cite[Section 5]{CEL}. Namely, we begin by improving the interpolation inequality between $\mathscr{D}_{\dd}^{(\g)}$ and $\mathscr{D}_{\dd}^{(0)}$ provided by inequality \eqref{eq:interpolDe}.
\begin{lem}\label{lem:dissipation_interpolation_exp}
For a given $a>0$ and {$q >0$} define, for any $0 \leq g \leq \dd^{-1}$,
$$\Gamma^{a,q}_{\dd}(g)=\int_{\R^{3}\times\R^{3}}|v-\vet|^{2}\exp(a|v-\vet|^{q})\bm{\Xi}_{\dd}[g](v,\vet)\d v\d \vet\,,$$
where $\bm{\Xi}_{\dd}$ is defined by formula (\ref{eq:Xidd}).
Then for any $\gamma<0$ one has that
\begin{equation}\label{eq:inequExp}
\mathscr{D}_{\dd}^{(\g)}(g) \geq \frac{1}{2}\left[\frac{1}{a} \log\left(\frac{\Gamma_{\dd}^{(a,q)}(g)}{\mathscr{D}_{\dd}^{(0)}(g)}\right)\right]^{\frac{\g}{q}}\,\mathscr{D}^{(0)}
_{\dd}(g),
\end{equation}
where $\mathscr{D}_{\dd}^{(\g)}(g)$ is defined by formula (\ref{defdeta}).
\end{lem}

\begin{proof}
For a given $R>0$, we set $\mathcal{Z}_{a,R}=\left\{(v,\vet) \in \R^{3}\times \R^{3}\;;\;|v-\vet| \leq \left(\frac{R}{a}\right)^{\frac{1}{q}}\right\}$ and denote by $\mathcal{Z}_{a,R}^{c}$ its complementary in $\mathbb{R}^{6}.$ We see that
\begin{multline*}
\mathscr{D}_{\dd}^{(0)}(g)
=\frac{1}{2}\int_{\mathcal{Z}_{a,R}}|v-\vet|^{\g}\,|v-\vet|^{-\g}\,|v-\vet|^{2}\bm{\Xi}_{\dd}[g](v,\vet)\d v\d\vet\\
+\frac{1}{2}\int_{\mathcal{Z}_{a,R}^{c}}\exp\left(-a|v-\vet|^{q}\right)\exp\left(a|v-\vet|^{q}\right)\,|v-\vet|^{2}\bm{\Xi}_{\dd}[g](v,\vet)\d v\d\vet\\
\leq  \left(\frac{R}{a}\right)^{\frac{|\g|}{q}}\mathscr{D}_{\dd}^{(\g)}(g)
+\frac{1}{2}\exp(-R)\Gamma_{\dd}^{(a,q)}(g)\,.
\end{multline*}
{We also notice that} for any $a,q>0$, we have that $1\leq \exp\left(a|v-\vet|^q\right)$, and therefore
$2\mathscr{D}_{\dd}^{(0)}(g) \leq \Gamma^{a,q}_{\dd}(g)$. Thus, the choice
$$R:= \log\left(\frac{\Gamma_{\dd}^{a,q}(g)}{\mathscr{D}_{\dd}^{(0)}(g)}\right) \ge \log 2 >0$$
is possible, and yields 
$$\mathscr{D}_{\dd}^{(0)}(g) \leq \left[\log\left(\frac{\Gamma_{\dd}^{(a,q)}(g)}{\mathscr{D}_{\dd}^{(0)}(g)}\right)\right]^{\frac{|\g|}{q}}\,a^{\frac{\g}{q}}\,{\mathscr{D}_{\dd}^{(\g)}(g)} +\frac{1}{2}\mathscr{D}^{(0)}_{\dd}(g)\,,$$
which completes the proof.
\end{proof}

\begin{rmq} Applying this inequality to a weak solution $f(t,v)$ to \eqref{LFD} and assuming that the initial datum $f_{\mathrm{in}}$ and $\dd >0$ are such that \eqref{eq:D0lamb} holds, together with the estimate
\begin{equation}\label{eq:Gammabound}
\sup_{t\geq1}\Gamma_{\dd}^{(a,q)}(f(t)) \leq \bar{\Gamma}\,,
\end{equation}
we expect that the relative entropy
$$\bm{y}(t)=\mathcal{H}_{\dd}(f(t)|\M_{\dd})$$ 
satisfies a differential inequality of the type
$$\dfrac{\d}{\d t}\bm{y}(t) \leq -\frac{\bar{\lambda}_{0}}{2}\left[\frac{1}{a} \log\left(\frac{\bar{\Gamma}}{\bar{\lambda}_{0}\bm{y}(t)}\right)\right]^{\frac{\g}{q}}\,\bm{y}(t), \qquad t \geq 1,$$
leading to an estimate like
$$\bm{y}(t) \leq A\exp(-Bt^{\frac{q}{q-\g}})\,, \qquad t \geq 1, $$
for some positive constant $A,B >0$. We will see that, even if we cannot prove directly \eqref{eq:Gammabound}, the (at most) slowly increasing growth of $\Gamma_{\dd}^{(a,q)}(f(t))$ will be such that the above decay still holds. 
\end{rmq} 
Following the paths of Section \ref{sec:converge}, we first look for suitable upper bound for $\Gamma_{\dd}^{(a,q)}(g).$ We proceed as in Lemma \ref{lem:Ds-est} to get the following result.
\begin{lem}\label{lem:Ds-est-exp}
 For any $0 \leq g \leq \dd^{-1}$ satisfying \eqref{eq1} and any  {$a>0$, $q \in (0,1)$} one has 
\begin{equation}\label{eq:estGamma}
\Gamma_{\dd}^{(a,q)}(g) \leq \frac{{32}}{\kappa_{0}(g)}\left\|g\,\mu_{a,q}\right\|_{L^{1}_{2}}\,\int_{\R^{3}}\langle v\rangle^{2}\left|\nabla \sqrt{g(v)}\right|^{2}\mu_{a,q}(v)\d v,
\end{equation}
where
\begin{equation}\label{nndef}
\mu_{a,q}(v)=\exp\left(a\, \langle v\rangle^{q}\right), \qquad v \in \R^{3}.
\end{equation}
\end{lem}

\begin{proof} Recalling definition \eqref{eq:Xidd}, we see that
$$\Gamma_{\dd}^{(a,q)}(g)=\int_{\R^{6}}|v-\vet|^{2}\exp(a|v-\vet|^{q})\,F\,F_{\ast}\left|\Pi(v-\vet)\left[\nabla h-\nabla h_{\ast}\right]\right|^{2}\d v\d\vet , $$
where $h(v)=\log(g(v))-\log(1-\dd g(v))$ and $F=g(1-\dd g)$. Using the obvious estimate 
$$\left|\Pi(v-\vet)\left[\nabla h-\nabla h_{\ast}\right]\right|^{2} \leq 2|\nabla h|^{2}+2|\nabla h_{\ast}|^{2}, $$
and {$|v-\vet|^{2}\exp(a|v-\vet|^{q}) \leq 2\langle v\rangle^{2}\mu_{a,q}(v)\langle \vet\rangle^{2}\mu_{a,q}(\vet)$}  {since $q \in (0,1)$,}
one deduces that
\begin{equation*}\begin{split}
\Gamma_{\dd}^{(a,q)}(g) &\leq {8} \int_{\R^{3}}\langle v\rangle^{2}\mu_{a,q}(v)F(v)\,|\nabla h(v)|^{2}\d v\int_{\R^{3}} F(v_*)\,\mu_{a,q}(\vet)\langle \vet\rangle^{2}\d\vet\\
&\leq {8}\int_{\R^{3}}\frac{|\nabla g(v)|^{2}}{g(1-\dd g)}\langle v\rangle^{2}\mu_{a,q}(v)\d v\int_{\R^{3}}\langle \vet\rangle^{2}g_{\ast}\mu_{a,q}(\vet)\d \vet.
\end{split}\end{equation*}
This yields the result.
\end{proof}
As for Proposition \ref{prop:weightFish}, on the basis of \eqref{eq:estGamma}  {and \eqref{eq:inequExp}},  it is useful to get a uniform in time upper bound of the above  Fisher information with exponential weights along solutions to \eqref{LFD}. Before doing so, let us introduce the following objects.
\begin{defi}\label{defi:expon} Given $a,q >0$, we recall definition (\ref{nndef}).
Then, for any nonnegative measure function $g\::\:\R^{3}\to \R$, we set
$$\up_{a,q}(g) :=\int_{\R^{3}}g^{2}(v)\mu_{a,q}(v)\d v, \qquad  \vr_{a,q}(g) :=\int_{\R^{3}}g(v)\mu_{a,v}(v)\d v.$$
Given $s \geq0$, we also introduce
$$\overline{\up}_{a,q,s}(g) :=\int_{\R^{3}}g^{2}(v)\langle v\rangle^{s}\mu_{a,q}(v)\d v, \qquad  \overline{\vr}_{a,q,s}(g) :=\int_{\R^{3}}g(v)\langle v\rangle^{s}\mu_{a,q}(v)\d v.$$
If $f(t,v)$ is a weak-solution to \eqref{LFD}, we will moreover simply write
$$\up_{a,q}(t) :=\up_{a,q}(f(t)), \qquad \vr_{a,q}(t) :=\vr_{a,q}(f(t)), \qquad t \geq 0\,,$$
with similar notations for $\overline{\up}_{a,q,s}(t),\overline{\vr}_{a,q,s}(t)$. We also set
$$\bm{\Pi}_{a,q}(t) :=\tfrac{1}{2}\up_{a,q}(t)+\vr_{a,q}(t).$$
\end{defi}
Estimates on the evolution of the above family of moments are easily deduced from Theorem~\ref{theo:main-moments} since we keep track, for the evolution of $\bm{E}_{s}(t)$, of the dependency with respect to $s$. Namely, one has the following proposition, with a proof quite similar to that of \cite[Corollary 4.1]{CDH}.

\begin{prop}\label{prop:exp-mom} Assume that $-2 < \g < 0$  and let a nonnegative initial datum $f_{\mathrm{in}}$ satisfying \eqref{hypci}--\eqref{eq:Mass} for some $\dd_0 >0$ be given. For $\dd \in (0,\dd_0]$, let  $f(t,\cdot)$ be a weak-solution to \eqref{LFD}. Let $a >0$ and $0 < q <\frac{4+2\g}{8-\g}$. Assume that
$$\int_{\R^{3}}\exp\left({2^{q(1+\frac{1}{|\g|})}}a\langle v\rangle^{q}\right)f_{\rm in}(v)\d v= {\vr_{\tilde{a},q}(f_{\rm in})} < \infty\,, \qquad \qquad  {\tilde{a}=2^{\frac{q(1+|\g|)}{|\g|}}a}.$$
{Then} there exists $\bm{C}_{a,q} >0$ depending only on $a,q$ and $f_{\rm in}$  such that
$$\bm{\Pi}_{a,q}(t) \leq \bm{C}_{a,q}\left(t^{-\frac{3}{2}}+t\right)\,, \qquad t >0\,.$$
\end{prop}

\begin{proof} As in \cite[Corollary 4.1]{CDH}, we notice that
$$\bm{\Pi}_{a,q}(t)=\sum_{n=0}^{\infty}\frac{a^{n}}{n!}\bm{E}_{nq}(t)\,,$$
so that thanks to Theorem \ref{theo:main-moments} $$\bm{\Pi}_{a,q}(t) \leq \Big( t^{-\frac{3}{2}}+t \Big)\sum_{n=0}^{\infty}\frac{a^{n}}{n!}\bm{C}_{nq}\,.$$
Consequently, in order to prove the result, we just need to show that the sum is finite. Using now  \eqref{rmq:Csfinal}, there is $\beta_{1} >0$ such that
$$\bm{C}_{nq} \leq \beta_{1}\left[\left(\beta_{1}\,q\,n\right)^{{{\frac{8-\g}{4+2\g}}(nq+\g-2)+1}}  +2^{\frac{nq}{|\g|}} \,  {(1 + nq)^{\frac{5}{2}}} \, \lm_{nq}(0)\right]\,\qquad \quad \quad (nq >6+|\g|).$$
Clearly, for $n$ large enough,{$(1+nq)^{\frac{5}{2}} \leq c_{0}2^{nq}$ for some universal $c_{0} >0$,  so that
$$\bm{C}_{nq} \leq 2\beta_{1} \left(\beta_{1}\,q\,n\right)^{nb+\ell} +  \beta_{1}c_{0}2^{nq(1+\frac{1}{|\g|})}\,\lm_{nq}(0)\,,$$}
with $b=\frac{8-\gamma}{4+2\gamma}q$ and $\ell= \frac{8-\gamma}{4+2\gamma} (\gamma-2)+1.$  Using Stirling formula and d'Alembert's ratio test, one sees easily that, if $b <1$, then the sum 
$$\sum_{n=0}^{\infty}\tfrac{a^{n}}{n!} \left(\beta_{1}\,q\,n\right)^{nb+\ell} \quad \text{ is finite for any } a >0\,,$$
whereas
$$\sum_{n=0}^{\infty}\tfrac{a^{n}}{n!} 2^{\frac{nq(1+|\g|)}{|\g|}}\lm_{nq}(0)=\int_{\R^{3}}\exp\left({2^{q(1+\frac{1}{|\g|})}}a\langle v\rangle^{q}\right)f_{\rm in}(v)\d v < \infty\,.$$
This proves the result.\end{proof}
\begin{rmq} From the above proof, one sees that, if $q= {\frac{4+2\g}{8-\g}}$, then the above result still holds if {$2^{\frac{q}{4}}a\beta_{1}q e <1$}.
\end{rmq}
We need in the sequel an analogue of  Lemma \ref{lem:flogf}. 

\begin{lem}\label{lem:flogf-exp} Assume that $-2 < \g < 0$  and let a nonnegative initial datum $f_{\mathrm{in}}$ satisfying \eqref{hypci}--\eqref{eq:Mass} for some $\dd_0 >0$ be given. For $\dd \in (0,\dd_0]$, let  $f(t,\cdot)$ be a weak-solution to \eqref{LFD}. Then, given $a,q >0$ and any $s >\frac32$, there exists $C_{s}(f_{\mathrm{in}}) >0$ depending on  $s$, $f_{\rm in}$,    (but not on $a,q$) such that, for any $k \geq 0$ and  any $t\ge 0$,
\begin{multline}\label{eq:cgflogf-exp}
-\int_{\R^{3}}\langle v\rangle^{k}\bm{c}_{\g}[f(t)]\,f(t,v)\big(1+|\log f(t,v)|\big)\mu_{a,q}(v)\d v \\
\leq C_{s}(f_{\mathrm{in}})\bigg(\sqrt{ \overline{\vr}_{2a,q,2(k+s)}(t)}+\overline{\up}_{a,q,k}(t)\\
+\left(\overline{\vr}_{\frac{3}{2}a,q,\frac{3}{2}k}(t)+\overline{\up}_{\frac{3}{2}a,q,\frac{3}{2}k}(t)\right)^{\frac{2}{3}}{\Big(1+\frac{1}{t}\Big)}\bigg)\,,
\end{multline}
 and
\begin{multline}\label{eq:flogf-exp}
\int_{\R^{3}}\langle v\rangle^{k+\g}f(t,v)\left(1+\left|\log f(t,v)\right|\right)\mu_{a,q}(v)\d v \\
\leq C_{s}(f_{\mathrm{in}})\left(\sqrt{ \overline{\vr}_{2a,q,2(k+s+\g)}(t)}+\overline{\up}_{a,q,k+\g}(t)\right)\,.
\end{multline}
\end{lem}

\begin{proof} The proof is very similar to that of Lemma \ref{lem:flogf} and is based upon \eqref{obest}. We use the same notations as in Lemma \ref{lem:flogf} and use the splitting $\bm{c}_{\g}[f]=\bm{c}_{\g}^{(1)}[f]+\bm{c}_{\g}^{(2)}[f]$.
One has
\begin{multline*}
-\int_{\R^{3}}\langle v\rangle^{k}\bm{c}_{\g}[f]\,f\left(1+|\log f|\right)\mu_{a,{q} }(v)\d v
\leq -C_{\frac43,\frac32}\int_{\R^{3}}\langle v\rangle^{k}\bm{c}_{\g}^{(1)}[f]\left(f^{\frac{2}{3}}+f^{\frac{4}{3}}\right)\mu_{a,q}(v)\d v \\
-C_{2,2}\int_{\R^{3}}\langle v\rangle^{k}\bm{c}_{\g}^{(2)}[f]\left(\sqrt{f}+f^{2}\right)\mu_{a,q}(v)\d v\,.\end{multline*}
As in Lemma \ref{lem:flogf}, a simple use of Cauchy-Schwarz inequality yields, for any $s > \frac{3}{2}$, 
$$
-\int_{\R^{3}}\langle v\rangle^{k}\bm{c}_{\g}^{(2)}[f(t)]\left(\sqrt{f}+f^{2}\right)\mu_{a,q}(v)\d v 
\leq C_{s}(f_{\mathrm{in}})\left(\sqrt{ \overline{\vr}_{2a,q,2(k+s)}(t)}+\overline{\up}_{a,q,k}(t)\right), $$
for some positive constant depending only on $s, \|f_{\mathrm{in}}\|_{L^1_{2}} $. In the same way, as in Lemma~\ref{lem:flogf}, we deduce from H\"older's inequality, and Proposition \ref{prop:boundedL2}  that
\begin{multline*}
-\int_{\R^{3}}\langle v\rangle^{k}\bm{c}_{\g}^{(1)}[f(t)]\left(f(t,v)^{\frac{2}{3}}+f(t,v)^{\frac{4}{3}}\right)\mu_{a,q}(v)\d v \\
\leq C_{\gamma} \left\|\langle \cdot \rangle^{k}\left(f^{\frac{2}{3}}+f^{\frac{4}{3}}\right)\mu_{a,q}\right\|_{L^{\frac{3}{2}}}\,\big(\lm_{0}(t)+\lM_{0}(t)\big)^{\frac{2}{3}}\\
\leq C(f_{\mathrm{in}})\Big(\overline{\vr}_{\frac{3}{2}a,q,\frac{3}{2}k}(t)+\overline{\up}_{\frac{3}{2}a,q,\frac{3}{2}k}(t)\Big)^{\frac{2}{3}}\Big(1+\frac{1}{t}\Big)\,.
\end{multline*}
 This proves \eqref{eq:cgflogf-exp}. Now, the proof of \eqref{eq:flogf-exp} follows the same lines as that of \eqref{eq:flogf}.
\end{proof}

\begin{prop}\label{prop:weightFish-mu}
Assume that $-2 < \g < 0$  and let a nonnegative initial datum $f_{\mathrm{in}}$ satisfying \eqref{hypci}--\eqref{eq:Mass} for some $\dd_0 >0$ be given. For $\dd \in (0,\dd_0]$, let  $f(t,\cdot)$ be a weak-solution to \eqref{LFD}. Let $b,q >0$ be given,
 with $q < {\frac{4+2\g}{8-\g}}$. Assume moreover that
$$f_{\mathrm{in}} \in L^{1}\left(\R^{3},\mu_{2\tilde b+\delta,q}(v)\d v\right),\qquad \tilde{b} := { 2^{\frac{q(1+|\g|)}{|\g|}}b}\,, $$
for some $\delta >0$. Then,  for any $t_{0} >0$, there exists $C >0$ depending on $b,\delta,q,t_{0}$ and  $f_{\rm in}$, such that 
 $$\int_{t_{0}}^{t}\d \tau \int_{\R^{3}} \langle v\rangle^{\g}\,\mu_{b,q}(v)\left|\nabla \sqrt{f(\tau,v)}\right|^{2}\d v \leq C\big( 1+t \big)^{2}\,,  \qquad t >t_{0} >0.$$
\end{prop}
\begin{proof} Let us fix $b,q >0$. We investigate  the evolution of 
$$\bm{S}_{b,q}(t):=\int_{\R^{3}}\mu_{b,q}(v)f(t,v)\log f(t,v)\d v $$
for a solution $f=f(t,v)$ to \eqref{LFD}. To simplify notations, we write $F=F(t,v)=f(1-\dd f)$. One checks, similar to \eqref{evol:Sk}, that
\begin{equation*}\label{evol:Smu} 
\dfrac{\d}{\d t}\bm{S}_{b,q}(t)=\dfrac{\d}{\d t}\vr_{b,q}(t)+\int_{\R^{3}}\mu_{b,q}(v)\nabla\cdot\left(\bm{\Sigma}[f]\nabla f\right)\log f\d v-\int_{\R^{3}}\mu_{b,q}(v)\nabla \cdot \left(\bm{b}[f]F\right)\log f\d v ,
\end{equation*}
with
\begin{multline*}
\int_{\R^{3}}\mu_{b,q}(v)\nabla\cdot\left(\bm{\Sigma}[f]\nabla f\right)\log f\d v=-\int_{\R^{3}}\mu_{b,q}(v)\bm{\Sigma}[f]\nabla f\cdot \frac{\nabla f}{f}\d v\\
+\int_{\R^{3}}\nabla \cdot \big(\bm{\Sigma}[f]\nabla \mu_{b,q}\big)\,\big[f\log f-f\big]\d v\,,\end{multline*}
and
\begin{multline*}
\int_{\R^{3}}\mu_{b,q}(v)\nabla \cdot \left(\bm{b}[f]F\right)\log f\d v=-bq \int_{\R^{3}}\langle v\rangle^{q-2}\,F\log f\,\big(\bm{b}[f]\cdot v\big)\mu_{b,q}(v)\, \d v\\
+b\,q \int_{\R^{3}}\big(f-\frac{\dd}{2}f^{2}\big)\langle v\rangle^{q-2}\big(\bm{b}[f]\cdot v\big)\,\mu_{b,q}(v)\d v {+}\int_{\R^{3}}\mu_{b,q}(v)\left(f-\frac{\dd}{2}f^{2}\right)\bm{c}_{\g}[f]\d v.
\end{multline*}
For the latter, we notice that
\begin{align*}
\nabla \mu_{b,q}(v)&=b\,q\,v\langle v\rangle^{q-2}\mu_{b,q}(v)\,,\qquad \text{and}\\
\nabla \cdot \big(\bm{\Sigma}[f]\nabla \mu_{b,q}\big)&=b\,q\,\mu_{b,q}(v)\Big(\langle v\rangle^{q-2}\bm{B}[f]\cdot v +\langle v\rangle^{q-4}\mathrm{Trace}\big(\bm{\Sigma}[f]\cdot \bm{A}_{\mu}(v)\big)\Big),
\end{align*}
with $\bm{A}_{\mu}(v)=\langle v\rangle^{2}\mathbf{Id}+\left[(q-2)+b\,q\,\langle v\rangle^{q}\right]\,v\otimes v$.

\medskip
\noindent
As in the proof of Proposition \ref{prop:weightFish}, using that both $|\bm{b}[f]\cdot v|$ and $\frac{1}{2}|\bm{B}[f]\cdot v|$ satisfy \eqref{eq:bm}, and using now that
$$\mathrm{Trace}\left(\bm{\Sigma}[f]\cdot \bm{A}_{\mu}(v)\right)  \leq  {C_{b,q}}
\langle v\rangle^{q+4+\g}\|f\|_{L^{1}_{\g+2}}\,,$$ one deduces the following analogue of \eqref{eq:SkUe1},
\begin{multline}\label{eq:SkUeexp}
\frac{\d}{\d t}\bm{S}_{b,q}(t)-\dfrac{\d}{\d t}\vr_{b,q}(t) +K_{0}\int_{\R^{3}}\mu_{b,q}(v)\langle v\rangle^{\g}\frac{|\nabla f|^{2}}{f}\d v \\
\leq C_{b,q}(f_{\mathrm{in}})\int_{\R^{3}}\langle v\rangle^{2q+\g}f\left(1+\left|\log f\right|\right)\mu_{b,q}(v)\d v\\
-C_{b,q}(f_{\mathrm{in}})\int_{\R^{3}}\langle v\rangle^{q}\bm{c}_{\g}[f]\,f\left(1+\left|\log f\right|\right)\mu_{b,q}(v)\d v ,
\end{multline}
for some positive constant $C_{b,q}(f_{\mathrm{in}})$ depending on $b,q$ and $f_{\mathrm{in}}$ only through $\|f_{\mathrm{in}}\|_{L^{1}_{2}}$. We use now the results of Lemma \ref{lem:flogf-exp} to get for $s=2$ that
\begin{multline}\label{eq:SkUeexp1}
\frac{\d}{\d t}\bm{S}_{b,q}(t)-\dfrac{\d}{\d t}\vr_{b,q}(t) +K_{0}\int_{\R^{3}}\mu_{b,q}(v)\langle v\rangle^{\g}\frac{|\nabla f|^{2}}{f}\d v \\
\leq C_{b,q} (f_{\mathrm{in}})\left(\sqrt{ \overline{\vr}_{2b,q,2r+4}(t)}+\overline{\up}_{b,q,r}(t)+\left(\overline{\vr}_{\frac{3}{2}b,q,\frac{3}{2}q}(t)+\overline{\up}_{\frac{3}{2}b,q,\frac{3}{2}q}(t)\right)^{\frac{2}{3}}{\Big(1+\tfrac{1}{t}\Big)}\right),
\end{multline}
where $r=\max(2q+\g,q)$. For any $\delta >0$ and $t_{0} >0$, since $\overline{\vr}_{2b,q,2r+4}(t) \leq C_{\delta}\,{\vr}_{2b+\delta,q}(t)$ and similarly for $\overline{\up}_{b,q,r}(t)$ and the remainder terms, one has that, {for $t\ge t_0$},
\begin{multline}\label{eq:SkUeexp1}
\frac{\d}{\d t}\bm{S}_{b,q}(t)-\dfrac{\d}{\d t}\vr_{b,q}(t) +K_{0}\int_{\R^{3}}\mu_{b,q}(v)\langle v\rangle^{\g}\frac{|\nabla f|^{2}}{f}\d v \\
\leq C_{b,q,\delta}(f_{\mathrm{in}},t_{0})\left(\sqrt{\vr_{2b+\delta,q}(t)}+ {\up}_{b+\delta,q}(t)+\left( \vr_{\frac{3}{2}b+\delta,q}(t)+ {\up}_{\frac{3}{2}b+\delta,q}(t)\right)^{\frac{2}{3}}\right)\,.
\end{multline}
Using now Proposition \ref{prop:exp-mom}, assuming that  {$q < \frac{4+2\g}{8-\g}$} and
$ {\vr}_{2\tilde{b}+\delta,q}(0) < \infty$ (after renaming $\delta>0$) we deduce that {for $t\ge t_0$ and $\delta>0$},
$$
\frac{\d}{\d t}\bm{S}_{b,q}(t)-\dfrac{\d}{\d t}\vr_{b,q}(t) +K_{0}\int_{\R^{3}}\mu_{b,q}(v)\langle v\rangle^{\g}\frac{|\nabla f|^{2}}{f}\d v 
\leq C_{\delta,b,q,t_{0}}(f_{\mathrm{in}})\big(1+t\big)\,,$$
for some positive constant $C_{\delta,b,q,t_{0}}(f_{\mathrm{in}})$ depending only on $\delta,b,q,t_{0}$ and  $f_{\mathrm{in}}$. Integrating this inequality over $(t_{0},t)$ yields  
\begin{align*}
K_{0}\int_{t_{0}}^{t}\d\tau\int_{\R^{3}}&\langle v\rangle^{\g}\frac{|\nabla {f(\tau,v)}|^{2}}{f(\tau,v)}\mu_{b,q}(v)\d v \\
&\leq\bm{S}_{b,q}({t_0})-\bm{S}_{b,q}(t)+\vr_{b,q}(t) + \frac{1}{2}C_{\delta,b,q,t_{0}}(f_{\mathrm{in}})\big(1+t \big)^{2}.
\end{align*}
Arguing as in the proof of \cite[Eq. (B.3), Lemma B.4]{fisher}, {introducing $A=\{v \in \R^{3}\;;\;f(t,v) <1\}$, one checks easily that
\begin{equation*}
-\bm{S}_{b,q}(t)=-\int_{\R^{3}}\mu_{b,q}(v)f(t,v)\left|\log f(t,v)\right|\d v +2\int_{A}\mu_{b,q}(v)f(t,v)\log\left(\frac{1}{f(t,v)}\right)\d v.\end{equation*}
Introducing then $B:=\{v \in\R^{3}\;;\;f(t,v)\geq \exp\left(-{3}\,b\langle v\rangle^{q}\right)\}$, one splits the integral over $A$ into the integral over $A \cap B$ and $A \cap B^{c}$. On the one hand, for $v \in A \cap B$, $\log \frac{1}{f(t,v)} \leq  {3}\,b\langle v\rangle^{q}$ and, for any $\delta >0$, there exists $C_{\delta}=C(\delta,q,b) >0$ such that
$$2\int_{A \cap B}\mu_{b,q}(v)f(t,v)\log\left(\frac{1}{f(t,v)}\right) \d v  \leq C_{\delta}\int_{\R^{3}}\mu_{b+\delta,q}(v)f(t,v)\d v=C_{\delta}\vr_{b+\delta,q}(t).$$
On the other hand, for $v \in A\cap B^{c}$, using that $x\log\frac{1}{x} \leq \frac{2}{e}\sqrt{x}$, one has
$$f(t,v)\log\left(\frac{1}{f(t,v)}\right) \leq \frac{2}{e}\sqrt{f(t,v)} \leq \frac{2}{e}\exp\left( {-\frac{3}{2}}b\langle v\rangle^{q}\right)\,,$$
so that
$$\int_{A\cap B^{c}}\mu_{b,q}(v)f(t,v)\log\left(\frac{1}{f(t,v)}\right)\d v \leq \frac{2}{e}\int_{\R^{3}}\exp\left(- {\frac{1}{2}}b\langle v\rangle^{q}\right)\d v=:C_{b,q} <\infty.$$
This shows that, for any $\delta >0$,
$$-\bm{S}_{b,q}(t) \leq -\int_{\R^{3}}\mu_{b,q}(v)f(t,v)|\log f(t,v)|\d v+C_{\delta}\vr_{b+\delta,q}(t)+2C_{b,q}\,.$$}  Moreover, we deduce from  {Eq. \eqref{eq:flogf-exp} in Lemma \ref{lem:flogf-exp} together with Prop. \ref{prop:exp-mom} that $\bm{S}_{b,q}( {t_0})$ is finite under our assumption on $f_{\rm in}$}   and 
$$K_{0}\int_{ {t_0}}^{t}\d\tau\int_{\R^{3}}\langle v\rangle^{\g}\mu_{b,q}(v)|\nabla \sqrt{f(\tau,v)}|^{2}\d v \\
\leq C\big( 1+t \big)^{2},$$
for some finite $C >0$ depending on $b,q,\delta, {t_0} >0$ and $f_{\mathrm{in}}$.
\end{proof}

We deduce from this the following estimate for $\Gamma_{\dd}^{(a,q)}(f(t))$.

\begin{cor}\label{cor:Gamma}Assume that $-2 < \g < 0$  and let a nonnegative initial datum $f_{\mathrm{in}}$ satisfying \eqref{hypci}--\eqref{eq:Mass} for some $\dd_0 >0$ be given. For $\dd \in (0,\dd_0]$, let  $f(t,\cdot)$ be a weak-solution to \eqref{LFD}. Let $a >0$ and $0<q <{\frac{4+2\g}{8-\g}}$, and assume that for some $\delta >0$,
$$\int_{\R^{3}} f_{\mathrm{in}}(v) \exp\left(\left(2a+\delta\right){ 2^{\frac{q(1+|\g|)}{|\g|}}}\langle v\rangle^{q}\right)\d v < \infty.$$
Then, there exists $C_{\delta,a,q}(f_{\mathrm{in}})>0$ depending only on $\delta,a,q$ and $f_{\mathrm{in}}$ such that
{$$\int_{t_{0}}^{t}\Gamma_{\dd}^{(a,q)}(f(\tau))\d \tau \leq C_{\delta,a,q}(f_{\mathrm{in}})\,\sup_{t_0\le \tau\le t} \frac{\vr_{a+\delta,q}(\tau)}{\kappa_{0}(\tau)}\; (1+t)^{2} \,,\qquad 0< t_{0} < t\,, 
 $$}
where we recall that $\kappa_{0}(\tau)=1-\dd\,\|f(\tau)\|_{L^{\infty}}$, $\tau \geq0.$
\end{cor}
\begin{proof} The proof follows from Lemma \ref{lem:Ds-est-exp}, Proposition \ref{prop:weightFish-mu} with $b=a$, and the fact that $\langle v\rangle^{2}\mu_{a,q}(v) \leq {C_{\delta,q}} {\langle v\rangle^{\g}}\mu_{a+\delta,q}(v)$ for any $\delta >0$.
\end{proof}

\begin{theo}\label{theo:expon}Assume that $-2 < \g < 0$  and let a nonnegative initial datum $f_{\mathrm{in}}$ satisfying \eqref{hypci}--\eqref{eq:Mass} for some $\dd_0 >0$ be given. For $\dd \in (0,\dd_0]$, let  $f(t,\cdot)$ be a weak-solution to \eqref{LFD}. Let $a_0 >0$ and $0<q <{\frac{4+2\g}{8-\g}}$, and assume that
$$\int_{\R^{3}}f_{\mathrm{in}}(v) \exp\big(a_0\langle v\rangle^{q}\big)\d v < \infty.$$
Then,  there exists some explicit $\lambda >0$ depending only on $a_0,q$ and $f_{\rm in}$  such that, for any $\dd \in (0,\dd_{\star})$ (where $\dd_{\star}$ is prescribed by Theorem \ref{theo:main}),
$$\mathcal{H}_{\dd}(f(t))|\M_{\dd}) \leq \max\left(1,\mathcal{H}_{\dd}(f_{\mathrm{in}}|\M_{\dd})\right)\exp\left(-\lambda\,t^{\frac{q}{q-\g}}\right), \qquad t \geq 2.$$
As a consequence,
$$\left\|f(t)-\M_{\dd}\right\|_{L^{1}} \leq \sqrt{\max\left(2,2\mathcal{H}_{\dd}(f_{\mathrm{in}}|\M_{\dd})\right)}\,\exp\left(-\frac{\lambda}{2}\,t^{\frac{q}{q-\g}}\right)\,, \qquad t \geq 2\,.
$$
\end{theo}

\begin{proof} We first observe that, thanks to Theorem \ref{theo:main} and under the assumptions on the initial datum $f_{\mathrm{in}}$, there is $\dd_{\star} \in (0,\dd_{0}]$ depending only on $f_{\mathrm{in}}$  such that for any $\dd \in (0,\dd_{\star}]$,
$$\kappa_{0} = 1 - \dd\,\sup_{t\geq1}\| f(t) \|_{\infty}>0.$$
Let us write 
$$\bm{y}(t)=\mathcal{H}_{\dd}(f(t)|\M_{\dd}), \qquad t\geq0.$$
{One uses then \eqref{eq:inequExp} and \eqref{eq:D0lamb} which, by Proposition \ref{prop:D0ddgen} and Theorem \ref{theo:main}, actually holds for $-2<\gamma<0$. We first  deduce that, for $\dd \in (0,\dd_{\star})$, $t \geq 1$, and $a >0$}, {$q \in (0,1)$},
\begin{align*}
\mathscr{D}_{\dd}^{(\g)}(f(t)) 
&\geq \frac{1}{2}\left[\frac{1}{a} \log\left(\frac{\Gamma_{\dd}^{(a,q)}(f(t))}{\mathscr{D}_{\dd}^{(0)}(f(t))}\right)\right]^{\frac{\g}{q}}\,\mathscr{D}^{(0)}
_{\dd}(f(t))\\
&\geq \frac{\bar{\lambda}_{0}}{2}\left[\frac{1}{a} \log\left(\frac{\Gamma_{\dd}^{(a,q)}(f(t))}{\bar{\lambda}_{0}\bm{y}(t)}\right)\right]^{\frac{\g}{q}}\,\bm{y}(t),
 \end{align*}
 where we recall that we already know that $\frac{\Gamma_{\dd}^{(a,q)}(f(t))}{\mathscr{D}_{\dd}^{(0)}(f(t))} >1$,
 so that $ \Gamma_{\dd}^{(a,q)}(f(t)) \geq {\bar{\lambda}_{0}\, \bm{y}(t)}
$. We deduce then that
$$\dfrac{\d}{\d t}\bm{y}(t) \leq  -\frac{\bar{\lambda}_{0}}{2}\left[\frac{1}{a} \log\left(\frac{\Gamma_{\dd}^{(a,q)}(f(t))}{\bar{\lambda}_{0}\bm{y}(t)}\right)\right]^{\frac{\g}{q}}\,\bm{y}(t), \qquad t \geq 1.$$
Using Gronwall Lemma, we get
\begin{equation}\label{eq:yGr}
\bm{y}(t) \leq \bm{y}(1) \exp\left\{-\frac{\bar{\lambda}_{0}}{2a^{\frac{\g}{q}}}\int_{1}^{t}\left[- \log\left(\frac{\bar{\lambda}_{0}\bm{y}(\tau)}{\Gamma_{\dd}^{(a,q)}(f(\tau))}\right)\right]^{\frac{\g}{q}}\d \tau\right\}, \qquad t\geq1.\end{equation}
We introduce
$$I_{q}(t):=\int_{1}^{t}\left[- \log\left(\frac{\bar{\lambda}_{0}\bm{y}(\tau)}{\Gamma_{\dd}^{(a,q)}(f(\tau))}\right)\right]^{\frac{\g}{q}}\d \tau, \qquad t \geq 2\,,$$
so that 
\begin{align*}
I_{q}(t) &\geq \int_{\frac{t}{2}}^{t}\left[- \log\left(\frac{\bar{\lambda}_{0}\bm{y}(\tau)}{\Gamma_{\dd}^{(a,q)}(f(\tau))}\right)\right]^{\frac{\g}{q}}\d \tau\\
&=\int_{\frac{t}{2}}^{t}\left[\log\Gamma_{\dd}^{(a,q)}(f(\tau))-\log\bar{\lambda}_{0}-\log\bm{y}(\tau)\right]^{\frac{\g}{q}}
\d\tau, \qquad t \geq 2.
\end{align*}
Assume now that there is $t_{0} >2$ and some $m {>} 0$ such that
\begin{equation}\label{eq:alter}
\bm{y}(t_{0}) \geq \exp\left(-\left(\frac{t_{0}}{2}\right)^{m}\right)\,.
\end{equation}
Then, since $ \tau \mapsto \bm{y}(\tau)$ is non increasing, one has
$$\bm{y}(\tau) \geq \bm{y}(t_{0}) \geq \exp\left(-\tau^{m}\right)  \ge \exp\left(-t_{0}^{m}\right)\,, \qquad \tau \in \left(\frac{t_{0}}{2},t_{0}\right)\,,$$
and
$$I_{q}(t_{0}) \geq \int_{\frac{t_{0}}{2}}^{t_{0}}\left[\log\Gamma_{\dd}^{(a,q)}(f(\tau))-\log\bar{\lambda}_{0}+t_{0}^{m}\right]^{\frac{\g}{q}}
\d\tau.$$
Using now that, for any $\alpha \in \R,$ the function $r >\exp(-\alpha) \mapsto (\alpha+\log r)^{\frac{\g}{q}}$ is convex,
 and applying it with $\alpha=-\log\bar{\lambda}_{0}+t_{0}^{m}$, we deduce from Jensen's inequality that
\begin{align*}
I_{q}(t_{0}) &\geq  \frac{t_{0}}{2}\int_{\frac{t_{0}}{2}}^{t_{0}}\left[\log\Gamma_{\dd}^{(a,q)}(f(\tau))+\alpha\right]^{\frac{\g}{q}}
\frac{2\,\d\tau}{t_{0}}\\
&\geq \frac{t_{0}}{2}\left[t_{0}^{m}-\log\bar{\lambda}_{0}+
\log\left(\frac{2}{t_{0}}\int_{\frac{t_{0}}{2}}^{t_{0}}\Gamma_{\dd}^{(a,q)}(f(\tau))\,\d\tau\right)\right]^{\frac{\g}{q}}.
\end{align*}
Using Corollary \ref{cor:Gamma} together with Proposition \ref{prop:exp-mom}, choosing parameters $a,\delta>0$ such that $a_0=\Big(2a+\delta\Big){ 2^{\frac{q(1+|\g|)}{|\g|}}}$, 
there are positive constants $C_{0},C_{1}>0$ (independent of $t_{0}$) such that
$$\int_{\frac{t_{0}}{2}}^{t_{0}}\Gamma_{\dd}^{(a,q)}(f(\tau))\d\tau \leq {C_{0}\left(1+t_{0}\right)^{{3}}}\,,$$
so that
$$I_{q}(t_{0}) \geq \frac{t_{0}}{2}\Big[t_{0}^{m}-\log\bar{\lambda}_{0}+\log  (2C_{0})+{3}\log(1+t_{0}) -  \log t_0\Big]^{\frac{\g}{q}}\,.$$
Consequently, there exists $C_{2} >0$ such that
$$I_{q}(t_{0}) \geq C_{2}\, t_{0}^{1+\frac{m\g}{q}}.$$ 
Inserting this in \eqref{eq:yGr}, we deduce now
\begin{equation}\label{eq:alter1}
\bm{y}(t_{0}) \leq \bm{y}(1)\exp\left(-\frac{\bar{\lambda}_{0}\,C_{2}}{2a^{\frac{\g}{q}}}\, t_{0}^{1+m\frac{\g}{q}}\right).\end{equation}
Since we proved that assumption \eqref{eq:alter} implies estimate \eqref{eq:alter1}, we see that for any $t >2$ and any $m  {>} 0$, we have the following alternative:
\begin{itemize}
\item [$(i)$] either $\;\;\bm{y}(t) \leq \exp\left(-\left(\frac{t}{2}\right)^{m}\right)$,
\item [$(ii)$] or $\;\;\bm{y}(t) \leq \bm{y}(1)\exp\left(-\frac{\bar{\lambda}_{0}C_{2}}{2a^{\frac{\g}{q}}}\,t^{1+m\frac{\g}{q}}\right)$.
\end{itemize}
At this state, choosing $m >0$ so that $m=1+m\frac{\g}{q}$ (that is $m=\frac{q}{q-\g}$), we see that
$$\bm{y}(t) \leq \max\big(1,\bm{y}(1)\big)\exp\big(-c_{a}t^{m}\big), \qquad  t \geq 2\,,$$
for some explicit  
$c_{a}:=\min\left(2^{-m},\frac{\bar{\lambda}_{0}C_{2}}{2a^{\frac{\g}{q}}}\right).$ This concludes the proof.
\end{proof}

\appendix

\section{Regularity estimates}\label{ree}

We collect here several \emph{a priori} regularity estimates for the solutions to \eqref{LFD}. Clearly, it is possible to interpolate between $L^{1}$ and $L^{\infty}$ thanks to Theorem \ref{Linfinito*} to obtain a control on the $L^{p}$-norms with $1<p<\infty$. We adopt another approach here which consists in {directly investigating the evolution of  the $L^{p}$-norms}:
\begin{prop}\label{prop_Lp}
Assume that $-2<\gamma<0$. Let $p\ge 1$ and  $f_{\mathrm{in}}\in L^p(\R^{3}) \cap L^{2}(\R^{3})\cap L^1_{\bm{z}_{p}}(\R^{3}) $ satisfying \eqref{hypci}--\eqref{eq:Mass} for some $\dd_0 >0$, with
\begin{equation}\label{def_zp}\bm{z}_{p}:=\frac{1}{p}\begin{cases}(3p-2)|\g| \qquad &\text{ if } \,\g \leq -1\,,\\
3p-1+\g \qquad &\text{ if }\, \g \in (-1,0)\,.\end{cases}
\end{equation}
 Let $\dd \in (0,\dd_0]$ and let $f(t,v)$ be a weak solution to \eqref{LFD}. Then, there exists some constant $C_p(f_{\mathrm{in}})$ depending on $p$  and $f_{\mathrm{in}}$ such that, for every $T>0$,
\begin{equation}\label{ran}
  \sup_{t \in [0,T)}\left\|f(t,\cdot)\right\|^p_{L^{p}} + \int_{0}^{T}\int_{\R^{3}}\langle v\rangle^{\frac{\g}{2}}\left|\nabla f^{\frac{p}{2}}(t,v)\right|^{2}
\d v \d t \leq C_{p}(f_{\mathrm{in}})\left(1+T\right)^{p+1}.\end{equation}
\end{prop}

\begin{proof} 
We start with the formulation \eqref{LFD}:
\begin{equation}\label{LFD2}
\left\{
\begin{array}{ccl}
\;\partial_{t} f &= &\grad \cdot \big(\,\bm{\Sigma}[f]\, \grad f
- \bm{b}[f]\, f(1-\dd f)\big),\medskip\\
\;f(t=0)&=&f_{\mathrm{in}}\,.
\end{array}\right.
\end{equation}
For $p \ge 1$, multiplying this identity with $f^{p-1}(t,v)$ and integrating over $\R^{3}$, one deduces
\begin{multline*}
\frac{1}{p}\dfrac{\d}{\d t}\int_{\R^{3}} f^{p}(t,v)\d v  + (p-1) \int_{\R^{3}} f^{p-2}(t,v) \, [ \bm{\Sigma}[f(t)](v) \nabla f(t,v) ] \cdot \nabla f(t,v) \d v \\
=(p-1)\int_{\R^{3}} f^{p-2}(t,v) F(t,v) \,\bm{b}[f(t)](v) \cdot \nabla f(t,v), \qquad F=f(1-\dd f).\end{multline*}
Using the coercivity estimate in Proposition \ref{diffusion}, $|F| \le f$, and noticing that $f^{\frac{p}{2}}\nabla f^{\frac{p}{2}}=\frac{p}{2}f^{p-1}\nabla f$, we obtain after integration between $0$ and $T$:
\begin{multline*}
\frac{1}{p}\int_{\R^{3}}  f^{p}(T,v)\, \d v  + (p-1) \, K_0 \int_0^T \d t \int_{\R^{3}} \langle v\rangle^{\g} f^{p-2}(t,v)\, |\nabla f(t,v)|^2\d v \\
\leq
 \frac{1}{p}\int_{\R^{3}}  f^{p}_{\mathrm{in}}(v)\, \d v + \frac{2(p-1)}{p} \int_0^T\d t  \int_{\R^{3}} f^{\frac{p}{2}}(t,v) \, \left|\bm{b}[f](t,v)\right|\,  |\nabla  f^{\frac{p}{2}}(t,v)|\d v\,.
\end{multline*} 
Note that from this point on, the estimates do not use $\dd$ and are thus similar to that of the usual Landau equation. Since $f^{p-2}|\nabla f|^{2}=\frac{4}{p^{2}}|\nabla f^{\frac{p}{2}}|^{2}$, using Young's inequality, we get 
\begin{multline*}
\int_{\R^{3}}  f^{p}(T,v)\d v + \frac{2K_{0}\,(p-1)}{p}\int_0^T\d t \int_{\R^{3}} \langle v	\rangle^{\g}  |\nabla  f^{\frac{p}{2}}(t,v)|^2\d v \\
\leq 
\int_{\R^{3}}  f^{p}_{\mathrm{in}}(v)\d v + \frac{p(p-1)}{2 K_0} \int_0^T \d t \int_{\R^{3}} \langle v\rangle^{-\gamma}   f^{p}(t,v)\,\left|\bm{b}[f](t,v)\right|^{2}\d v\\
\leq \int_{\R^{3}} f^{p}_{\mathrm{in}}(v)\d v +\frac{p(p-1)C(f_{\mathrm{in}})^{2}}{2K_{0}}\int_{0}^{T}\d t\int_{\R^{3}}f^{p}(t,v)\langle v\rangle^{\max(-\g,2+\g)}\d v ,
\end{multline*}
where we used {\eqref{eq:bm2}, \eqref{eq:bm3} and \eqref{propLMs}} in the last term, $C(f_{\mathrm{in}})$ depending only on $\|f_{\mathrm{in}}\|_{L^{2}}$ and $\|f_{\mathrm{in}}\|_{L^{1}_{2}}$.
Since $$\langle v\rangle^{\gamma} \left|\nabla  f^{\frac{p}{2}}(t,v)\right|^2 \geq \frac{1}{2}\left|\nabla \left(\langle v\rangle^{\gamma/2} \,  f^{\frac{p}2}(t,v)\right)\right|^2-\left|\nabla (\langle v\rangle^{\frac{\g}{2}})\right|^2 f^{p}(t,v)\,,$$
we get that
\begin{multline}\label{LFD2*} 
\int_{\R^{3}}  f^{p}(T,v)\, \d v + \frac{p-1}{p} \, K_0 \int_0^T \d t\int_{\R^{3}}\left|\nabla \left(\langle v\rangle^{\gamma/2} \,  f^{\frac{p}2}(t,v)\right)\right|^2 \d v \\
\leq
\int_{\R^{3}}  f^{p}_{\mathrm{in}}(v)\d v + C_{0}p(p-1)  \int_0^T \d t \int_{\R^{3}} \langle v\rangle^{\max(-\g,2+\g)} \,   f^{p}(t,v)\d v ,
\end{multline}
for some positive constant $C_{0}$ depending only on $\|f_{\mathrm{in}}\|_{L^{1}_{2}}$, $\|f_{\rm in}\|_{L^{2}}$ and $H(f_{\rm in})$.

\medskip
\noindent
Choosing now $q >1$, $a >0$ and $\theta \in (0,1)$  such that $q'\,\theta =1$ and $$q(p - \theta)= 3p,\qquad  \frac{1}{q}+\frac{1}{q'}=1,$$
and applying H\"older's inequality we see that with $\varpi=\max(-\g,2+\g)$,
\begin{equation*}
\begin{split}  
\int_{\R^{3}} \langle v\rangle^{\max(-\g,2+\g)} \,   f^{p}(t,v)&\d v =    \int_{\R^{3}} \langle v\rangle^{\varpi + a} \,   f(t,v)^{\theta} \, 
 \langle v\rangle^{-  a} \,   f(t,v)^{p - \theta}\d v\\
& \leq \bigg[ \int_{\R^{3}} \langle v\rangle^{(\varpi + a) q'} \,   f(t,v)^{q' \theta}\d v \bigg]^{\frac{1}{q'}} \,
\bigg[ \int_{\R^{3}} \langle v\rangle^{- a q} \,   f(t,v)^{q(p- \theta)}\d v \bigg]^{\frac{1}{q}} \\
& \leq \bigg[ \int_{\R^{3}} \langle v\rangle^{\frac{\varpi + a}{\theta} } \,   f(t,v)\d v \bigg]^{\theta} \,
\bigg[ \int_{\R^{3}}  \bigg(\langle v\rangle^{- \frac{a q}{6}} \,   f(t,v)^{\frac{p}{2}}  \bigg)^6 \d v\bigg]^{\frac{1}{q}}\,.
\end{split}
\end{equation*}
Using the Sobolev inequality \eqref{eq:sobC}, we conclude that
$$\int_{\R^{3}} \langle v\rangle^{\max(-\g,2 +\g)} \,f^p(t,v)\d v  \le  C_{\mathrm{Sob}}^{\frac{6}{q}}\lm_{\frac{\varpi+a}{\theta}}(t)^{\theta} \,   \,
\bigg[ \int_{\R^{3}}\bigg| \nabla \bigg(\langle v\rangle^{-\frac{a q}{6}} \,   f^{\frac{p}{2}}(t,v)  \bigg) \bigg|^2\d v \bigg]^{3 \, (1 - \theta)}  . $$
At this point, observe that $\theta = \frac{2p}{3p-1}$ and select $a$ in such a way that $aq = 3|\gamma|$.  Thus,
\begin{equation}\label{eq:varpia}
\int_{\R^{3}} \langle v\rangle^{\max(-\g,2 +\g)} \,f^p(t,v)\, \d v  \le C_{p}\lm_{\bm{z}_{p}}(t)^{\frac{2p}{3p-1}} 
\bigg[ \int_{\R^{3}} \bigg| \nabla \bigg(\langle v\rangle^{\frac{\g}{2}} \,   f^{\frac{p}{2}}(t,v)  \bigg) \bigg|^2\d v \bigg]^{ \frac{3p-3}{3p-1} },
\end{equation}
with $\bm{z}_{p}=\frac{\varpi+a}{\theta}$.  Note that, for $\g \in (-2,0)$, $\bm{z}_{p} \leq \frac{6p-4}{p} < 6$.  Then, using Young's inequality it holds for any $\delta >0$,
$$ x^{\frac{2p}{3p-1}} \, y^{\frac{3p-3}{3p-1}} \le \delta^{-\frac{p-1}{2}}\, x^{p} + \delta^{\frac{1}{3}} \,y\,, \qquad \forall\, x,\,y >0.$$
Choosing $C_{0}\,C_{p}\,\delta^{\frac{1}{3}} =  \frac{1}{2p^{2}} \, K_0$, we get after combining \eqref{LFD2*} and \eqref{eq:varpia} that
\begin{multline*} 
\int_{\R^{3}}  f^{p}(T,v)\d v + \frac{p-1}{2p} \, K_0 \int_0^T \d t\int_{\R^{3}}\left|\nabla \left(\langle v\rangle^{\frac{\gamma}{2}} \,  f^{\frac{p}{2}}(t,v)\right)\right|^2 \d v \\   \leq
\int_{\R^{3}}  f^{p}_{\mathrm{in}}(v)\d v  + \tilde{C}_{p}\int_{0}^{T}\lm_{\bm{z}_{p}}(t)^{p}\d t .\end{multline*}
One sees from the second part of \eqref{res19} in Theorem \ref{theo:main-moments} that
$$\int_{0}^{T}\lm_{\bm{z}_{p}}(t)^{p}\d t \leq C_{p}(f_{\rm in})\left(1+T\right)^{p+1}\,,$$
for some  $C_{p}(f_{\rm in})$ depending only on $\|f_{\rm in}\|_{L^{1}_{2}}$,  and $H(f_{\rm in})$  and $p$. We deduce from this estimate \eqref{ran}.
 \end{proof}

For simplicity of notations, we introduce here $L^{1}_{\infty}(\R^{3}) :=\bigcap_{s\geq 0}L^{1}_{s}(\R^{3})$
as the space of integrable functions with finite moments of \emph{any order}.

\begin{cor}\label{W1p}
  Assume $-2<\gamma<0$ and let $q_0 >1$. We assume that 
$f_{\mathrm{in}}\in L^{q_0}(\R^3) \cap  L^{2}(\R^{3})\cap  L^1_{\infty}(\R^{3})$
satisfies \eqref{hypci}--\eqref{eq:Mass}  for some $\dd_0 >0$. Let $\dd \in (0,\dd_0]$ and $f(t,v)$ be a weak solution to \eqref{LFD}. Then, for any $m >0, q \in [1,q_{0})$ and for any $T>0$, there  exists some constant $C$ depending on $q_0$, $m$, $q$, $f_{\rm in}$ such that 
\begin{align}\label{rap} 
\sup_{t \in [0, T]} \int_{\R^{3}}  \langle v\rangle^{m} \,  f(t,v)^q\, \d v < C\, (1 + T)^{1 + q_0}\,.
\end{align}
\end{cor}

\begin{proof}
This is a direct consequence of an interpolation between the spaces $L^{q_0}$ and $L^1_{s}$, using Proposition \ref{prop_Lp} and the universal growth of the $L^{1}$-moments in Theorem \ref{theo:main-moments}.
\end{proof}

This corollary will be used in a crucial way to prove the following proposition. We introduce here the notation $L^{\infty-0}(\R^{3}) :=\bigcap_{q \geq 1}L^{q}(\R^{3})$.
\begin{prop}\label{nner}
Assume $-2<\gamma<0$ and let $f_{\mathrm{in}}\in  L^1_{\infty}(\R^{3}) \bigcap L^{\infty-0}(\R^3)$   satisfying \eqref{hypci}--\eqref{eq:Mass}  for some $\dd_0 >0$. Let $\dd \in (0,\dd_0]$ and $f(t,v)$ be a weak solution to \eqref{LFD}. For any choice of $m \ge 0$ and {$p \geq 2$},
 if $f_{\rm in} \in W^{1,p}_{m}(\R^{3})$, then there is some $C_T(f_{\mathrm{in}})$ depending on $m$, $p$, $T$ and $f_{\mathrm{in}}$ such that
\begin{equation}\label{rapp} 
 \sup_{t\in [0,T]}  \int_{\R^3}  \langle v\rangle^{m} \, | \nabla  f(t,v)  |^p  \d v+  \sum_{i=1}^3  \int_0^T\d t\int_{\R^3} \langle v\rangle^{ m +\g} \, |\pa_i f|^{p-2}\,  |\nabla \pa_i f|^2 \d v  \le C_T(f_{\mathrm{in}}).
 \end{equation} 
\end{prop}

\begin{proof}
Taking the derivative of \eqref{LFD} with respect to a component $i \in \{1,2,3\}$,
we end up with the equation 
 \begin{equation}\label{LFDder}
\left\{
\begin{array}{ccl}
\;\partial_{t} (\pa_i f) &= &\grad \cdot \big(\,\bm{\Sigma}[f]\, \grad (\pa_i f) 
+ [\pa_i \bm{\Sigma}[f] ] \grad f 
- \bm{b}[f]\, \pa_i F - [\pa_i  \bm{b}[f]] \, F \big),\medskip\\
\;\pa_i f(0, \cdot)&=&\pa_i f_{\mathrm{in}}\,.
\end{array}\right.
\end{equation}
Compute
\begin{multline*}
\frac{1}{p} \frac{\d}{\d t} \sum_i \int_{\R^3} \langle v\rangle^m\, |\pa_i f|^{p} \d v =
 (p-1)\sum_i   \int_{\R^3}  \langle v\rangle^m\,|\pa_i f|^{p-2}\bigg\{   \nabla (\pa_i f) \cdot  \bm{b}[f] \pa_i F + \\\nabla (\pa_{i}f) \cdot \pa_{i}\bm{b}[f]\,F
 -
\bm{\Sigma}[f] \nabla (\pa_i f) \nabla (\pa_i f) 
-  \nabla (\pa_i f) \pa_i \bm{\Sigma}[f]   \nabla f\bigg\} \d v\\
- \sum_i  \int_{\R^3} |\pa_i f|^{p-2} \pa_i f \, \bm{\Sigma}[f] \nabla (\pa_i f)  \nabla \langle v\rangle^m  \d v
 - \sum_i \int_{\R^3}  \nabla \langle v\rangle^m\, |\pa_i f|^{p-2} \pa_i f \,  [\pa_i \bm{\Sigma}[f]]   \nabla f \d v\\
+ \sum_i \int_{\R^3}  \nabla \langle v\rangle^m \cdot  \bm{b}[f]\, \, |\pa_i f|^{p-2} \pa_i f \, \pa_i F \d v+ 
 \sum_i  \int _{\R^3} \nabla \langle v\rangle^m \cdot  [\pa_i  \bm{b}[f]]  \, |\pa_i f|^{p-2} \pa_i f \, F \d v.
\end{multline*}
Using the coercivity estimate in Proposition \ref{diffusion} and Young's inequality we see that
\begin{multline}\label{LFDgrad*}
  \frac{\d}{\d t}  \sum_i \int_{\R^3} \langle v\rangle^m {|\pa_i f|^{p}} \d v +  \frac{K_0}{2}\;{p(p-1)}  \sum_i  \int_{\R^3} \langle v\rangle^{ m +\g} \, |\pa_i f|^{p-2}\,  |\nabla \pa_i f|^2 \d v\\
\le C_p \int_{\R^3} \langle v\rangle^{m + |\gamma|}  |\nabla f|^{p-2} \bigg[ |\nabla f|^2 |\nabla  \bm{\Sigma}[f] |^2 + F^2\, |\nabla \bm{b}[f]|^2\\
+  |\nabla F|^2\, | \bm{b}[f]|^2 + \langle v\rangle^{-2} |\nabla f|^2 | \bm{\Sigma}[f] |^2  \bigg] \d v\\
+ C_p \int_{\R^3}   \langle v\rangle^{m-1} \, |\nabla f|^{p}  \,  |\nabla \bm{\Sigma}[f]|  \d v+ C_p \int_{\R^3}  \langle v\rangle^{m-1}\, |\nabla f|^{p-1}  \, \bigg( |\bm{b}[f]|\, |\nabla F| + |\nabla  \bm{b}[f]| \, F \bigg)\d v,
\end{multline}
where $C_p>0$ only depends on $p$ {and $m$}.  Then, we observe that 
\begin{equation}\label{sta1}
\big|\bm{\Sigma}[f] (v)\big| \le C(f_{\mathrm{in}})\langle v\rangle^{2 +\g}\,,\,\qquad  |\nabla F(v)| \le C \, (1+f(v))\, |\nabla f(v)|,
\end{equation}
 \begin{equation}\label{sta2}
\big| \nabla  \bm{\Sigma}[f] (v) \big| +  \big| \bm{b}[f](v) \big|  \le C(f_{\mathrm{in}})\, \langle v\rangle^{ \max\{ 0,1 +\g\} }\,, \qquad  |\nabla \bm{b}[f](t,v)| \le C\, (|\cdot|^{- |\g|} * f)(v),
\end{equation}
and  \eqref{LFDgrad*} leads to 
\begin{multline}\label{LFDgrad**}
 \frac{\d}{\d t}  \sum_i \int_{\R^3} \langle v\rangle^m {|\pa_i f|^{p}} \d v +  \frac{K_0}{2} \; {p(p-1)} \sum_i  \int_{\R^3} \langle v\rangle^{ m +\g} \, |\pa_i f|^{p-2}\,  |\nabla \pa_i f|^2 \d v \\
\le C_p(f_{\mathrm{in}})  \bigg( \int_{\R^3} \langle v\rangle^{ m + \max(|\g|, 2 - |\g|)} (1+f)^2 \,  |\nabla f|^{p} \d v
+ \int_{\R^3} \langle v\rangle^{m +|\g|} f^2 \,  |\nabla f|^{p-2} \, {( |\cdot|^{- |\g|} * f)^{2}}\,  \d v\\
+  \int_{\R^3} \langle v\rangle^{ m -1} f \,  |\nabla f|^{p-1}  \, ( |\cdot|^{- |\g|} * f)\,   \d v\bigg).
\end{multline}
 {For $p >2$,
using} Young's inequality  $x\,y \le x^a + y^{\frac{a}{a-1}}$ with $a=\frac{p}{p-2}$ for the second term of the right-hand side of the aforementioned estimate, and with $a =\frac{p}{p-1}$ 
for the third term, we conclude that for some $r_1, r_2 \ge 0$
\begin{multline}\label{LFDse}
 \frac{{\rm d}}{{\rm d}t} \sum_i \int_{\R^3} \langle v\rangle^m {|\pa_i f|^{p}} \d v +  \frac{K_0}{2}  {p(p-1)} \sum_i  \int_{\R^3} \langle v\rangle^{ m +\g} \, |\pa_i f|^{p-2}\,  |\nabla \pa_i f|^2 \d v \\
\le C_p(f_{\mathrm{in}})  \bigg( \int_{\R^3} \langle v\rangle^{ m + \max(|\g|, 2 - |\g|)} (1+f)^2 \,  |\nabla f|^{p} \d v
+ \int_{\R^3} \langle v\rangle^{r_1} f^2  \, ( |\cdot|^{- |\g|} * f)^p\,  \d v\\
+  \int_{\R^3} \langle v\rangle^{r_2} f  \, {( |\cdot|^{- |\g|} * f)^{p}}\,   \d v\bigg).
\end{multline}
{Such an inequality is also easily deduced from \eqref{LFDgrad**} when $p=2$.} 
We see then, thanks to Young's inequality, that for all $r_3\ge 0$, there exists $r_4\ge 0$ such that 
$$\int_{\R^3} \langle v\rangle^{r_3} f\,(1+f)  \, ( |\cdot|^{- |\g|} * f)^p\,  \d v \le \int_{\R^3} \langle v\rangle^{r_4} f^2\,(1+f)^2  \,  \d v $$
$$ +  {2^{p-1}}\int_{\R^3} \langle v\rangle^{-4} \,  ( |\cdot|^{- |\g|}\, \ind_{|\cdot|\le 1} * f)^{2p}\,  \d v
+  {2^{p-1}} \int_{\R^3} \langle v\rangle^{-4} \,  ( |\cdot|^{- |\g|}\, \ind_{|\cdot|\ge 1} * f)^{2p}\,  \d v $$
\begin{equation}\label{ll1} \le  \int_{\R^3} \langle v\rangle^{r_4} f^2\,(1+f)^2  \,  \d v  + C_p\, \|f\|_{L^{2p}}^{2p} + C\, \|f\|_{L^{1}}^{2p}\,,
\end{equation}
where we used that 
\begin{equation*}
\int_{\R^3} \langle v\rangle^{-4}\,  (|\cdot|^{- |\g|}\, \ind_{|\cdot|\le 1} \ast f)^{2p}(v)\, \d v 
\leq  \|\, |\cdot|^{- |\g|}\, \ind_{|\cdot|\le 1} \ast f \|_{L^{2p}}^{2p} 
\le  \| \, |\cdot|^{- |\g|}\, \ind_{|\cdot|\le 1}\|_{L^{1}}^{2p} \, \|f\|_{L^{2p}}^{2p}\,, 
\end{equation*}
and $\int_{\R^{3}}\langle v\rangle^{-4}\d v < \infty$.  Then, 
\begin{multline*}
\int_{\R^3} \langle v\rangle^{m + \max(|\g|, 2 - |\g|)} (1+f)^2 \,  |\pa_i f|^{p} \d v\\
= -  \int_{\R^3} \pa_i [  \langle v\rangle^{m + \max(|\g|, 2 - |\g|)} (1+f)^2 \,   |\pa_i f|^{p-2} \, \pa_i f ]\, f \d v\\ 
= -  \int_{\R^3} \pa_i [\langle v\rangle^{m + \max(|\g|, 2 - |\g|)}]\, (1+f)^2 \,f   |\pa_i f|^{p-2} \, \pa_i f \, \d v \\
 -  2\int_{\R^3}   \langle v\rangle^{m + \max(|\g|, 2 - |\g|)}  \,{f(1+f)}\,   |\pa_i f|^{p} \d v \\
 - (p-1) \int_{\R^3}\langle v\rangle^{m + \max(|\g|, 2 - |\g|)} (1+f)^2 \,f\,   |\pa_i f|^{p-2} \, \pa_{ii} f \d v .
\end{multline*}
{The second integral is nonnegative whereas we can estimate the third integral using Young's inequality to get, for any $\delta >0$,}
\begin{multline*}
\int_{\R^3} \langle v\rangle^{m + \max(|\g|, 2 - |\g|)}  {(1+f)^2} \,  |\pa_i f|^{p} \d v 
\leq C \int_{\R^3} \langle v\rangle^{m -1 + \max(|\g|, 2 - |\g|)}  {(1+f)^2} \,f\,  |\pa_i f|^{p-1} \d v\\
 +\, \delta  \int_{\R^3} \langle v\rangle^{m -|\g|} \,  |\pa_i f|^{p-2}\, |\pa_{ii} f|^2 \d v  +  {\frac{(p-1)^2}{4\delta}} \int_{\R^{3}} {\langle v\rangle^{m + \max(3|\g|, 4-|\gamma|)} (1+f)^4} \,f^2\,  |\pa_i f|^{p-2} \d v.\end{multline*}
{To estimate the first  integral, we now use Young's inequality in the form 
\begin{equation*}\begin{split}
\langle v\rangle^{-1}{(1+f)^2} \,f\,  |\pa_i f|^{p-1} &\leq \frac{1}{4}|\pa_{i}f|^{p}+ C_{p}\langle v\rangle^{-p}(1+f)^{2p}f^{p} \\
&\leq \frac{1}{4}(1+f)^{2}|\pa_{i}f|^{p}+ C_{p}\langle v\rangle^{-p}(1+f)^{2p}f^{p}\,,\end{split}\end{equation*}
while, for the  {third} integral, since $\max(3|\g|,4-|\g|)  {\le} \max(|\g|,2-|\g|)+2\max(|\g|,1)$,
 one can use Young's inequality in the form 
$$\frac{(p-1)^{2}}{ {4}\delta}\langle v\rangle^{2\max(|\g|,1)}(1+f)^{2}f^{2}|\pa_{i}f|^{p-2} \leq \frac{1}{4}|\pa_{i}f|^{p} + C_{\delta,p}\langle v\rangle^{ {p}
\, \max(|\g|,1)}(1+f)^{p}f^{p}\,,$$
for some positive constant $C_{\delta,p} >0$. Therefore, one can find $C_{\delta,p}>0$ such that 
\begin{multline*}\int_{\R^3} \langle v\rangle^{m + \max(|\g|, 2 - |\g|)}  {(1+f)^2} \,  |\pa_i f|^{p} \d v   \le  {\frac14} \int_{\R^3} \langle v\rangle^{m + \max(|\g|, 2 - |\g|)}  \,  (1+f)^{2}|\pa_i f|^{p} \d v \\
 + \,C  \int_{\R^3} \langle v\rangle^{r_5} {(1+f)^{2p}} \, f^p   \d v  + \delta  \int_{\R^3} \langle v\rangle^{m -|\g|} \,  |\pa_i f|^{p-2}\, |\pa_{ii} f|^2 \d v\\
 + \,  {\frac14} \int_{\R^3} \langle v\rangle^{m + \max(|\g|, 2 - |\g|)}  {(1+f)^2} \,  |\pa_i f|^{p} \d v  +   {C_{\delta,p}}  \int_{\R^3} \langle v\rangle^{r_6}  {(1+f)^{2 + p}} \, f^p   \d v ,
\end{multline*}
with $r_{5}:=m + \max(|\g|, 2 - |\g|)-p$, $r_{6}:=m + \max(|\g|, 2 - |\g|) +
 p\,\max(|\g|,1).$} As a consequence,
\begin{multline}\label{ll2}  {\frac12}\int_{\R^3} \langle v\rangle^{m + \max(|\g|, 2 - |\g|)} (1+f)^2 \,  |\pa_i f|^{p} \d v \le  \delta  \int_{\R^3} \langle v\rangle^{m -|\g|} \,  |\pa_i f|^{p-2}\, |\pa_{ii} f|^2 \d v\\
 + C_{\delta,p}  \int_{\R^3} \langle v\rangle^{\max(r_5,r_6)}   {(1+f)^{2 + 2p}} \, f^p   \d v, \qquad \forall \delta >0. 
\end{multline}
Using estimates \eqref{ll1}, \eqref{ll2}, \eqref{LFDse}, and remembering \eqref{rap}, we conclude the proof.
\end{proof}

Notice that in particular, thanks to Proposition \ref{nner}, a Sobolev estimate in the $v$ variable shows that $f$ satisfies an $L^{\infty}$ (local w.r.t. time) estimate.  More specifically, we have the following result
\begin{cor}\label{nner2}
Assume $-2<\gamma<0$, and let $f_{\mathrm{in}}\in L^{\infty-0}(\R^3) \cap L^1_{\infty}(\R^{3}) $ satisfying \eqref{hypci}--\eqref{eq:Mass} for some $\dd_0 >0$. Let $\dd \in (0,\dd_0]$ and let $f(t,v)$ be a weak solution to \eqref{LFD}. Then for any  {$p > 2$} and any $T >0$,  {if $f_{\mathrm{in}}\in W^{1,p}_{2}(\R^3)$}, there is some $C_T(p, f_{\mathrm{in}})$ depending on $p$, $T$ and $f_{\mathrm{in}}$, such that
\begin{equation}\label{rapp} 
\left\|\, \bm{\Sigma}[f]\, \nabla f  - \bm{b}[f]\, [f\,(1 - \var\,f)]\, \right\|_{L^{\infty}([0,T] ; L^p(\R^3))}   
 \le C_T(p, f_{\mathrm{in}})  . 
\end{equation} 
\end{cor}

\begin{proof} Observe that thanks to (\ref{sta1}) and (\ref{sta2}),
$$\left|\langle v\rangle^{-2-\g}\bm{\Sigma}[f]\right| \leq C\,\int_{\R^{3}}\langle \vet\rangle^{2+\g}f(\vet)\, \d \vet {\le C \|f_{\mathrm{in}}\|_{L^1_2}}\,,$$
so that
$\langle \cdot \rangle^{-2-\gamma}\bm{\Sigma}[f]$ is bounded in $L^{\infty}([0,T]\times \R^{3})$. In the same way $\langle \cdot \rangle^{\min\{-1-\gamma,0\}}\bm{b}[f]$ is bounded 
in  $L^{\infty}([0,T] \times \R^3)$. We conclude using Proposition \ref{nner}  {and Corollary \ref{W1p}}. 
\end{proof}
\begin{cor}\label{nner3} Assume $-2<\gamma<0$ and let $f_{\mathrm{in}}\in L^{\infty-0}(\R^3) \cap L^1_{\infty}(\R^{3})$ satisfying \eqref{hypci}--\eqref{eq:Mass} for some $\dd_0 >0$. Let $\dd \in (0,\dd_0]$ and  let $f(t,v)$ be a weak solution to \eqref{LFD}. Then for any  {$p > 2$}, $T>0$,  {if $f_{\mathrm{in}}\in W^{1,p}_{2}(\R^3)$} then there is  some $C_T(p, f_{\mathrm{in}})$ depending on $T$, $p$ and $f_{\mathrm{in}}$, such that
\begin{equation}\label{rapp2} 
\|f \|_{W^{\frac{1}{3},p}([0,T]  \times \R^3)}    \le C_T(p, f_{\mathrm{in}})  ,
\end{equation} 
and for any $\alpha \in (0, \frac{1}{3}), T >0$, there is some $C_T(\alpha, f_{\mathrm{in}})$ depending on $T$, $\alpha$ and $f_{\mathrm{in}}$, such that
\begin{equation}\label{rapp3} 
\|f \|_{C^{0,\alpha}({[0,T]}  \times \R^3)}   \le C_T(\alpha, f_{\mathrm{in}}).
\end{equation} \end{cor}

\begin{proof}
Using the equation and Corollary \ref{nner2}, we see that, for all $1  \leq p < \infty$, {if $f_{\mathrm{in}}\in W^{1,p}_{2}(\R^3)$ then}  $f$ is bounded in $W^{1,\infty}((0,T) ; W^{-1, p}(\R^3))$. Proposition \ref{nner} also ensures that
$f$ is bounded in $ L^{\infty}((0,T); W^{1, p}(\R^3))$. We get  inequality \eqref{rapp2}   thanks to an interpolation, and deduce \eqref{rapp3} from \eqref{rapp2} thanks to a Sobolev inequality.
\end{proof}
 
{We conclude this Appendix with the proof that (for suitable initial data) the solutions of the LFD equation with moderately soft potentials are in fact classical.}

\begin{cor}\label{cor:A6}
Let $\gamma \in (-2,0)$. Consider an initial datum $f_{\mathrm{in}}\in L^{\infty-0}(\R^3) \cap L^1_{\infty}(\R^{3})\cap W^{1,p}_2(\R^3)$ for some $p>2$ satisfying \eqref{hypci}--\eqref{eq:Mass} for some $\dd_0 >0$. For any $\dd\in(0,\dd_{0}]$, any weak solution $f$ to \eqref{LFD} given by Theorem \ref{existence} is actually a classical solution, that is $f$ is continuously differentiable with respect to $t$ and twice continuously differentiable with respect to $v$ on $(0, \infty)\times \R^3$. 
\end{cor}

\begin{proof}
We observe that $f$ is a weak solution to the linear equation (with unknown $u$)
\begin{equation*}
\partial_{t}u=\nabla \cdot \left( \bm{\Sigma}[f] \nabla u  \right)\: -(1-2\dd\,f)\bm{b}[f] \cdot \nabla u - \bm{c}_\g[f]\,(1-\dd f) u\,.\end{equation*}
Let $R>0$ and $\Omega=\{v\in \R^3 ; |v|\le R\}$. The  coefficients
$\bm{\Sigma}[f]$, $(1-2\dd\,f)\bm{b}[f]$, $\bm{c}_\g[f]\,(1-\dd\,f) $ and also $\nabla\bm{\Sigma}[f]$  are H\"older-continuous on $(0,T) \times \Omega$ for any $T >0$ 
thanks to Corollary \ref{nner3}
 and belong to $L^\infty((0,T) \times \Omega)$. We then deduce from  Proposition \ref{diffusion} and \cite[Chapter III, Theorem 12.1]{Lad} that $\partial_tf$ and   $\pa^2_{v_i v_j} f$ are also H\"older-continuous on $(0,\infty) \times \Omega$. 
\end{proof}
We now have all the ingredients for the
\begin{proof}[Proof of Theorem \ref{smoothn}]
The first statement in the Theorem is a direct consequence of Corollary \ref{W1p}, while the second one is obtained thanks to Proposition \ref{nner} and Corollary \ref{nner3}  {whenever $p > 2$. For $p \in [1,2]$, one deduces that $f \in L^{\infty}([0,T]; W^{1,p}_s(\R^3))$ by a simple interpolation.}  
\end{proof}

\section{About the Cauchy Theory}\label{app:cauchy}

We give the detailed proof of Theorem \ref{existence} about the existence of solutions to \eqref{LFD}.  We follow the approach of \cite{bag}. Let  $(\Psi_\nu)_{\nu\in(0,1)}$ be a family of smooth bounded functions on $\R_+$ that coincide with  $\Psi(r)=r^{\g+2}$ for  $0<\nu<r<\nu^{-1} $ and satisfy 
\begin{enumerate}[(i)]
\item The functions $\Psi_\nu'$,  $\Psi_\nu''$, $\Psi_\nu^{(3)}$ and  $\Psi_\nu^{(4)}$ are bounded.
\item The following hold
\begin{equation}\label{smallr}
 \Psi_\nu(r)\geq \frac{\nu^\gamma r^{2}}{2}\; \qquad \forall\, 0 < r <\nu\,, \qquad \quad  \Psi_\nu(r) \geq \frac{\nu^{-(2+\gamma)}}{2}>0, \qquad \forall \, r >\nu^{-1}\,.
\end{equation}
\item For any $r\in\R_+$, 
\begin{equation}\label{majpsinu}
\Psi_\nu(r)\leq 2 \,r^{2+\gamma} \qquad \text{ and } \qquad |\Psi'_\nu(r)|\leq C r^{1+\gamma},
\end{equation}
for some constant $C$ independent of $\nu$.
\end{enumerate}
\noindent
We then set  
\begin{align*}\begin{cases}
 a^\nu(z)& = \left(a^\nu_{i,j}(z)\right)_{i,j} \qquad \mbox{ with }
\qquad a^\nu_{i,j}(z) =  \Psi_\nu (|z|) \, 
\left( \delta_{i,j} -\frac{z_i  z_j}{|z|^2} \right), \\
 b^\nu_i(z) & = \sum_k \partial_k a^\nu_{i,k}(z) 
= - \: \frac{2 \, z_i}{|z|^2} \: \Psi_\nu(|z|),  \\
 c^\nu(z)& =\sum_{k,l} \partial^2_{kl} a^\nu_{k,l}(z)
=  - \: \frac{2}{|z|^2} \: \Big[\,\Psi_\nu(|z|) + |z| \,\Psi_\nu'(|z|)\,\Big],\end{cases}
\end{align*}
and we consider the following regularized problem
\begin{equation}\label{pbappr}
\left\{
\begin{array}{l}
\displaystyle \partial_t f = \nabla \cdot \Big(\bm{\Sigma}^{\nu}[f]
\grad f - \bm{b}^{\nu}[f]  f(1-\dd f) \Big) 
+ \nu \Delta  f\\
\displaystyle f(0,.)= f_{\mathrm{in}}\,,
\end{array}
\right.
\end{equation}
where, as above,  $\bm{\Sigma}^{\nu}[f]= a^{\nu}\ast (f(1-\dd f))$, $\bm{b}^{\nu}[f]=b^{\nu}\ast f$.\\

{We note here that the initial condition of the regularized problem is not assumed to satisfy \eqref{eq:Mass}. For such an initial condition, Lemma \ref{L2unif} still holds.} We first investigate the well-posedness of \eqref{pbappr} and prove the following result.   
\begin{prop}\label{solapp}
Consider $f_{\mathrm{in}} \in \C^\infty (\R^3)\cap H^1(\R^3) \cap W^{3,\infty}(\R^3)$  such that 
\begin{equation}\label{fd}
0<\alpha_1 e^{-\be_1 |v|^2} \leq f_{\mathrm{in}}(v)
\leq \frac{ \alpha_2 \, e^{ -\be_2 \, |v|^2}}
{1+  \dd\alpha_2 \, e^{-\be_2 \, |v|^2}} \qquad \mbox{  for every } v\in\R^3,
\end{equation}
for some positive constants $\alpha_1$, $\alpha_2$, $ \be_1$ and $\be_2$. Let $\nu>0$ and $T>0$. Then, there exists a solution $f^\nu$ to the regularized problem \eqref{pbappr} such that, for every $s>0$, $$f^\nu \in L^\infty((0,T);L^1_{s}(\R^3))\cap L^2((0,T);H^1_{s}(\R^3))\,.$$
\end{prop}

The proof of this Proposition can be easily adapted from the proof of \cite[Theorem 4.2]{bag}. One begins with freezing the non-local coefficients in \eqref{pbappr}. The smoothness and boundedness of $\Psi_\nu$ are used here in order to obtain some parabolic operator with smooth coefficients and deduce the existence of a unique classical solution from \cite[Chapter V, Theorem 8.1]{Lad}. Finally, some fixed-point argument enables to conclude.

In order to pass to the limit $\nu\to 0$ in \eqref{pbappr} and obtain a solution to \eqref{LFD}, we need to prove uniform estimates on $f^\nu$ (with respect to $\nu$). First, as in \cite[Lemma 4.8]{bag}, one has the lemma:  
\begin{lem}\label{massandenergy}
For any $\sigma, t\in[0,T]$, $\sigma\le t$, for any $\nu \in (0,1)$, the function $f^\nu$ satisfies 
\begin{eqnarray}
\int_{\R^3}f^\nu(t,v) \d v & = & \varrho \,,   \label{mass_app}\\
\int_{\R^3}f^\nu(t,v) |v|^2 \, \d v & = & {\theta}  +6\nu t \varrho   
\leq  {\theta}  +6 T \varrho  \,, \label{energy_app}\\
{\mathcal S}_{\dd}( f_{\mathrm{in}}) &\le & {\mathcal S}_{\dd}(f^\nu(\sigma)) \; \leq\;{\mathcal S}_{\dd}(f^\nu(t)) \,.\label{entropy_app}
\end{eqnarray}
{where $\varrho= \int_{\R^3} f_{\mathrm{in}}(v)\d v$ and $\theta=\int_{\R^3} f_{\mathrm{in}}(v) \,|v|^2\d v$.}
\end{lem}

Next, we consider the ellipticity of the diffusion matrix. As in \cite[Proposition 4.9 and Corollary 4.10]{bag}, one has the following proposition.
\begin{prop}\label{diffusionapp}
Let $0\leq f_{\mathrm{in}}\in L^{1}_{2}(\R^{3})$ be fixed and satisfying \eqref{hypci}  for some $\dd_0 >0$. Let $\dd \in (0,\dd_0]$ and  $R(f_{\mathrm{in}})$ and $\eta(f_{\mathrm{in}})$ be given by the first point of Lemma \ref{L2unif}. Let $\overline{\eta}$ be the constant given by the second point of Lemma \ref{L2unif} for $\delta=\eta(f_{\mathrm{in}})$.
Let 
$$0<\nu \leq \min\left\{(3R(f_{\mathrm{in}}))^{-1} , \left(\frac{3\overline{\eta}}{4\pi}\right)^{\frac{1}{3}},1\right\}.$$ 
Then, 
\begin{enumerate}
\item there exists a positive constant $K_{0} > 0$ depending on  $\gamma$,  {$\|f_{\mathrm{in}}\|_{L^1_2}$, and $H(f_{\mathrm{in}})$}, such that, for any $v,\, \xi \in \R^3$, 
$$
\sum_{i,j} \, \left( \bm{\Sigma}_{i,j}[f](v) +\nu \, \delta_{i,j} \right)\, \xi_i \, \xi_j 
\geq K_{0} \langle v \rangle^{\g} \,\min \left\{(\nu^{-1}|v|)^{-\gamma}, 2^{-\gamma}, 2(\nu|v|)^{-(2+\gamma)}\right\} |\xi|^2$$
holds for any $\dd \in (0,\dd_0]$ and $f \in \mathcal{Y}_{\dd}(f_{\mathrm{in}})$;
\item there exists a positive constant $\kappa > 0$ depending on $\gamma$,   {$\|f_{\mathrm{in}}\|_{L^1_2}$, and $H(f_{\mathrm{in}})$}, such that
$$\forall\, v,\, \xi \in \R^3, \qquad 
\sum_{i,j} \, \left( \bm{\Sigma}_{i,j}[f](v) +\nu \, \delta_{i,j} \right)\, \xi_i \, \xi_j \geq \kappa \; \frac{ |\xi|^2}{1+|v|^2},$$
holds for any $\dd \in (0,\dd_0]$ and $f \in \mathcal{Y}_{\dd}(f_{\mathrm{in}})$.
\end{enumerate}
\end{prop}

The proof of the first point of this Proposition can be easily adapted from that of \cite[Proposition 2.3]{ALL}. Indeed, $\Psi_\nu$ may be bounded from below thanks to \eqref{smallr}:
$$\Psi_\nu(r)\ge \min\left\{\frac{\nu^\gamma r^2}{2}, r^{2+\gamma}, \frac{\nu^{-(2+\gamma)}}{2}\right\} \qquad \mbox{for any } r\in\R_+.$$
 The second point follows easily from the proof of the first point by using that $\nu^{-1} \ge 3R(f_{\mathrm{in}})$. This gives some uniform (with respect to $\nu$) ellipticity estimate. 

\begin{lem}\label{2plusgama}
Let  $f_{\mathrm{in}}\in \C^\infty (\R^3)\cap H^1(\R^3) \cap W^{3,\infty}(\R^3)$ satisfying 
(\ref{fd}). Let $f^\nu$ be a solution to (\ref{pbappr}) given by Proposition \ref{solapp}. Then, for any $T>0$ and $s>2$, there exists some constant $C_s$ depending only on $s$, $T$ and  {$\|f_{\mathrm{in}}\|_{L^1_2}$} such that 
\begin{equation}\label{L1s}
\sup_{t\in[0,T]} \| f^\nu (t)\|_{L^1_{s}} \leq \|f_{\mathrm{in}}\|_{L^1_{s}} \:\exp\left({C_s T}\right).
\end{equation} 
\end{lem}

\begin{proof}
With notations similar to those in \eqref{eq:mom-s}, one has 
\begin{multline*}
 \frac{\d}{\d t} \int_{\R^3} f^\nu (t,v)\, \langle v\rangle^s \,\d v=\mathscr{J}^\nu_{s,1}(f^\nu,f^\nu)-\dd  \mathscr{J}^\nu_{s,1}(f^\nu,(f^\nu)^2)\\
+\mathscr{J}^\nu_{s,2}(f^\nu,F^\nu) + s\nu\int_{\R^3} f^\nu \langle v \rangle ^{s-4} (3+(s+1)|v|^2) \, \d v, 
\end{multline*}
where  
\begin{align*}
\mathscr{J}^\nu_{s,1}(h,g)&=2s\int_{\R^{3}\times \R^{3}}h(v)g(\vet)\,\frac{\Psi_\nu(|v-\vet|)}{|v-\vet|^2} \left(\langle v\rangle^{s-2}-\langle \vet\rangle^{s-2}\right)\left(|\vet|^{2}-(v \cdot \vet)\right)\d v\d\vet\,,\\
\mathscr{J}^\nu_{s,2}(h,g)&=s(s-2)\int_{\R^{3}\times\R^{3}}\langle v\rangle^{s-4}h(v)g(\vet)\frac{\Psi_\nu(|v-\vet|)}{|v-\vet|^2} \left(|v|^{2}\,|\vet|^{2}-(v\cdot \vet)^{2}\right)\d v\d \vet\,.
\end{align*}
As  in Lemma \ref{lem:mom} and Remark \ref{rmq:negaJs1} one has 
$$\mathscr{J}^\nu_{s,1}(f^\nu,f^\nu)=2s\int_{\R^{3}\times \R^{3}}f^\nu \fet^\nu \,\frac{\Psi_\nu(|v-\vet|)}{|v-\vet|^2} \langle v\rangle^{s-2} \left(\langle \vet \rangle^{2}-\langle v\rangle^{2}\right) \d v\d\vet \leq 0.$$
One now splits $\mathscr{J}^\nu_{s,2}(f^\nu,F^\nu)$ according to $|v-\vet|\ge 1$ and $|v-\vet|< 1$\,, 
$$\mathscr{J}^\nu_{s,2}(f^\nu,F^\nu) = I_1+I_2,$$
where 
\begin{eqnarray*}
I_1 & = &  s(s-2)  \int_{|v-\vet|\ge 1}  f^\nu \fet^\nu(1-{\dd}\fet^\nu) \: \frac{\Psi_\nu (|v-\vet|)}{|v-\vet|^2}\: \langle v\rangle^{s-4} (|v|^2|\vet|^2 -(v\cdot\vet)^2) \,\d v\, \d \vet, \\
I_2 & = &  s(s-2)  \int_{|v-\vet|< 1}  f^\nu \fet^\nu(1-{\dd}\fet^\nu) \: \frac{\Psi_\nu (|v-\vet|)}{|v-\vet|^2}\: \langle v\rangle^{s-4} (|v|^2|\vet|^2 -(v\cdot\vet)^2) \,\d v\, \d \vet\,.
\end{eqnarray*}
Since $|v|^2 |\vet|^2 -(v\cdot \vet)^2 \le \langle v \rangle^2 \langle \vet\rangle^2$,  $\Psi_\nu$ satisfies \eqref{majpsinu} and $F^\nu\le f^\nu$, one has 
\begin{eqnarray*}
I_1 & \leq &  2s(s-2)  \int_{|v-\vet|\ge1} |v-\vet|^\gamma f^\nu \fet^\nu \langle v \rangle^{s-2} \langle \vet \rangle^2 \d v \, \d \vet \\
& \leq &  2s(s-2)  \int_{|v-\vet|\ge1}f^\nu \fet^\nu \langle v \rangle^{s-2} \langle \vet \rangle^2 \d v \, \d \vet \leq 2  s(s-2)  \bm{m}_{s-2}^{\nu}(t) \bm{m}_2^{\nu}(t),
\end{eqnarray*}
where $\bm{m}_{s}^{\nu}(t) = \int_{\R^3} f^{\nu}(t,v) \langle v \rangle^s \d v$.
For $I_2$, we use  \eqref{majpsinu}, $F^\nu \le f^\nu$, and $|v|^2|\vet|^2 -(v\cdot \vet)^2 \le |v|\, |\vet|\, |v-\vet|^2\,,$
to get that
\begin{eqnarray*}
I_2 & \leq &  2s(s-2) \int_{|v-\vet|<1} |v-\vet|^{2+\gamma} f^\nu \fet^\nu \langle v \rangle^{s-4} |v| \, |\vet| \, \d v \, \d \vet \\
& \leq &  2s(s-2) \int_{|v-\vet|<1} f^\nu \fet^\nu \langle v \rangle^{s-4} |v| \, |\vet| \, \d v \, \d \vet \leq   2s(s-2) \bm{m}_{s-3}^{\nu}(t) \bm{m}_1^{\nu}(t).
\end{eqnarray*}
One also splits $\mathscr{J}^\nu_{s,1}(f^\nu,(f^\nu)^2)$ according to $|v-\vet|\ge 1$ and $|v-\vet|< 1$,
$$\dd \mathscr{J}^\nu_{s,1}(f^\nu,(f^\nu)^2) = J_1+J_2,$$
where 
\begin{eqnarray*}
J_1 & = &  2s \dd \int_{|v-\vet|\ge1}  f^\nu \left(\fet^\nu\right)^2 \: \frac{\Psi_\nu (|v-\vet|)}{|v-\vet|^2}\: (\langle v\rangle^{s-2}- \langle \vet\rangle^{s-2}) (\vet \cdot (\vet-v))\,\d v \,\d \vet ,\\
J_2 & = &  2s \dd \int_{|v-\vet|<1}  f^\nu \left(\fet^\nu\right)^2 \: \frac{\Psi_\nu (|v-\vet|)}{|v-\vet|^2}\: (\langle v\rangle^{s-2}- \langle \vet\rangle^{s-2}) (\vet \cdot (\vet-v))\,\d v \,\d \vet\,.
\end{eqnarray*}
Since $\Psi_\nu$ satisfies  \eqref{majpsinu} and $\dd f^\nu \leq 1$, one has 
\begin{eqnarray*}
J_1 & \leq & 4s   \int_{|v-\vet|\ge 1} |v-\vet|^\gamma f^\nu\fet^\nu \, {(\langle v\rangle^{s-2}+ \langle \vet\rangle^{s-2}) (\langle \vet\rangle^2  +\langle \vet \rangle \langle v  \rangle )}\,\d v \,\d \vet \\
& \leq & 4s\, ( \bm{m}^\nu_{s-2}(t)\bm{m}^\nu_2(t)+2\bm{m}^\nu_{s-1}(t)\bm{m}^\nu_1(t)+  {\bm{m}^\nu_0(t)}\bm{m}^\nu_s(t))\,.
\end{eqnarray*}
For $J_2$, we use that $\Psi_\nu(r)\le2$ for $r<1$ by \eqref{majpsinu} and $\dd f^\nu\le1$. We also have $|\vet\cdot (\vet-v)|\le |\vet| \, |v-\vet| , $
and 
{$$
\left|\langle v \rangle^{s-2} - \langle \vet \rangle^{s-2} \right|
\leq  (s-2) \,|v-\vet| \,\sup_{t\in(0,1)}  \langle tv+(1-t)\vet\rangle^{s-3} \,.
$$
Hence, 
{\renewcommand{\arraystretch}{1.3} $$\left|\langle v \rangle^{s-2} - \langle \vet \rangle^{s-2} \right| \le \left\{ 
\begin{array}{lcl}
 (s-2)\,|v-\vet|\, &\mbox{ if } &s\le3\,, \\
C(s-2)\,|v-\vet|\,  (\langle v\rangle^{s-3} +\langle\vet \rangle^{s-3} )
&\mbox{ if } & s>3\,,
\end{array}\right.$$}
for some $C$ depending only on $s$. Consequently, if $s\le3$, we obtain
$$J_2\le 4s(s-2)\bm{m}^\nu_0(t)\bm{m}^\nu_1(t),$$
whereas, if $s>3$, 
$$J_2\le 4C s(s-2)\big( \bm{m}^\nu_{s-3}(t)\bm{m}^\nu_1(t) + \bm{m}^\nu_0(t)\bm{m}^\nu_{s-2}(t) \big).$$}
Finally, 
$$ \nu s \int_{\R^3} f^\nu  \, \langle v\rangle^{s-4} (3+(1+s)|v|^2) \,\d v \le 
\nu s (s+3)\bm{m}^\nu_{s-2}(t).$$
Combining the above estimates and  {\eqref{mass_app}-\eqref{energy_app}}, we deduce the existence of some constant $C_s$ depending on $s$, $\gamma$, $T$ and {$\| f_{\mathrm{in}}\|_{L^1_2}$} such that $\frac{\d}{\d t }\bm{m}^\nu_s(t) \leq C_s \bm{m}^\nu_s(t), $
and  \eqref{L1s} follows. 
\end{proof}
{\begin{lem}\label{Lem:L2app}
Let  $f_{\mathrm{in}}\in \C^\infty (\R^3)\cap H^1(\R^3) \cap W^{3,\infty}(\R^3)$ satisfying 
\eqref{fd}. Let $f^\nu$ be a solution to (\ref{pbappr}) given by Proposition \ref{solapp}.  Then, for any  $T>0$ and $s\ge 2$, there exists some constant $C>0$ depending only on $s$, $\dd$, $T$ and $\| f_{\mathrm{in}}\|_{L^1_2}$ such that, for any $t\in(0,T)$,
{\begin{equation}\label{evol_L2}
\frac{\d }{\d t}\lM^\nu_s(t)+ \kappa \:\lD^\nu_{s-2}(t) \leq  C \,  \lM^\nu_s(t)\,,
\end{equation}}
with
$$ \lM^\nu_s(t) = \int_{\R^3} (f^\nu(t,v))^2 \langle v \rangle^s \d v, \qquad \qquad  \lD_s^\nu (t)=\int_{\R^3} |\grad (\langle v \rangle^{\frac{s}{2}} f^\nu(t,v) )|^2\d v .$$ 
\end{lem}
\begin{proof}
Let $s\geq 0$. We deduce from \eqref{pbappr} that 
\begin{multline}\label{reg}
\frac{1}{2}\frac{\d}{\d t} \int_{\R^3} (f^\nu)^2 \langle v \rangle^s \d v =  -\int_{\R^3} \left(\bm{\Sigma}^\nu[f^\nu] +\nu I_3 \right) \grad f^\nu \cdot \grad f^\nu 
\langle v \rangle^s  \d v\\
- s \int_{\R^3} \left(\bm{\Sigma}^\nu[f^\nu]  f^\nu \grad f^\nu \right) \cdot v \langle v \rangle^{s-2} \d v  
  + \int_{\R^3} \over{b}^\nu[f^\nu] \cdot\grad f^\nu  f^\nu (1-\dd f^\nu) \langle v \rangle^s  \d v\\
+ s \int_{\R^3} \over{b}^\nu[f^\nu] \cdot v \, (f^\nu)^2(1-\dd f^\nu) 
\langle v \rangle^{s-2} \d v
- s\,\nu  \int_{\R^3}  f^\nu \grad f^\nu \cdot v\, \langle v \rangle^{s-2} \d v\,.\end{multline}
It follows from the second point of Proposition \ref{diffusionapp} that 
$$ \int_{\R^3} \left(\bm{\Sigma}^\nu[f^\nu] +\nu I_3 \right)\grad f^\nu \grad f^\nu 
\langle v \rangle^s  \d v \geq \kappa \int_{\R^3} |\grad f^\nu|^2 \, 
\langle v \rangle^{s-2}\, \d v\,.$$ 
Proceeding as in the proof of \eqref{eq:dtMs}, we obtain
\begin{multline}\label{Mnu}
 \frac{1}{2}\frac{\d}{\d t}\lM^\nu_s(t) +  \frac{\kappa}{2} \lD^\nu_{s-2}(t) \leq   s \int_{\R^3}\langle v\rangle^{s-2} \left( (f^\nu)^2 -\frac{2\dd}3 (f^\nu)^3\right) \bm{b}^\nu[f^\nu] \cdot v \d v \\
 - \int_{\R^3} \langle v \rangle^s \left(\frac12 (f^\nu)^2 -\frac{\dd}3 (f^\nu)^3\right) \bm{c}^\nu[f^\nu] \d v  
 +  \frac{s}2 \int_{\R^3} \langle v\rangle^{s-4} (f^\nu)^2 \mathrm{Trace}\left(\bm{\Sigma}^\nu[f^\nu]\cdot \bm{A}(v)\right)\d v  \\
 -  \frac{\dd s}{2} \int_{\R^3} (f^\nu)^2 \bm{b}^\nu[(f^\nu)^2]\cdot v \langle v\rangle^{s-2} \d v+ \kappa \:\frac{(s-2)^2}{4}\int_{\R^3}  (f^\nu)^2 \langle v \rangle^{{s-4}} \d v   \\
 +   \frac{\nu s }{2} \int_{\R^3} (f^\nu)^2 \langle v \rangle^{s-4}(3+(s+1)|v|^2) \d v\,,
\end{multline}
where $\bm{A}(v)=\langle v\rangle^{2}\mathbf{Id}+(s-2)\,v\otimes v$, $v \in \R^{3}.$ For the last two integrals in \eqref{Mnu}, we clearly have 
$$\kappa \:\frac{(s-2)^2}{4}\int_{\R^3}  (f^\nu)^2 \langle v \rangle^{{s-4}} \d v \le \kappa \:\frac{(s-2^2}{4}\; \lM^\nu_{{s-4}}(t),$$
and
$$\frac{\nu s }{2} \int_{\R^3} (f^\nu)^2 \langle v \rangle^{s-4}(3+(s+1)|v|^2) \d v \le \frac{\nu s (s+3)}{2} \;\lM_{s-2}^{\nu}(t).$$
For the integral involving $\bm{A}$ in \eqref{Mnu}, we have by \eqref{majpsinu}, for every $i,j$, $|\bm{A}_{i,j}|\le s\langle v\rangle^2$ and $| \bm{\Sigma}^\nu_{i,j}[f^\nu]| \le 2\Psi^\nu *f^\nu \le 4 |\cdot|^{2+\gamma} *f^\nu\,. $
Hence, 
\begin{multline*}
\left|\frac{s}2 \int_{\R^3} \langle v\rangle^{s-4} (f^\nu)^2 \mathrm{Trace}\left(\bm{\Sigma}^\nu[f^\nu]\cdot \bm{A}(v)\right)\d v\right| \\
\le 18 s^2 \int_{\R^6} \langle v \rangle^{s-2} (f^\nu)^2 |v-\vet|^{2+\gamma} \fet^\nu \d v \d\vet \le 18 s^2\, \bm{m}^\nu_{2+\gamma}(t) \, \lM^\nu_{s+\gamma}(t).
\end{multline*}
For the first integral in \eqref{Mnu}, since $0\le \frac13 (f^\nu)^2\le (f^\nu)^2- \frac{2\dd}3 (f^\nu)^3 \le  (f^\nu)^2$, we have by \eqref{majpsinu}
\begin{eqnarray*}
& & \hspace{-1cm}\left| s \int_{\R^3}\langle v\rangle^{s-2} \left( (f^\nu)^2 -\frac{2\dd}3 (f^\nu)^3\right) \bm{b}^\nu[f^\nu] \cdot v \d v \right|
\leq  s  \int_{\R^3}\langle v\rangle^{s-2} (f^\nu)^2 \, | \bm{b}^\nu[f^\nu]|\, |v| \,\d v  \\
& & \hspace{4cm} \leq 2s \int_{\R^6} \frac{ \Psi_\nu(|v-\vet|)}{|v-\vet|} \:\fet^\nu \, \langle v\rangle^{s-1} \, (f^\nu)^2 \, \d v \d \vet \\
& & \hspace{4cm} \leq 4s \int_{\R^6} |v-\vet|^{1+\gamma} \:\fet^\nu \, \langle v\rangle^{s-1} \, (f^\nu)^2 \, \d v \d \vet\,.
\end{eqnarray*}  
Now, as {for \eqref{eq:bm2}-\eqref{eq:bm3}}, there exists some universal constant $C>0$ such that for every $v\in\R^3$ and every $t\in[0,T]$,
\begin{align}
  \left| \int_{\R^3} |v-\vet|^{1+\gamma} \:\fet^\nu(t,\vet)\d\vet \right| & \leq C (\|f_{\mathrm{in}}\|_{L^1_2} +\|{f^{\nu}(t)}\|_{L^2})\langle v\rangle^{\max\{0,1+\gamma\}} \nonumber\\
  & \leq  C (\|f_{\mathrm{in}}\|_{L^1_2} +\dd^{-\frac12} \|f_{\mathrm{in}}\|^{\frac12}_{L^1})\langle v\rangle^{\max\{0,1+\gamma\}}, \label{eq:bm1app}
\end{align}
where we used that {$f^\nu\le \dd^{-1}$}. Consequently, we get that
$$ \left| s \int_{\R^3}\langle v\rangle^{s-2} \left( (f^\nu)^2 -\frac{2\dd}3 (f^\nu)^3\right) \bm{b}^\nu[f^\nu] \cdot v \d v \right|
 \leq C_{\dd} \int_{\R^3} \langle v\rangle^{\max\{s-1,s+\gamma\}} \, (f^\nu)^2 \, \d v  \leq C_{\dd} \lM^\nu_{s}(t).$$
Similarly, since $\dd (f^\nu)^2 \le f^\nu$, one has 
$$\left|-\frac{\dd s}{2} \int_{\R^3} (f^\nu)^2 \bm{b}^\nu[(f^\nu)^2]\cdot v \langle v\rangle^{s-2} \d v\right|  \leq \frac{ s}{2} \int_{\R^3} (f^\nu)^2 |\bm{b}^\nu[f^\nu] | \langle v\rangle^{s-1} \d v \leq C_{\dd} \lM^\nu_{s}(t)\,.$$
For the second integral in \eqref{Mnu}, since $0\le \frac16 (f^\nu)^2 \le \frac12 (f^\nu)^2-\frac{\dd}{3} (f^\nu)^3 \le \frac12 (f^\nu)^2$, we have by \eqref{majpsinu}
\begin{eqnarray*}
& & \hspace{-1cm}\left| -\int_{\R^3} \langle v \rangle^s \left(\frac12 (f^\nu)^2 -\frac{\dd}3 (f^\nu)^3\right) \bm{c}^\nu[f^\nu] \d v \right|
 \leq  \frac12 \int_{\R^3} \langle v \rangle^s (f^\nu)^2 \, |\bm{c}^\nu[f^\nu]| \d v  \\
& & \hspace{2cm} \leq  \int_{\R^6}\left( \frac{ \Psi_\nu(|v-\vet|)}{|v-\vet|^2} + \frac{ |\Psi_\nu'(|v-\vet|)| }{|v-\vet|}\right) \: \fet^\nu \, \langle v \rangle^s \,(f^\nu)^2 \, \d v \d \vet\\
 & & \hspace{2cm} \leq  (2+C) \int_{\R^6} |v-\vet|^\gamma\, \fet^\nu \, \langle v \rangle^s \,(f^\nu)^2 \, \d v \d \vet \,.
\end{eqnarray*} 
{Now, for a given $v\in\R^3$, one has, thanks to the H\"older inequality,
\begin{align*}
  \int_{\R^3} |v-\vet|^\gamma\, \fet^\nu \, \d\vet & \leq \int_{|v-\vet|\ge 1} |v-\vet|^\gamma\, \fet^\nu \, \d\vet +\int_{|v-\vet|<1} |v-\vet|^\gamma\, \fet^\nu \, \d\vet \\
  &  \leq  \|f^\nu(t)\|_{L^{1}} +\|f^\nu(t)\|_{L^{p}}\,\left(\int_{|v-\vet|<1}|v-\vet|^{q\g}\d \vet\right)^{\frac{1}{q}} \\
  & \leq \bar{C} \left(\|f_{\mathrm{in}}\|_{L^{1}} + \dd^{-\frac{p-1}{p}} \|f_{\mathrm{in}}\|_{L^{1}}^{\frac1p} \right)
\end{align*}
for $p>1$ such that $-\gamma q<3$ where $\frac1{p}+\frac1{q}=1$. Hence,
$$\left| -\int_{\R^3} \langle v \rangle^s \left(\frac12 (f^\nu)^2 -\frac{\dd}3 (f^\nu)^3\right) \bm{c}^\nu[f^\nu] \d v \right|
 \leq \bar{C} \left(\|f_{\mathrm{in}}\|_{L^{1}} + \dd^{-\frac{p-1}{p}} \|f_{\mathrm{in}}\|_{L^{1}}^{\frac1p} \right) \lM^\nu_s(t).$$
 For $\nu\in(0,1)$ and $\gamma\in(-2,0)$, Lemma \ref{massandenergy} implies that all the above $L^1$-moments are bounded by some constant depending only on $T$ and {$\|f_{\mathrm{in}}\|_{L^1_2}$}.  Thus, gathering the above estimates, \eqref{evol_L2} follows.}\end{proof}

\begin{rmq}\label{rmq_L2app}
{Performing the same manipulations as above but using the first point of Proposition \ref{diffusionapp}, one obtains
$$ \frac{\d }{\d t}\lM^\nu_s(t)+ 2K_{0}  \int_{\R^3} |\grad f^\nu|^2 \, \langle v \rangle^{s+\g} \,\min \left\{(\nu^{-1}|v|)^{-\gamma}, 2^{-\gamma}, 2(\nu|v|)^{-(2+\gamma)}\right\} \, \d v \leq  C \, \lM^\nu_s(t)$$
where again $C$ depends on $s$, $\dd$, $T$ and $\| f_{\mathrm{in}}\|_{L^1_2}$.}
%where the notations are the same as in Lemma \ref{Lem:L2app}.}
\end{rmq}

\begin{proof}[Proof of Theorem \ref{existence}]
Let us fix $T>0$. Consider $f_{\mathrm{in}}\in L^1_{s_0}(\R^3)$ for some $s_0>2$ satisfying \eqref{hypci}--\eqref{eq:Mass} for some $\dd_0>0$. Then, there exists a sequence of functions $( f_{\mathrm{in},k})_{k\ge1}$ in $\C^\infty (\R^3)\cap H^1(\R^3) \cap W^{3,\infty}(\R^3)$ such that $( f_{\mathrm{in},k})_{k\ge1}$ converges towards $f$ in $L^1_{s_0}(\R^3)$ and 
$$\alpha'_k e^{-\beta'_k |v|^2}\le f_{\mathrm{in},k} \le \frac{\alpha_k e^{-\beta_k |v|^2}}{1+\dd \alpha_k e^{-\beta_k |v|^2}}, $$
for some positive constants $\alpha_k$, $\alpha'_k$, $\beta_k$ and $\beta'_k$. 

For every $k\in \N_*$, we set $\nu_k=\frac1k$ and $f_k=f^{\nu_k}$, where $f^{\nu_k}$ denotes a solution to \eqref{pbappr} with initial datum $ f_{\mathrm{in},k}$ given by Proposition \ref{solapp}. Since $(f_{\mathrm{in},k})_{k\ge1}$ is bounded in $L^1_{s_0}(\R^3)$, we deduce from Lemma \ref{2plusgama} that $(f_k)_{k\ge1}$ is bounded in $L^2((0,T);L^1_{s_0}(\R^3))$. We now apply  Lemma \ref{Lem:L2app} with $s=s_0> 2$. 
Since  $(f_{\mathrm{in},k})_{k\ge1}$ is bounded in $L^1_{s_0}(\R^3)\cap L^\infty(\R^3)$, it is bounded in $ L^2_{s_0}(\R^3)$ and we deduce that there exists some constant $C_{T,\dd}$ depending on $T$, $\dd$ and $\sup_{k\ge 1}\|f_{\mathrm{in},k}\|_{L^1_{s_0}}$ such that, for any $k\in\N_*$,
$$\sup_{t\in[0,T]} \|f_k(t)\|^2_{L^2_{s_0}} +\int_0^T \int_{\R^3}\langle v\rangle^{s_0-2} |\nabla f_k(t,v)|^2 \d v\, \d t \leq C_{T,\dd}.$$
Consequently, $(f_k)_{k\ge1}$ is  bounded in $L^2((0,T);H^1(\R^3))$. We then deduce from the weak formulation associated to \eqref{pbappr} that $(\partial_tf_k)_{k\ge1}$ is bounded in $L^1((0,T);(W^{2,\infty}(\R^3))')$ and thus, for $m\ge4$, in $L^1((0,T);(H^m(\R^3))')$. Now, for $m\ge4$, we have 
$$H^1(\R^3)\cap L^1_{s_0}(\R^3) \subset L^1(\R^3)\subset (H^m(\R^3))',$$
the embedding of $H^1(\R^3)\cap L^1_{s_0}(\R^3) $ in $L^1(\R^3)$ being compact. We may thus conclude from \cite[Corollary 4]{Simon} that $(f_k)_{k\ge1}$ is relatively compact in the space $L^2((0,T);L^1(\R^3))$. Therefore, there exists a function $f\in L^2((0,T);L^1(\R^3))$ and a subsequence of $(f_k)_{k\ge1}$ (not relabelled) such that $(f_k)_{k\ge 1}$ converges towards $f\in L^2((0,T);L^1(\R^3))$ and a.e. on $(0,T)\times \R^3$. For $\varphi \in \C_0^2(\R^3)$, it is easy to check that the sequence $(\int_{\R^3} f_k \varphi \, \d v)_{k\ge 1}$ is equicontinuous and bounded in $\C([0,T])$. The Arzel\`a-Ascoli Theorem thus ensures  that it is relatively compact in $\C([0,T])$. 

Finally, we obtain that  $(\int_{\R^3} f_k \varphi \, \d v)_{k\ge 1}$ converges towards  $\int_{\R^3} f \varphi \, \d v$ in $\C([0,T])$ and then that $(f_k)_{k\ge 1}$ converges towards $f$ in $\C_w([0,T];L^2(\R^3))$, where $\C_w([0,T];L^2(\R^3))$ denotes the space of weakly continuous functions in $L^2(\R^3)$. We easily check that $f$ preserves mass and energy and, passing to the limit $k\to \infty$ in the weak formulation, we obtain that $f$ satisfies \eqref{weakform}.
Moreover, we can deduce from \eqref{L1s} that $f\in L^\infty((0,T);L^1_{s_0}(\R^3))$, and from Remark \ref{rmq_L2app} that 
$\nabla f \in L^{2}((0,T);L^{2}_{s_0+\gamma}(\R^{3})).$

Let us now prove the monotonicity of the entropy. We  know that $f\in L^2((0,T);H^1_{2+\gamma}(\R^3))$. Then, we deduce as in \eqref{eq:bm1app} that there exists some constant $C_{\dd}$ depending on $\dd$ and $\|f_{\mathrm{in}}\|_{L^1_2}$ such that, for every $v\in\R^3$ and $t\in[0,T]$
$$\big|\bm{b}[f](t,v) \big|\le C_{\dd}\, \langle v \rangle^{\max\{0,1+\gamma\}}. $$
One also has, for every $v\in\R^3$ and every $t\in[0,T]$, that $|\bm{\Sigma}[f](t,v)| \leq C\|f_{\mathrm{in}}\|_{L^1_2} \,\langle v \rangle^{2+\gamma},$
for some universal constant $C>0$. It thus follows from the weak formulation associated to \eqref{LFD} that $\partial_t f\in L^2((0,T);(H^1_{2+\gamma}(\R^3))')$. We then deduce from \cite[Ch.III Lemma 1.2]{Temam} that $f\in\C([0,T];L^2(\R^3))$. As in \cite[Lemma 4.18]{bag}, one may then prove the monotonicity of $\mathcal{S}_{\dd}(f)$. 
\end{proof}


\begin{thebibliography}{99}

\bibitem{ALL}
\textsc{R. Alexandre, J. Liao, \& C-J. Lin,} Some a priori estimates for the homogeneous Landau equation with soft potentials,
\textit{Kinet. Relat. Models} {\bf 8} (2015), 617--650.

\bibitem{amuxy}
\textsc{R. Alexandre, Y. Morimoto, S. Ukai, C.-J. Xu, \& T. Yang,} Smoothing effect of weak solutions for the spatially homogeneous Boltzmann equation without angular cutoff, \textit{Kyoto J. Math.} {\bf 52} (2012), 433--463.

\bibitem{ricardo}
\textsc{R. Alonso}, Emergence of exponentially weighted $L^{p}$-norms and Sobolev regularity for the Boltzmann equation,
\textit{Comm. Partial Differential Equations} {\bf 44} (2019), 416--446.

\bibitem{fisher}
\textsc{R. Alonso, V. Bagland \& B. Lods,} Uniform estimates on the Fisher information for solutions to Boltzmann and Landau equations, \emph{Kinet. Relat. Models} {\bf 12} (2019), 1163--1183.

\bibitem{ABL}
\textsc{R. Alonso, V. Bagland, \& B. Lods}, Long time dynamics for the Landau-Fermi-Dirac equation with hard potentials, \textit{J. Differential Equations} {\bf 270} (2021), 596--663.

\bibitem{ABDL-entro}
\textsc{R. Alonso, V. Bagland, L. Desvillettes, \& B. Lods}, About the use of Entropy dissipation for the Landau-Fermi-Dirac equation, \textit{J. Stat. Phys} {\bf 183} (2021), article 10.

\bibitem{AMSY}
\textsc{R. Alonso, Y. Morimoto, W. Sun, T. Yang} De Giorgi argument for weighted $L^2\cap L^{\infty}$ solutions to the non-cutoff Boltzmann equation, arXiv:2010.10065 

\bibitem{bag}
\textsc{V. Bagland}, Well-posedness for the spatially homogeneous Landau-Fermi-Dirac equation for hard potentials, \textit{Proc. Roy. Soc. Edinburgh Sect. A.} {\bf 66 } (2004), 415--447.

\bibitem{beckner}
\textsc{W. Beckner}, Pitt's inequality with sharp convolution estimates, \textit{Proc. Amer. Math. Soc.} {\bf 136} (2008), 1871--1885.

\bibitem{caff}
\textsc{L. Caffarelli, C. H. Chan \& A. F. Vasseur,} Regularity theory for parabolic nonlinear integral operators,
\textit{J. Amer. Math. Soc.} {\bf 24} (2011), 849--869.

\bibitem{cameron}
\textsc{S. Cameron, L. Silvestre, \& S. Snelson,} Global a priori estimates for the inhomogeneous Landau equation with moderately soft potentials, 
\textit{Ann. Inst. H. Poincar\'e Anal. Non Lin\'eaire} {\bf 35} (2018), 625--642.

\bibitem{CEL}
\textsc{J. Ca\~{n}izo, A. Einav, \& B. Lods}, On the rate of convergence to equilibrium for the linear Boltzmann equation with soft potentials, \textit{J. Math. Anal. Appl.} {\bf 462} (2018) 801--839.

\bibitem{kleb}
\textsc{K. Carrapatoso,} On the rate of convergence to equilibrium for the homogeneous Landau equation with soft potentials,
\textit{J. Math. Pures Appl.} {\bf 104} (2015), 276--310.

\bibitem{CDH}
\textsc{K. Carrapatoso, L. Desvillettes, \& L. He}, Estimates for the large time behavior of the Landau equation in the Coulomb case,
\textit{Arch. Ration. Mech. Anal.} {\bf 224} (2017), 381--420.

\bibitem{chapman}
\textsc{S. Chapman \& T. G. Cowling,} \textbf{The mathematical theory of non-uniform gases}, Cambridge University Press, 1970.
 


\bibitem{DesvJFA}
\textsc{L. Desvillettes}, Entropy dissipation estimates for the Landau equation in the Coulomb case and applications. {\it 
J. Funct. Anal.} {\bf 269} (2015), 1359--1403.

\bibitem{Desv}
\textsc{L. Desvillettes,} Entropy dissipation estimates for the Landau equation: general cross sections. {\it From particle systems to partial differential equations. III,} 121--143, Springer Proc. Math. Stat., 162, Springer, 2016. 

\bibitem{desvmou}
\textsc{L. Desvillettes, \& C. Mouhot,} Large time behavior of the a priori bounds for the solutions to the spatially homogeneous Boltzmann equations with soft potentials, \textit{Asymptot. Anal.} {\bf 54} (2007),  235--245.


\bibitem{DeVi2} 
\textsc{L. Desvillettes \& C. Villani}, On the spatially 
homogeneous  Landau equation for hard potentials. Part II : H theorem and applications. {\it Comm. Partial Differential Equations}, 
{\bf 25} (2000), 261--298.


\bibitem{golse}
\textsc{F. Golse, C. Imbert, C. Mouhot, \& A. F. Vasseur,} Harnack inequality for kinetic Fokker-Planck equations with rough coefficients and application to the Landau equation, \textit{Ann. Sc. Norm. Super. Pisa Cl. Sci. (5)} {\bf 19} (2019), 253--295.

\bibitem{GG}
\textsc{M.-P. Gualdani, \& N. Guillen,} On $A_{p}$ weights and the Landau equation,
\textit{Calc. Var. Partial Differential Equations} {\bf 58} (2019), paper 17, 55pp.

\bibitem{Lad}
\textsc{O. A. Ladyzenskaja, V. V.Solonnikov, \& N. N. Ural'ceva}, \textbf{Linear and quasilinear equations of parabolic type}, Providence, RI. American Mathematical Society, 1968.

\bibitem{Lu}
\textsc{X. Lu}, On spatially homogeneous solutions of a modified Boltzmann equation for
Fermi-Dirac particles. \textit{J. Stat. Phys.} {\bf 105} (2001), 353--388.

\bibitem{LW}
\textsc{X. Lu, \& B. Wennberg}, On stability and strong convergence for the spatially homogeneous Boltzmann equation for Fermi-Dirac particles, \textit{Arch. Ration. Mech. Anal.} {\bf 168} (2003), 1--34. 

\bibitem{Simon}
\textsc{J. Simon}, Compact sets in the space $L^p(0,T;B)$. \textit{Ann. Mat. Pura Appl.} {\bf 146} (1987) 65--96.

\bibitem{Temam}
\textsc{R. Temam}, \textbf{Navier-Stokes Equations. Theory and numerical analysis,} North Holland, 1977.

\bibitem{vasseur}
\textsc{A. F. Vasseur, } The De Giorgi method for elliptic and parabolic equations and some applications, \textit{Lectures on the analysis of nonlinear partial differential equations, Part 4, 195--222,}
Morningside Lect. Math., 4, Int. Press, Somerville, MA, 2016.
 
\bibitem{Wu}
\textsc{K.-C. Wu,} Global in time estimates for the spatially homogeneous Landau equation with soft potentials, 
\textit{J. Funct. Anal.} {\bf 266} (2014), 3134--3155.


\end{thebibliography}
\end{document}